\documentclass[reqno,11pt]{amsart}
\usepackage{amsfonts}
\usepackage{a4wide}
\usepackage{color}
\usepackage{mathrsfs}
\usepackage{mathtools}
\usepackage{amsmath}
\usepackage{amssymb}
\usepackage{bbm}
\usepackage{esint}
\usepackage{nicefrac}
\numberwithin{equation}{section}
\usepackage[colorlinks,citecolor=green,linkcolor=red]{hyperref}

\usepackage[latin1]{inputenc}
\input xy
\xyoption{all}

\newtheorem{theorem}{Theorem}[section]
\newtheorem{Ca}[theorem]{Corollary}
\newtheorem{Th}[theorem]{Theorem}
\newtheorem{Lm}[theorem]{Lemma}
\newtheorem{Prop}[theorem]{Proposition}
\newtheorem{Def}[theorem]{Definition}
\newtheorem{Example}[theorem]{Example}
\newtheorem{Remark}[theorem]{Remark}

\newtheorem{ThA}{Theorem}[section]

\title{Almost sharp descriptions of traces of Sobolev $W_{p}^{1}(\mathbb{R}^{n})$-spaces to arbitrary compact subsets of $\mathbb{R}^{n}$. The case $p \in (1,n]$}
\author{Alexander I. Tyulenev}
\address{Steklov Mathematical Institute of Russian academy of Sciences}
\email{tyulenev-math@yandex.ru,tyulenev@mi-ras.ru}
\begin{document}
\date{\today}
\allowdisplaybreaks
\begin{abstract}
Let $S \subset \mathbb{R}^{n}$ be an arbitrary nonempty compact set such that the $d$-Hausdorff content $\mathcal{H}^{d}_{\infty}(S) > 0$ for some $d \in (0,n]$.
For each $p \in (\max\{1,n-d\},n]$, an almost sharp intrinsic description
of the trace space $W_{p}^{1}(\mathbb{R}^{n})|_{S}$ of the Sobolev space $W_{p}^{1}(\mathbb{R}^{n})$ to the set $S$ is obtained. Furthermore, for
each $p \in (\max\{1,n-d\},n]$ and $\varepsilon \in (0, \min\{p-(n-d),p-1\})$,
new bounded linear extension operators from the trace space $W_{p}^{1}(\mathbb{R}^{n})|_{S}$ into the space
$W_{p-\varepsilon}^{1}(\mathbb{R}^{n})$ are constructed.
\end{abstract}
\maketitle
\tableofcontents
%

\markright{\uppercase{Almost sharp descriptions of traces}}

\section{Introduction}

The trace problem, i.e., the problem of the sharp intrinsic description of traces of the first-order Sobolev space $W_{p}^{1}(\mathbb{R}^{n})$,
$p \in [1,\infty]$, to different subsets $S \subset \mathbb{R}^{n}$ is a classical long-standing problem in the function space theory. There is an
extensive literature devoted to the subject. However, without any additional regularity assumptions on $S$ the
problem becomes extremely complicated and remains open in the case $p \in [1,n]$.
The purpose of the present paper is to pose correctly and solve a weakened version of this trace problem.
Namely, \textit{we obtain almost sharp descriptions of the traces to compact sets} $S \subset \mathbb{R}^{n}$ of functions in the first-order Sobolev spaces $W_{p}^{1}(\mathbb{R}^{n})$  \textit{in the case} $p \in (1,n]$ \textit{without any additional regularity assumptions on $S$}. The case $p=1$ is special and will not be considered in this paper.
To precisely pose the above problems we recall some terminology concerning Sobolev spaces.

As usual, for each $p \in [1,\infty]$, we let $W_{p}^{1}(\mathbb{R}^{n})$ denote the corresponding Sobolev
space of all equivalence classes of real valued functions $F \in L_{p}(\mathbb{R}^{n})$ whose distributional partial
derivatives $D^{\gamma}F$ on $\mathbb{R}^{n}$ of order $|\gamma| = 1$ belong to $L_{p}(\mathbb{R}^{n})$.
This space is normed by
\begin{equation}
\notag
\|F|W_{p}^{1}(\mathbb{R}^{n})\|:=\|F|L_{p}(\mathbb{R}^{n})\|+\sum\limits_{|\gamma| = 1}\|D^{\gamma}F|L_{p}(\mathbb{R}^{n})\|.
\end{equation}
We assume that the reader is familiar with the (Bessel) $C_{1,p}$-capacities (see e.g., Section 2.2 in \cite{A}), the $d$-Hausdorff measures $\mathcal{H}^{d}$ and the
$d$-Hausdorff contents $\mathcal{H}^{d}_{\infty}$ (see e.g., \cite{A}, Section 5.1).
Recall (see e.g., \cite{A}, Section 6.2) that given $p \in (1,n]$, for every element $F \in W_{p}^{1}(\mathbb{R}^{n})$ there is a representative $\overline{F}$  of $F$ such that $\overline{F}$ has Lebesgue points $C_{1,p}$-quasi everywhere, i.e., everywhere on $\mathbb{R}^{n}$ except a set $E_{F}$ with $C_{1,p}(E_{F})=0$. Furthermore, according to the Sobolev imbedding theorem (see e.g., Theorem 1.2.4 in \cite{A}), if $p > n$, then for every $F \in W_{p}^{1}(\mathbb{R}^{n})$ there is a
unique representative $\overline{F}$ of $F$ which is locally $(1-\frac{n}{p})$-H\"older continuous.
In the sequel, given a parameter $p \in (1,\infty)$, for each element $F \in W_{p}^{1}(\mathbb{R}^{n})$ we will call $\overline{F}$ \textit{an $(1,p)$-good representative
of} $F$.

Clearly, if $p \in (1,n]$, then, for each element $F\in W_{p}^{1}(\mathbb{R}^{n})$, there are infinitely many $(1,p)$-good representatives $\overline{F}$ of $F$.
However, any two $(1,p)$-good representatives $\overline{F}_{1}$, $\overline{F}_{2}$ of $F$ coincide everywhere except a set of $p$-capacity zero.
As a result, given $p \in (1,n]$ and a set $S \subset \mathbb{R}^{n}$ with $C_{1,p}(S) > 0$, we define \textit{the $p$-sharp trace} $F|_{S}$
of each element $F \in W_{p}^{1}(\mathbb{R}^{n})$ as the equivalence class (modulo coincidence everywhere on $S$ except a set of $p$-capacity zero) of the pointwise restriction
of any $(1,p)$-good representative $\overline{F}$ of $F$ to the set $S$.
Since $C_{1,p}(S) > 0$, \textit{the $p$-sharp trace} $F|_{S}$ of $F$ is well defined in this case.
If $p > n$ and $S$ is an arbitrary nonempty set in $\mathbb{R}^{n}$ we define \textit{the $p$-sharp trace} $F|_{S}$
of each element $F \in W_{p}^{1}(\mathbb{R}^{n})$ as the pointwise restriction of a unique continuous representative $\overline{F}$ of $F$ to the set $S$.
Given $p \in (1,\infty)$ and a set $S \subset \mathbb{R}^{n}$, we define \textit{the $p$-sharp trace space} by letting
\begin{equation}
\notag
W_{p}^{1}(\mathbb{R}^{n})|_{S}:=\{F|_{S}: F \in W_{p}^{1}(\mathbb{R}^{n})\}
\end{equation}
and equip this space with the usual quotient-space norm. By $\operatorname{Tr}|_{S}$ we denote the corresponding \textit{$p$-sharp trace operator}
which takes $F \in W_{p}^{1}(\mathbb{R}^{n})$ and returns $F|_{S}$.

Now the problem of
\textit{the sharp intrinsic description} of traces of first-order Sobolev $W_{p}^{1}(\mathbb{R}^{n})$-spaces can be formulated as the following three
intimately related questions.

\textbf{Problem A}. \textit{Let $p \in (1,\infty]$ and let $S\subset \mathbb{R}^{n}$ be a closed nonempty set
with $C_{1,p}(S) > 0$.}

{\rm (\textbf{Q}1)} \textit{Given a Borel function $f: S \to \mathbb{R}$,
find necessary and sufficient conditions for the existence of a Sobolev extension $F$ of $f$,
i.e., $F \in W_{p}^{1}(\mathbb{R}^{n})$ and $F|_{S}=f$.}

{\rm (\textbf{Q}2)} \textit{Using only geometry of the set $S$ and values of the function $f$, compute the trace norm $\|f|W_{p}^{1}(\mathbb{R}^{n})|_{S}\|$
up to some universal constants.}

{\rm (\textbf{Q}3)} \textit{Does there exist a bounded linear operator $\operatorname{Ext}_{S,p}: W_{p}^{1}(\mathbb{R}^{n})|_{S} \to W_{p}^{1}(\mathbb{R}^{n})$
such that $\operatorname{Tr}|_{S} \circ \operatorname{Ext}_{S,p} = \operatorname{Id}$ on $W_{p}^{1}(\mathbb{R}^{n})|_{S}$?}

Typically, if one can answer some of the above questions, then one has a key
to the other questions.
Informally speaking, the essence of questions (\textbf{Q1})--(\textbf{Q3})
can be formulated as follows: find an equivalent intrinsically defined
norm in the sharp trace space $W_{p}^{1}(\mathbb{R}^{n})|_{S}$ and find a linear procedure
of an almost optimal extension. As we have already mentioned, if $S$ is not assumed to have any additional regularity properties,
Problem A is very difficult and still unsolved in the range $p \in (1,n]$. Below we present a
brief historical overview of the results related to Problem A.

{\rm (\textbf{H}0)} \textit{In the case $p=\infty$}, the Sobolev space $W^{1}_{\infty}(\mathbb{R}^{n})$ can be identified with the space
$\operatorname{LIP}(\mathbb{R}^{n})$  of Lipschitz functions on $\mathbb{R}^{n}$. Moreover, it is known (see the McShane-Whitney extension lemma in Section 4.1 of \cite{Kos}) that for
any closed set $S \subset \mathbb{R}^{n}$ the restriction $\operatorname{LIP}(\mathbb{R}^{n})|_{S}$
coincides with the space $\operatorname{LIP}(S)$ of Lipschitz functions on $S$ and that, furthermore, the
classical Whitney extension operator linearly and continuously maps the space $\operatorname{LIP}(S)$
into the space  $\operatorname{LIP}(\mathbb{R}^{n})$  (see e.g., \cite{St}, Chapter 6).

{\rm (\textbf{H}1)} In the case $p \in (1,\infty)$ the pioneering investigations go back to Gagliardo \cite{Gal}, where for each $p \in (1,\infty)$,
the trace problem was solved when $S$ is a graph of a Lipschitz function.
Note that this work extended the earlier results by Aronszajn \cite{Ar} and Slobodetskii and Babich \cite{Ba} concerning the
case $p=2$. It should be mentioned that the case $S=\mathbb{R}^{d}$ with $d \in [1,n-1] \cap \mathbb{N}$
was covered by Besov in the fundamental paper \cite{Besov}. Furthermore, the trace problems for higher-order Sobolev spaces
were firstly considered in \cite{Besov}.

{\rm (\textbf{H}2)} Recall \cite{JW} (see Chapter 2 therein) that, given a parameter $d \in (0,n]$, \textit{a closed set} $S \subset \mathbb{R}^{n}$ is said to be  \textit{$d$-set} if there are constants $c_{S,1}, c_{S,2} > 0$ such that
\begin{equation}
\label{eq14}
c_{S,1}l^{d} \le \mathcal{H}^{d}(Q_{l}(x) \cap S) \le c_{S,2}l^{d} \quad \hbox{for all} \quad x \in S \quad \hbox{and all} \quad l \in (0,1],
\end{equation}
where $Q_{l}(x):=\prod_{i=1}^{n}[x_{i}-\frac{l}{2},x_{i}+\frac{l}{2}]$ is a closed cube centered in $x=(x_{1},...,x_{n})$ with side length $l$.
In the literature $d$-sets are also known as \textit{Ahlfors--David $d$-regular sets} (see, e.g., \cite{Ihn}); condition \eqref{eq14} is often called \textit{the Ahlfors--David $d$-regularity condition}.
The problems of characterization of
traces of the Besov spaces $B^{s}_{p,q}(\mathbb{R}^{n})$ \cite{Ihn,JW,Shv1}, the Bessel potential spaces $L^{s}_{p}(\mathbb{R}^{n})$ \cite{JW}, and the Lizorkin--Triebel spaces $F^{s}_{p,q}(\mathbb{R}^{n})$ \cite{Ihn,Shv1} on $d$-sets were considered. The detailed analysis of these results is beyond the scope of our
paper. Recall that $L^{1}_{p}(\mathbb{R}^{n})=F^{1}_{p,2}(\mathbb{R}^{n})=W_{p}^{1}(\mathbb{R}^{n})$ for any $p \in (1,\infty)$. Hence, given
a closed $d$-set $S \subset \mathbb{R}^{n}$ with $d \in (0,n]$ and $p \in (\max\{1,n-d\},\infty)$,
the results obtained in \cite{Ihn,JW,Shv1} give sharp descriptions of the traces to the set $S$ of functions $F \in W_{p}^{1}(\mathbb{R}^{n})$.

{\rm (\textbf{H}3)} \textit{In the case} $p \in (n,\infty)$, Shvartsman \cite{Shv2} completely solved Problem A.
More precisely, for each $p \in (n,\infty)$ he
found several equivalent intrinsic descriptions of the sharp trace space $W_{p}^{1}(\mathbb{R}^{n})|_{S}$
to arbitrary closed sets $S \subset \mathbb{R}^{n}$.
It is interesting to note that in that case the classical Whitney extension
operator gives an almost optimal Sobolev extension. Furthermore, the
criteria presented in Theorems 1.2 and 1.4 in \cite{Shv2} do not use (explicitly)
any geometrical properties of  $S$.

{\rm (\textbf{H}4)} Recall \cite{Ry} that given a parameter $d \in [0,n]$, a set $S \subset \mathbb{R}^{n}$ is said to be \textit{$d$-thick}
if there is a constant $c_{S,3} > 0$ such that
\begin{equation}
\notag
c_{S,3}l^{d} \le \mathcal{H}^{d}_{\infty}(Q_{l}(x) \cap S) \quad  \hbox{for all} \quad x \in S \quad \hbox{and all} \quad l \in (0,1].
\end{equation}
Recently some interesting geometric properties of $d$-thick sets were studied in the papers \cite{Az1}, \cite{Az2}, where
they were called \textit{$d$-lower content regular sets.}
It should be noted that the class of all $d$-sets \textit{is strictly contained in the class}
of all $d$-thick sets, but the latter \textit{is much wider.} One can find several interesting examples \cite{TV}
demonstrating the huge difference between the concepts of $d$-sets and $d$-thick sets, respectively.
For example, one can show \cite{TV} that every path-connected set in $\mathbb{R}^{n}$ containing  more than one point is $1$-thick.
Recently \cite{TV}, given a number $d \in [0,n]$ and a closed $d$-thick  set $S \subset \mathbb{R}^{n}$, Problem A
was solved for each $p \in (\max\{1,n-d\},\infty)$. Furthermore, a new linear extension operator
was constructed. Very recently \cite{T} the criterion obtained in \cite{TV}
was essentially simplified and clarified for the case when $S=\Gamma \subset \mathbb{R}^{2}$ is a planar
rectifiable curve of positive length and without self-intersections.
It should be noted that V. Rychkov considered in \cite{Ry} trace problems for
the Besov $B^{s}_{p,q}(\mathbb{R}^{n})$-spaces and the Lizorkin--Triebel $F^{s}_{p,q}(\mathbb{R}^{n})$-spaces on $d$-thick sets. However, some extra restrictions
on parameters $n,d,s,p,q$ were imposed. In particular, for $s \in \mathbb{N}$
it was assumed additionally that $d \in (n-1,n]$.
Hence, even the results described in \cite{T} do not fall into the scope of \cite{Ry}.
Note that using the same technics as in \cite{TV} one can obtain solutions to
similar problems for the case of weighted Sobolev spaces with Muckenhoupt weights \cite{TV2}.

{\rm (\textbf{H}5)} We should mention recent investigations concerned with problems of exact descriptions of traces
of Sobolev $W_{p}^{1}(\operatorname{X})$-spaces to closed subsets $S$ of
general metric measure spaces $\operatorname{X}$ \cite{BBG,LLW,Saks}. However,
some extra regularity constraints on $S$ were imposed.

As a result, Problem A was completely solved in the case $p \in (n,\infty]$ only.
The case $p \in (1,n]$ is much more difficult, the elegant machinery developed in \cite{Shv2}
does not work. Indeed, in this case the geometry of a given closed set $S \subset \mathbb{R}^{n}$ plays a crucial role
and has an influence not only on the
corresponding trace criterion but also to the constructions of the corresponding extension operators. In particular,
the classical
extension method of H. Whitney does not work. As far as we know, in the case $p \in (1,n]$
the most general results available so far were obtained in \cite{TV}.

In the present paper
we make a next relatively big step towards the solution of Problem A by
solving a weakened version of this problem.
First of all, we introduce a little
bit more rough definition of the trace of a given Sobolev function.
For this purpose, we recall that if $p \in (1,n]$ and $d \in (n-p,n]$ then for any given set $S \subset \mathbb{R}^{n}$ the condition $C_{1,p}(S)=0$
implies $\mathcal{H}^{d}(S)=0$. On the other hand, given $p \in (1,n]$ and $S \subset \mathbb{R}^{n}$, the condition $\mathcal{H}^{n-p}(S) < +\infty$ implies $C_{1,p}(S)=0$.
As a result, if $p \in (1,n]$, $d \in (n-p,n]$ and $S \subset \mathbb{R}^{n}$ is an arbitrary set
with $\mathcal{H}^{d}_{\infty}(S) > 0$ we define \textit{the $d$-trace of any element} $F \in W_{p}^{1}(\mathbb{R}^{n})$ to the set $S$
as \textit{the class of all Borel functions $f:S \to \mathbb{R}$} that coincide $\mathcal{H}^{d}$-a.e. on $S$ with \textit{the pointwise restriction to $S$
of any $(1,p)$-good representative} $\overline{F}$ of $F$; we denote it by $F|^{d}_{S}$. By $W_{p}^{1}(\mathbb{R}^{n})|^{d}_{S}$ we denote
the corresponding \textit{$d$-trace space}, i.e., the linear space consisting of $d$-traces $F|^{d}_{S}$
of all elements $F \in W_{p}^{1}(\mathbb{R}^{n})$ equipped  with the usual quotient-space
norm. Finally, $\operatorname{Tr}|^{d}_{S}: W_{p}^{1}(\mathbb{R}^{n}) \to W_{p}^{1}(\mathbb{R}^{n})|^{d}_{S}$ denotes the corresponding \textit{$d$-trace operator.}

In the present paper we obtain a solution to the following problem of \textit{an almost sharp intrinsic description} of traces of $W_{p}^{1}(\mathbb{R}^{n})$-spaces
\textit{to arbitrary compact sets.} In analogy with Problem A, we formulate it as the following three closely related questions.


\textbf{Problem B}.
\textit{Let $p \in (1,n]$, $d^{\ast} \in (n-p,n]$ and $\varepsilon^{\ast}:=\min\{p-(n-d^{\ast}),p-1\}$. Let
$S\subset \mathbb{R}^{n}$ be an arbitrary compact set with $\mathcal{H}^{d^{\ast}}_{\infty}(S) > 0$.}

{\rm (\textbf{q}1)} \textit{Given $\varepsilon \in (0,\varepsilon^{\ast})$, find a Banach space $\operatorname{X}_{\varepsilon}(S)=\operatorname{X}(S,p,d^{\ast},\varepsilon)$
of Borel functions $f: S \to \mathbb{R}$ equipped with an intrinsically defined norm such that
\begin{equation}
\label{eq11}
W_{p}^{1}(\mathbb{R}^{n})|^{d^{\ast}}_{S} \subset \operatorname{X}_{\varepsilon}(S) \subset W_{p-\varepsilon}^{1}(\mathbb{R}^{n})|^{d^{\ast}}_{S}
\end{equation}
and for some constant $C=C(p,n,d^{\ast},\varepsilon) > 0$, the following inequalities
\begin{equation}
\label{eq1.1''}
\frac{1}{C}\|f|W_{p-\varepsilon}^{1}(\mathbb{R}^{n})|^{d^{\ast}}_{S}\| \le \|f|\operatorname{X}_{\varepsilon}(S)\| \le C \|f|W_{p}^{1}(\mathbb{R}^{n})|^{d^{\ast}}_{S}\|
\end{equation}
hold for all $f \in W_{p}^{1}(\mathbb{R}^{n})|^{d^{\ast}}_{S}$.}

{\rm (\textbf{q}2)} \textit{Given $\varepsilon \in (0,\varepsilon^{\ast})$, does there exist a bounded linear operator
\begin{equation}
\label{eq1.2}
\operatorname{Ext}=\operatorname{Ext}(S,d^{\ast},\varepsilon):\operatorname{X}_{\varepsilon}(S) \to W_{p-\varepsilon}^{1}(\mathbb{R}^{n})
\end{equation}
such that
$\operatorname{Tr}|^{d^{\ast}}_{S} \circ \operatorname{Ext} = \operatorname{Id}$ on $\operatorname{X}_{\varepsilon}(S)$?}

We should make several \textit{important remarks} concerning the statement of the problem.

{\rm (\textbf{R}0)}
Using methods introduced in this paper one could attack an analog of Problem B posed
for arbitrary unbounded closed sets $S \subset \mathbb{R}^{n}$.
It would be ideologically similar but much more
technical;

{\rm (\textbf{R}1)} By \textit{an intrinsically defined norm} in the space $\operatorname{X}_{\varepsilon}(S)$ we mean
a norm whose expression contains a computationally explicit procedure \textit{exploiting only}
values of a given function $f:S \to \mathbb{R}^{n}$ and geometric properties of the set $S$;

{\rm (\textbf{R}2)} Since the set $S$ is compact it is easy to show that $W_{p}^{1}(\mathbb{R}^{n})|^{d^{\ast}}_{S}
\subset W_{p-\varepsilon}^{1}(\mathbb{R}^{n})|^{d^{\ast}}_{S}$. Hence, (q1) makes sense;

{\rm (\textbf{R}3)} Since $0 < \varepsilon < \varepsilon^{\ast} \le p$ we get $p-\varepsilon > n-d^{\ast}$. Hence, the operator $\operatorname{Tr}|^{d^{\ast}}_{S}$ is well defined
on the space $W^{1}_{p-\varepsilon}(\mathbb{R}^{n})$ and
the composition $\operatorname{Tr}|^{d^{\ast}}_{S} \circ \operatorname{Ext}(S,d^{\ast},\varepsilon)$ makes sense;

{\rm (\textbf{R}4)} By \eqref{eq11} and \eqref{eq1.1''} the trace space $W_{p}^{1}(\mathbb{R}^{n})|^{d^{\ast}}_{S}$
is continuously imbedded in the space $\operatorname{X}_{\varepsilon}(S)$. Hence, any bounded linear
operator $\operatorname{Ext}(S,d^{\ast},\varepsilon):\operatorname{X} _{\varepsilon}(S)\to W_{p-\varepsilon}^{1}(\mathbb{R}^{n})$
maps $W_{p}^{1}(\mathbb{R}^{n})|^{d^{\ast}}_{S}$ to $W_{p-\varepsilon}^{1}(\mathbb{R}^{n})$ linearly and continuously.

Note that the term ``almost sharp'' in the title of our paper can be informally justified as follows. There exists an arbitrary small $\varepsilon$-gap
between the $d^{\ast}$-trace space $W_{p}^{1}(\mathbb{R}^{n})|^{d^{\ast}}_{S}$ and the space
$\operatorname{X}_{\varepsilon}(S)$.
If we could formally put $\varepsilon = 0$ in Problem B, then we would
obtain in fact Problem A up to a slightly rougher definition of the trace.

To the best of our knowledge, the results of the present paper are the first to have been obtained
for the range $p \in (1,n]$ in such a high generality.
In \cite{HM} a similar problem was considered under the additional assumption
that the set $S \subset \mathbb{R}^{n}$ is Ahlfors--David $d^{\ast}$-regular.

In this paper we introduce several methods and techniques which
were never used before. Despite the fact that our machinery
does not allow to solve Problem A, we strongly believe that
our new ideas and tools will provide a fundament for further investigations. Moreover, they could
be useful in solving similar extension problems in the context of other spaces of smooth functions.

\textbf{Structure of the paper.} The paper is organized as follows.

\textit{Section 2} contains the necessary preliminaries concerning Hasudorff measures, Sobolev spaces, and
Frostman-type measures.

\textit{In Section 3} we recall basic results of our recent paper \cite{T2}. In particular, for any given set $S$ with $\mathcal{H}^{d}_{\infty}(S) > 0$,
those results allow one to built a specially ordered sequence of families of dyadic
cubes. Every such a family consists of ``thick with respect to $S$'' noneoverlapping dyadic cubes and
covers the set $S$ up to a set of  $\mathcal{H}^{d}_{\infty}$-zero content.
The sequence of families of cubes will play a role of a skeleton for an extension operator constructed in Section 5.
Furthermore, in Section 3 we introduce several new combinatorial concepts that can be of independent interest. Namely,
we introduce \textit{$(d,\lambda,c)$-covering cubes}, \textit{$(d,\lambda,c)$-shadows}, and \textit{$(d,\lambda,c)$-icebergs}. Furthermore, we introduce
\textit{hollow cubes} which will be natural substitutions for porous cubes.
Those concepts will be keystone tools
in proving a direct trace theorem in Section 7.

\textit{In Section 4} we introduce far-reaching generalizations of the Calder\'on-type maximal functions and establish some elementary properties of them.
They will be keystone tools in derivation of estimates for the gradients of extensions of functions from $S$ to $\mathbb{R}^{n}$. Furthermore, we introduce some function spaces
needed in solving Problem B.

\textit{In Section 5} we built a new extension operator.
We think that our construction is the most interesting part of the present paper.
It provides a far-reaching generalization of the classical extension operator introduced by H.~Whitney.
Roughly speaking, in contrast to the previously used extension methods, the new operator exploits only those values of the trace function
which are concentrated on the ``thick with respect to $S$'' dyadic cubes. Namely, the cubes from the families constructed
in Section 3 do the job.


\textit{Section 6} contains the so-called reverse trace theorem. The proof depends heavily
on estimates of Section 5 and the reflexivity of the classical Sobolev spaces
$W_{p}^{1}(\mathbb{R}^{n})$ for $p \in (1,\infty)$.

\textit{Section 7} is devoted to the so-called direct trace theorem with a detailed proof. The proof is based on the tools introduced in Section 3.
Furthermore, the section contains some lemmas of independent interest.

Finally, \textit{in Section 8} we
present a complete solution to Problem B.


\section{Preliminaries}

Throughout the paper, $C,C_{1},C_{2},...$ will be generic positive
constants. These constants can change even
in a single string of estimates. The dependence of a constant on certain parameters is expressed, for example, by
the notation $C=C(n,p,k)$. We write $A \approx B$ if there is a constant $C \geq 1$ such that $A/C \le B \le C A$.
For any $c \in \mathbb{R}$ we denote by $[c]$ the integer part of $c$, i.e.,
\begin{equation}
\notag
[c]:=\max\{k \in \mathbb{Z}:k \le c\}.
\end{equation}

We use notation $\mathbb{N}_{0}:=\mathbb{N}\cup\{0\}$. By $\mathbb{R}^{n}$ we denote the linear space of all strings $x=(x_{1},...,x_{n})$ of real numbers \textit{equipped  with the uniform  norm} $\|\cdot\|:=\|\cdot\|_{\infty}$, i.e., $\|x\|:=\|x\|_{\infty}:=\max\{|x_{1}|,...,|x_{n}|\}$. Given a set $E \subset \mathbb{R}^{n}$, we denote by $\operatorname{cl}E$, $\operatorname{int}E$ and $E^{c}$
the closure, the interior, and the complement (in $\mathbb{R}^{n}$) of $E$, respectively. Given a set $E \subset \mathbb{R}^{n}$, we denote by
$\chi_{E}$ the characteristic function of $E$ and by $\# E$ we denote \textit{the
cardinality} of $E$.

Given a family $\mathcal{G}$ of subsets of $\mathbb{R}^{n}$, by $M(\mathcal{G})$ we denote its covering multiplicity, i.e.,
the minimal $M' \in \mathbb{N}_{0}$ such that every point $x \in \mathbb{R}^{n}$ belongs to at most $M'$ sets from $\mathcal{G}$.
Let  $\mathcal{G}=\{G_{\alpha}\}_{\alpha \in \mathcal{I}}$ be a family of subsets of $\mathbb{R}^{n}$ and let $U \subset \mathbb{R}^{n}$ be a set.
We define
\textit{the restriction of the family} $\mathcal{G}$ to the set $U$ by letting
\begin{equation}
\label{restriction of the family}
\mathcal{G}|_{U}:=\{G \in \mathcal{G}: G \subset U\}.
\end{equation}
We say that a family $\mathcal{G}$ of subsets of $\mathbb{R}^{n}$ is \textit{nonoverlapping} if
different sets in $\mathcal{G}$ have disjoint interiors.

In what follows, by \textit{a measure} we will mean only a nonnegative Borel regular (outer) measure on $\mathbb{R}^{n}$. By $\mathcal{L}^{n}$ we denote the classical $n$-dimensional Lebesgue measure in
$\mathbb{R}^{n}$.  We say that a set $E \subset \mathbb{R}^{n}$ is \textit{measurable} if it is $\mathcal{L}^{n}$-measurable.
Given a measurable set $E \subset \mathbb{R}^{n}$, by $L_{0}(E)$ we denote the linear space of all equivalence classes of
measurable functions $f: E \to [-\infty,+\infty]$. Given a Borel set $E \subset \mathbb{R}^{n}$, by $\mathfrak{B}(E)$ we denote the set of all Borel
functions $f: E \to [-\infty,+\infty]$. If $\mathfrak{m}$ is a measure and $f \in \mathfrak{B}(\operatorname{supp}\mathfrak{m})$, then by the symbol
$[f]_{\mathfrak{m}}$ we will denote the $\mathfrak{m}$-equivalence class of $f$.
Given a measure $\mathfrak{m}$ and a nonempty Borel set $S \subset \mathbb{R}^{n}$, \textit{the restriction of} $\mathfrak{m}$ to $S$ is the
measure defined by
\begin{equation}
\notag
\mathfrak{m}\lfloor_S (G) := \mathfrak{m}(G \cap S) \quad \hbox{for any Borel set} \quad G \subset \mathbb{R}^{n}.
\end{equation}
Given a measure $\mathfrak{m}$  and $p \in [1,\infty)$, there is a natural isomorphism between the spaces $L_{p}(\mathbb{R}^{n},\mathfrak{m})$
and $L_{p}(\operatorname{supp}\mathfrak{m},\mathfrak{m})$ respectively. Keeping in mind this fact we will use the
symbol $L_{p}(\mathfrak{m})$ to denote any of that spaces. Similarly, we identify $L_{p}^{loc}(\mathbb{R}^{n},\mathfrak{m})$ and $L_{p}^{loc}(\mathfrak{m})$.

Given $f \in L^{\text{\rm loc}}_{1}(\mathfrak{m})$ and a Borel set $G \subset \mathbb{R}^{n}$ with $\mathfrak{m}(G) < +\infty$, we put
\begin{equation}
\label{eq.average}
f_{G,\mathfrak{m}}:=\fint\limits_{G}f(x)\,d\mathfrak{m}(x):=
\begin{cases}
\frac{1}{\mathfrak{m}(G)}\int\limits_{G}f(x)\,d\mathfrak{m}(x) \quad \hbox{if} \quad \mathfrak{m}(G) > 0;\\
0 \quad  \hbox{if} \quad \mathfrak{m}(G) = 0.
\end{cases}
\end{equation}
Let $\{\mathfrak{m}_{k}\}:=\{\mathfrak{m}_{k}\}_{k \in \mathbb{N}_{0}}$ be a sequence of Borel measures. We say that $\{\mathfrak{m}_{k}\}$ is \textit{an admissible sequence of measures}
if for each $k \in \mathbb{N}_{0}$ the measure $\mathfrak{m}_{k}$ is absolutely continuous with respect to $\mathfrak{m}_{j}$
for any $j \in \mathbb{N}_{0}$.
Given $p \in [1,\infty)$ and an admissible sequence of measures $\{\mathfrak{m}_{k}\}$, we put
\begin{equation}
\label{eqq.intersection_of_all_lp}
L_{p}(\{\mathfrak{m}_{k}\}):=\bigcap\limits_{k=0}^{\infty} L_{p}(\mathfrak{m}_{k}).
\end{equation}
Furthermore, given a set $S \supset \cap_{k=0}^{\infty}\operatorname{supp}\mathfrak{m}_{k}$, by $\mathfrak{B}(S) \cap L_{p}(\{\mathfrak{m}_{k}\})$ we mean a linear subspace of $\mathfrak{B}(S)$
composed of all Borel functions whose $\mathfrak{m}_{0}$-equivalence classes belong to $L_{p}(\{\mathfrak{m}_{k}\})$. As a result, given $f \in \mathfrak{B}(S) \cap L_{p}(\{\mathfrak{m}_{k}\})$,
we have $[f]_{\mathfrak{m}_{0}} \in L_{p}(\{\mathfrak{m}_{k}\})$.

Throughout this paper, the word ``cube'' will always mean \textit{a closed cube} in $\mathbb{R}^{n}$ whose sides are
parallel to the coordinate axes. We let $Q_{l}(x)$ denote the cube in $\mathbb{R}^{n}$ centered at $x$ with side
length $l$, i.e.,
$Q_{l}(x):=\prod_{i=1}^{n}[x_{i}-\frac{l}{2},x_{i}+\frac{l}{2}].$
Given $c > 0$ and a cube $Q$, we let $cQ$ denote the dilation of $Q$ with respect to its
center by a factor of $c$, i.e., $cQ_{l}(x):= Q_{cl}(x).$
Given a cube $Q$ we will denote by $l(Q)$ the diameter of $Q$ computed in the $\|\cdot\|_{\infty}$-norm, i.e., its side length.

By \textit{a dyadic cube} we mean an arbitrary \textit{closed cube} $Q_{k,m}:=\prod_{i=1}^{n} [\frac{m_{i}}{2^{k}}, \frac{m_{i}+1}{2^{k}}]$
with $k \in \mathbb{Z}$ and $m=(m_{1},...,m_{n}) \in \mathbb{Z}^{n}$. For each $k \in \mathbb{Z}$, by $\mathcal{D}_{k}$ we denote the family of all closed dyadic cubes
of side lengths $2^{-k}$. We set
\begin{equation}
\notag
\mathcal{D}:=\bigcup_{k \in \mathbb{Z}}\mathcal{D}_{k}, \quad \mathcal{D}_{+}:=\bigcup_{k \in \mathbb{N}_{0}}\mathcal{D}_{k}.
\end{equation}
For each $c \geq 1$ and $k \in \mathbb{Z}$, we set
\begin{equation}
\notag
c\mathcal{D}_{k}:=\{cQ_{k,m}: m \in \mathbb{Z}^{n}\}.
\end{equation}

Given $k \in \mathbb{Z}$ and $m \in \mathbb{Z}^{n}$, we define the set of \textit{$k$-neighboring indices for $m$} by letting
\begin{equation}
\notag
\operatorname{n}_{k}(m):=\{m' \in \mathbb{Z}^{n}:Q_{k,m'} \subset 3Q_{k,m}\}.
\end{equation}
Furthermore, given a cube $Q \in \mathcal{D}_{k}$, the family of all \textit{neighboring cubes for $Q$} is defined as
\begin{equation}
\notag
\operatorname{n}(Q):=\{Q' \in \mathcal{D}_{k}| Q' \cap Q \neq \emptyset\}.
\end{equation}


We will also need the following simple facts. We omit the elementary proofs.

\begin{Prop}
\label{Prop21''''}
Let $c \geq 1$ and $k \in \mathbb{Z}$. Let cubes $Q,Q' \in \mathcal{D}_{k}$ be such that $cQ \cap Q' \neq \emptyset$. Then
$[c]Q \cap Q' \neq \emptyset$.
\end{Prop}

\begin{Prop}
\label{Prop21}
Let $c \geq 1$ and $k \in \mathbb{Z}$. Then $M(c\mathcal{D}_{k})  \le ([c]+2)^{n}.$
\end{Prop}

Given parameters $\sigma \in [1,\infty)$, $s \in [0,n]$ and a scale $R \in (0,+\infty]$, for any function $f \in L_{\sigma}^{\operatorname{loc}}(\mathbb{R}^{n})$ we define the restricted
fractional Hardy--Littlewood maximal function of $f$ by the formula
\begin{equation}
\label{eq2.2}
\mathcal{M}^{R}_{\sigma,s}[f](x):=\sup\limits_{l \in (0,R)}\Bigl(l^{s}\fint \limits_{Q_{l}(x)}|f(y)|^{\sigma}\,dy\Bigr)^{\frac{1}{\sigma}}, \quad x \in \mathbb{R}^{n}.
\end{equation}
In the case $s = 0$ we write $\mathcal{M}^{R}_{\sigma}[f]$ instead of $\mathcal{M}^{R}_{\sigma,0}[f]$. Furthermore, we put $\mathcal{M}_{\sigma}[f]:=\mathcal{M}^{+\infty}_{\sigma}[f]$.
We recall the classical fact, which is an immediate consequence of Theorem 1 from Chapter 1 in \cite{St}.

\begin{Prop}
\label{Prop.max.funct}
Let $1 \le \sigma < p < \infty$. Then there is a constant $C=C(n,\sigma,p) > 0$ such that, for any $R \in (0,+\infty]$,
\begin{equation}
\int\limits_{\mathbb{R}^{n}}\Bigl(\mathcal{M}^{R}_{\sigma}[f](x)\Bigr)^{p}\,dx \le
C \int\limits_{\mathbb{R}^{n}}|f(x)|^{p}\,dx \quad \hbox{for all} \quad f \in L_{p}(\mathbb{R}^{n}).
\end{equation}
\end{Prop}
Sometimes it will be convenient to use maximal functions  to  estimate from above
the average values of functions.

\begin{Prop}
\label{Prop2.2}
Let $p,\sigma \in [1,\infty)$. Let $\underline{Q}=Q_{\underline{l}}(\underline{x})$ be a cube with $\underline{l} > 0$ and
$\Omega \subset \underline{Q}$ be a Borel set. Then
\begin{equation}
\mathcal{L}^{n}(\Omega) \Bigl(\fint\limits_{\underline{Q}}|f(x)|^{\sigma}\,dx\Bigr)^{\frac{p}{\sigma}} \le 2^{\frac{np}{\sigma}}\int\limits_{\Omega}
\Bigl(\mathcal{M}_{\sigma}[f](x)\Bigr)^{p}\,dx \quad \hbox{for all} \quad f \in L^{\operatorname{loc}}_{\sigma}(\mathbb{R}^{n}).
\end{equation}
\end{Prop}

\begin{proof}
Note that $\underline{Q} \subset Q_{2\underline{l}}(y)$ for every $y \in \Omega$. This and \eqref{eq2.2} gives
\begin{equation}
\notag
\Bigl(\fint\limits_{\underline{Q}}|f(x)|^{\sigma}\,dx\Bigr)^{\frac{1}{\sigma}} \le
2^{\frac{n}{\sigma}}\Bigl(\fint\limits_{Q_{2\underline{l}}(y)}|f(x)|^{\sigma}\,dx\Bigr)^{\frac{1}{\sigma}} \le 2^{\frac{n}{\sigma}}
\mathcal{M}_{\sigma}[f](y) \quad  \hbox{for all} \quad y \in \Omega.
\end{equation}
Hence, using this observation we obtain the required estimate
\begin{equation}
\notag
\mathcal{L}^{n}(\Omega) \Bigl(\fint\limits_{\underline{Q}}|f(x)|^{\sigma}\,dx\Bigr)^{\frac{p}{\sigma}}
\le 2^{\frac{np}{\sigma}}\mathcal{L}^{n}(\Omega)\inf\limits_{y \in \Omega}\Bigl(\mathcal{M}_{\sigma}[f](y)\Bigr)^{p} \le 2^{\frac{np}{\sigma}}\int\limits_{\Omega}
\Bigl(\mathcal{M}_{\sigma}[f](x)\Bigr)^{p}\,dx.
\end{equation}
\end{proof}

In proving the main results of this paper we will use a quite delicate fact, which in turn is a particular case of a remarkable result by Sawyer
\cite{Sawyer}.

\begin{ThA}
\label{Th.Sawyer}
Let $d \in [0, n]$, $s \in [0, n)$, $R \in (0,+\infty]$ and $q \in (1,\infty)$. Let $\mathfrak{m}$ be
a Radon measure on $\mathbb{R}^{n}$ such that, for some (universal) positive constant $C(\mathfrak{m},R) > 0$,
\begin{equation}
\notag
\mathfrak{m}(Q_{l}(x)) \le C(\mathfrak{m}) l^{d} \quad \text{for all} \quad x \in \mathbb{R}^{n}, l \in (0,R).
\end{equation}
If $qs \geq n-d$, then the operator $\mathcal{M}^{R}_{1,s}$ is bounded from $L_{q}(\mathbb{R}^{n})$  into $L_{q}(\mathbb{R}^{n},\mathfrak{m})$.
Furthermore, the operator norm depends only on $n,d,q,s$ and $C(\mathfrak{m},R)$.
\end{ThA}

The following fact is elementary. We omit the proof.

\begin{Prop}
\label{Prop2.4}
Let $\mathfrak{m}$ be a measure on $\mathbb{R}^{n}$ and $f \in L_{1}(\mathfrak{m})$.
Let $\{E_{\alpha}\}_{\alpha \in \mathcal{I}}$ be an at most countable family of subsets of $\mathbb{R}^{n}$ such that $M(\{E_{\alpha}\}_{\alpha \in \mathcal{I}}) < +\infty$.
Then
\begin{equation}
\sum\limits_{\alpha \in \mathcal{I}}\int\limits_{E_{\alpha}}f(x)\,d\mathfrak{m}(x) \le M(\{E_{\alpha}\}_{\alpha \in \mathcal{I}})\int\limits_{E}f(x)\,d\mathfrak{m}(x),
\end{equation}
where we set $E:=\cup_{\alpha \in \mathcal{I}}E_{\alpha}$.
\end{Prop}

\subsection{Coverings}

In what follows, we will use the following notation. Given a family of cubes $\{Q_{\alpha}\}_{\alpha \in \mathcal{I}}$ in $\mathbb{R}^{n}$,
we put
\begin{equation}
\notag
l_{\alpha}:=\operatorname{diam}Q_{\alpha}=l(Q_{\alpha}), \quad \alpha \in \mathcal{I}.
\end{equation}

Given two nonoverlapping  families $\mathcal{Q}:=\{Q_{\alpha}\}_{\alpha \in \mathcal{I}}$ and
$\mathcal{Q}':=\{Q_{\alpha'}\}_{\alpha \in \mathcal{I'}}$  of dyadic cubes (with $\mathcal{I}, \mathcal{I}' \subset \mathbb{Z} \times \mathbb{Z}^{n}$) we write
$\mathcal{Q} \succeq \mathcal{Q}'$ provided that for every $\alpha' \in \mathcal{I'}$
there exists a unique $\alpha \in \mathcal{I}$ such that  $Q_{\alpha} \supset Q_{\alpha'}$.
If, in addition, $l_{\alpha} > l_{\alpha'}$ for all such $\alpha$ and $\alpha'$, we write
$\mathcal{Q} \succ \mathcal{Q}'$.
We say that two families of dyadic nonoverlapping cubes $\mathcal{Q}:=\{Q_{\alpha}\}_{\alpha \in \mathcal{I}}$ and $\mathcal{Q}':=\{Q_{\alpha'}\}_{\alpha' \in \mathcal{I}'}$ \textit{are comparable}
provided that
\begin{equation}
\notag
\hbox{either} \quad \mathcal{Q} \succeq \mathcal{Q}' \quad \hbox{or} \quad
\mathcal{Q}' \succeq \mathcal{Q}.
\end{equation}
Otherwise we call the corresponding families \textit{incomparable.}

Given a set $E \subset \mathbb{R}^{n}$, \textit{by a covering of $E$} we mean any family $\{U_{\beta}\}_{\beta \in \mathcal{J}}$
of subsets of $\mathbb{R}^{n}$ such that $E \subset \cup_{\beta \in \mathcal{J}}U_{\beta}$.
Any nonoverlapping family $\mathcal{Q} \subset \mathcal{D}$ that covers $E$ will be called
\textit{a dyadic nonoverlapping covering of} $E$.

\subsection{Hausdorff contents and Hausdorff measures}

In this paper instead of the classical
Hausdorff measures and Hausdorff contents, it will be convenient to work with their corresponding dyadic analogs.

Given a nonempty set $E \subset \mathbb{R}^{n}$ and $d \in [0,n]$,  we set, for any $\delta \in (0,\infty]$,
\begin{equation}
\label{eq2.1}
\begin{split}
\mathcal{H}^{d}_{\delta}(E):=\inf \sum\limits_{\alpha \in \mathcal{I}} (l_{\alpha})^{d},
\end{split}
\end{equation}
where
the infimum is taken over all dyadic nonoverlapping coverings $\{Q_{\alpha}\}_{\alpha \in \mathcal{I}}$
of the set $E$ such that $l_{\alpha} < \delta$ for all $\alpha \in \mathcal{I}$.
The value $\mathcal{H}^{d}_{\infty}(E)$ is called
\textit{the $d$-Hausdorff  content} of the set $E$.
We define \textit{the $d$-Hausdorff measure of $E$}  by the formula $\mathcal{H}^{d}(E):=\lim_{\delta \to 0}\mathcal{H}^{d}_{\delta}(E).$

\begin{Remark}
\label{Rem2.1}
It is easy to show that the $d$-Hausdorff contents and the $d$-Hausdorff measures
defined above coincide, up to some universal constants,
with their classical predecessors.

Combining this observation with Lemma 4.6 in \cite{Mat} we get $\mathcal{H}^{d}(E)=0 \Longleftrightarrow \mathcal{H}^{d}_{\delta}(E)=0$ for every
$\delta \in (0,+\infty].$
\end{Remark}

In the sequel, we will deal not only with coverings but also with some families which, for a given set $E$,
cover $E$ with some small error. More precisely, we introduce the following concept.

\begin{Def}
\label{dyadic.cov}
Let $E \subset \mathbb{R}^{n}$ be an arbitrary set. Given $d \in [0,n]$, we say that a family $\{U_{\beta}\}_{\beta \in \mathcal{J}}$
is \textit{a $d$-almost covering of} $E$ if there exists a set $E' \subset E$ with $\mathcal{H}^{d}_{\infty}(E')=0$ such that
$\{U_{\beta}\}_{\beta \in \mathcal{J}}$ is a covering of $E \setminus E'$.
\end{Def}

Recall a simple fact (for an elementary proof see \cite{Evans}, Section 2.4.3).

\begin{Prop}
\label{Prop.2.5}
Let $d \in [0,n)$ and $F \in L_{1}^{\operatorname{loc}}(\mathbb{R}^{n})$. Then there exists a set $E_{F}$
with $\mathcal{H}^{d}(E_{F})=0$ such that
\begin{equation}
\notag
\lim\limits_{l \to 0}\frac{1}{l^{d}}\int\limits_{Q_{l}(x)}|F(y)|\,dy = 0, \quad \text{for every} \quad x \in \mathbb{R}^{n} \setminus E_{F}.
\end{equation}
\end{Prop}

Given $p \in (1,\infty)$, recall the notion of \textit{$C_{1,p}$-capacity} (see e.g., \cite{A}, Section 2.1). In what follows, we say that some property holds $(1,p)$-quasieverywhere ($(1,p)$-q.e. for short)
if it holds everywhere except a set $E \subset \mathbb{R}^{n}$ with $C_{1,p}(E)=0$.
The following proposition summarizes some connections between the $C_{1,p}$-capacities and
the $d$-Hausdorff measures (see Theorems 5.1.9, 5.1.13 of \cite{A} for details).

\begin{Prop}
\label{Hausdorf_capacity}
Let $p \in (1,n]$ and let $E \subset \mathbb{R}^{n}$. If $\mathcal{H}^{n-p}(E) < +\infty$ then $C_{1,p}(E) = 0$.
If $C_{1,p}(E) = 0$, then $\mathcal{H}^{d}(E)=0$ for every $d > n-p$.
\end{Prop}

\subsection{Thick sets and Frostman-type measures}

As we briefly mentioned in the introduction, cubes whose intersections
with a given closed set $S \subset \mathbb{R}^{n}$ are ``massive'' enough will be important for the construction of our extension operator in Section 5. We formalize this as follows.

Given a nonempty set $E \subset \mathbb{R}^{n}$ and numbers $d \in (0,n]$, $\lambda \in (0,1]$, we say that a \textit{cube $Q$ with $l(Q) \in (0,1]$ is $(d,\lambda)$-thick  with respect to the set $E$} if
\begin{equation}
\label{eq3.1}
\mathcal{H}^{d}_{\infty}(Q \cap E) \geq \lambda (l(Q))^{d}.
\end{equation}
We say that a \textit{cube $Q$ with $l(Q) \in (0,1]$ is $(d,\lambda)$-thin with respect to the set $E$} if
\begin{equation}
\label{eq3.2}
\mathcal{H}^{d}_{\infty}(Q \cap E) < \lambda (l(Q))^{d}.
\end{equation}
We also define the family
\begin{equation}
\label{eqq.thick_family}
\begin{split}
&\mathcal{F}_{E}(d,\lambda):=\{Q: \hbox{$Q$ is $(d,\lambda)$-thick w.r.t. $E$}\}.\\
\end{split}
\end{equation}

For the construction of the extension operator we will need a very special sequence of measures.

\begin{Def}
\label{Def.Frostman}
Let $d \in (0,n]$  and let $S \subset \mathbb{R}^{n}$ be a closed set with $\mathcal{H}^{d}_{\infty}(S) > 0$. We say that a sequence of
measures $\{\mathfrak{m}_{k}\}=\{\mathfrak{m}_{k}\}_{k \in \mathbb{N}_{0}}$ is \textit{$d$-Frostman on $S$} if  the following conditions hold:

{\rm (\textbf{M}1)} for every $k \in \mathbb{N}_{0}$,
\begin{equation}
\label{eq28'}
\operatorname{supp}\mathfrak{m}_{k}=S;
\end{equation}

{\rm (\textbf{M}2)} there exists a constant $C_{1} > 0$ such that, for each $k \in \mathbb{N}_{0}$,
\begin{equation}
\label{eq29'}
\mathfrak{m}_{k}(Q_{l}(x)) \le C_{1} l^{d} \quad \hbox{for every} \quad x \in \mathbb{R}^{n} \quad \hbox{and every} \quad l \in (0,2^{-k}];
\end{equation}

{\rm (\textbf{M}3)} there exists a constant $C_{2} > 0$ such that, for each $k \in \mathbb{N}_{0}$,
\begin{equation}
\label{eq30'}
\mathfrak{m}_{k}(Q_{k,m} \cap S) \geq
C_{2}\mathcal{H}^{d}_{\infty}(Q_{k,m} \cap S) \quad \hbox{for every}  \quad m \in \mathbb{Z}^{n};
\end{equation}

{\rm (\textbf{M}4)}  $\mathfrak{m}_{k}=w_{k}\mathfrak{m}_{0}$ with $w_{k} \in L_{\infty}(\mathfrak{m}_{0})$ for every $k \in \mathbb{N}_{0}$
and there exists a constant $C_{3} > 0$ such that, for all $k \in \mathbb{N}_{0}$ and $j \in \mathbb{N}$,
\begin{equation}
\label{eq31'}
\frac{1}{C_{3}}2^{(d-n)j}w_{k+j}(x) \le w_{k}(x) \le C_{3}w_{k+j}(x) \quad \hbox{for} \quad \mathfrak{m}_{0}-\hbox{a.e.} \quad x \in S.
\end{equation}

\textit{The class of sequences of
measures, which are $d$-Frostman on $S$} will be denoted by $\mathfrak{M}^{d}(S)$.
\end{Def}

\begin{Remark}
It is easy to see that there exist smallest constants $C_{1} > 0$ and $C_{3} > 0$
for which \eqref{eq29'} and \eqref{eq31'} hold. We denote them by $\operatorname{C}_{\{\mathfrak{m}_{k}\},1}$ and $\operatorname{C}_{\{\mathfrak{m}_{k}\},3}$, respectively.
Similarly, there exists
largest constant $C_{2} > 0$ for which \eqref{eq30'} holds, we denote it by $\operatorname{C}_{\{\mathfrak{m}_{k}\},2}$.
\end{Remark}
\hfill$\Box$

\begin{Example}
\label{Ex.2.1}
Let $d^{\ast} \in (0,n]$ and let $S \subset \mathbb{R}^{n}$ be a closed Ahlfors--David $d^{\ast}$-regular set. It is easy to see that, given $d \in (0,d^{\ast})$,
letting $\mathfrak{m}_{k}:=2^{k(d^{\ast}-d)}\mathcal{H}^{d^{\ast}}\lfloor_{S}$, $k \in \mathbb{N}_{0}$ we obtain a sequence of measures $\{\mathfrak{m}_{k}\} \in \mathfrak{M}^{d}(S)$.
\end{Example}

The following obvious observation will be currently used in the sequel.

\begin{Remark}
\label{Rem.Frostman}
Let $S \subset \mathbb{R}^{n}$ be a closed set with $\mathcal{H}^{d}_{\infty}(S) > 0$ for some $d \in (0,n]$. Let
$\{\mathfrak{m}_{k}\} \in \mathfrak{M}^{d}(S)$.
By \eqref{eq31'} it is easy to see that, given $p \in [1,\infty)$,
$f \in L_{p}(\mathfrak{m}_{k_{0}})$ for some fixed $k_{0} \in \mathbb{N}_{0}$
if and only if $f \in L_{p}(\mathfrak{m}_{k})$ for all $k \in \mathbb{N}_{0}$. Hence,
$L_{p}(\{\mathfrak{m}_{k}\})=L_{p}(\mathfrak{m}_{k_{0}})$ for each $k_{0} \in \mathbb{N}_{0}$.

If $S \subset \mathbb{R}^{n}$ is a compact set, then an application H\"older's inequality gives for any $1<q<p<\infty$,
\begin{equation}
\label{eqq.simplest_imbedding}
\|f|L_{q}(\mathfrak{m}_{0})\| \le \Bigl(\mathfrak{m}_{0}(S)\Bigr)^{\frac{p-q}{p}}\|f|L_{p}(\mathfrak{m}_{0})\| \quad \text{for all} \quad f \in L_{p}(\{\mathfrak{m}_{k}\}).
\end{equation}
\end{Remark}

Now we recall a variant of the classical Frostman-type theorem formulated in the form
adapted for our purposes (compare with Theorem 5.1.12 in \cite{A}). We will often use it in the sequel. One can find the detailed proof in Section 3.4
of \cite{TV}.

\begin{ThA}
Let $d \in (0,n]$  and let $S \subset \mathbb{R}^{n}$ be a closed set with $\mathcal{H}^{d}_{\infty}(S) > 0$. Then $\mathfrak{M}^{d}(S) \neq \emptyset$.
\end{ThA}

\subsection{Sobolev spaces}

Recall that, given parameters $p \in [1,\infty]$, $n \in \mathbb{N}$ and an open set $G \subset \mathbb{R}^{n}$, the Sobolev space
$W_{p}^{1}(G)$  is the linear space of all (equivalence classes of) real-valued functions $F \in L_{p}(G)$
whose distributional partial derivatives $D^{\gamma}F$, $|\gamma| = 1$ on~%
$G$ belong to $L_{p}(G)$. This space is equipped with the norm
\begin{equation}
\label{eq1.1}
\|F|W^{1}_{p}(G)\|:=\|F|L_{p}(G)\|+\sum\limits_{|\gamma| = 1}\|D^{\gamma}F|L_{p}(G)\|.
\end{equation}
Given $p \in [1,\infty]$, by $W_{p}^{1,\operatorname{loc}}(G)$ we denote the linear space of all (equivalence classes of) real-valued functions $F \in L^{\operatorname{loc}}_{p}(G)$
whose distributional partial derivatives $D^{\gamma}F$, $|\gamma| = 1$ on~%
$G$ belong to $L^{\operatorname{loc}}_{p}(G)$.

We will use the following notation. Given an element $F \in W^{1}_{p}(G)$, we will denote by $\nabla F$ its distributional gradient on $G$
and we put
\begin{equation}
\notag
\|\nabla F(x)\|:=\|\nabla F(x)\|_{\infty}:=\max\limits_{|\gamma|=1}\{|D^{\gamma}F(x)|\}, \quad x \in G.
\end{equation}

\begin{Def}
\label{Def.good.rep}
Let $p \in (1,n]$. Given $F \in L^{\operatorname{loc}}_{1}(\mathbb{R}^{n})$, we say that a Borel
function $\overline{F}$ is a $(1,p)$-good
representative of $F$ if the function $\overline{F}$ has Lebesgue points everywhere except a set $E_{F} \subset \mathbb{R}^{n}$ with
$C_{1,p}(E_{F}) = 0$.
\end{Def}

The following property is a very particular case of Theorem 6.2.1 of \cite{A}.

\begin{Prop}
\label{Th2.4}
Given $p \in (1,n]$, for each $F \in W_{p}^{1}(\mathbb{R}^{n})$ there exists a $(1,p)$-good representative $\overline{F}$ of $F$.
\end{Prop}

As we have already mentioned in the present paper, we
consider almost sharp intrinsic descriptions of traces of $W_{p}^{1}(\mathbb{R}^{n})$-spaces.
This motivates us to introduce the following concept.

\begin{Def}
\label{Def.d.trace}
Let $p \in (1,n]$, $d \in (n-p,n]$ and let $S \subset \mathbb{R}^{n}$ be a Borel
set  with $\mathcal{H}^{d}_{\infty}(S) > 0$. Given $F \in W_{p}^{1}(\mathbb{R}^{n})$, we define \textit{the $d$-trace} $F|_{S}^{d}$ of the element $F$ to the set $S$ as the
equivalence class $[\overline{F}]_{d}$ of the pointwise restriction to $S$ of any $(1,p)$-good representative $\overline{F}$ of $F$ modulo coincideness $\mathcal{H}^{d}$-a.e., i.e.,
\begin{equation}
\notag
F|_{S}^{d}:=\{\widetilde{f} \in \mathfrak{B}(S) : \widetilde{f}(x)=\overline{F}(x) \text{ for $\mathcal{H}^{d}$-a.e.\ $x \in S$}\}.
\end{equation}
\textit{We define the $d$-trace space} $W^{1}_{p}(\mathbb{R}^{n})|^{d}_{S}$ as the linear space of  $d$-traces $f=F|^{d}_{S}$ of all elements $F \in W_{p}^{1}(\mathbb{R}^{n})$
to the set $S$ equipped with the quotient-space norm, i.e., for each $f \in W^{1}_{p}(\mathbb{R}^{n})|^{d}_{S}$,
\begin{equation}
\label{eq214}
\|f|W^{1}_{p}(\mathbb{R}^{n})|^{d}_{S}\|:=\inf\{\|F|W_{p}^{1}(\mathbb{R}^{n})\|: F \in W_{p}^{1}(\mathbb{R}^{n}) \hbox{ and } f = F|^{d}_{S}\}.
\end{equation}
We also define the corresponding \textit{$d$-trace operator} $\operatorname{Tr}|^{d}_{S}: W_{p}^{1}(\mathbb{R}^{n}) \to W^{1}_{p}(\mathbb{R}^{n})|^{d}_{S}$ by letting
\begin{equation}
\operatorname{Tr}|^{d}_{S}[F]:=F|^{d}_{S} \quad \hbox{for every} \quad F \in W_{p}^{1}(\mathbb{R}^{n}).
\end{equation}
\end{Def}

\begin{Remark}
\label{Rem25}
Note that Definition \ref{Def.d.trace} is correct. Indeed, by Proposition \ref{Hausdorf_capacity} if
$p \in (1,n]$, $d \in (n-p,n]$ and $\mathcal{H}^{d}_{\infty}(S) > 0$, then $C_{1,p}(S) > 0$. Hence, given a Sobolev element
$F \in W_{p}^{1}(\mathbb{R}^{n})$, the restriction $\overline{F}|_{S}$
of any $(1,p)$-good representative of $F$ to the set $S$ is well defined. Furthermore, it follows from Proposition \ref{Hausdorf_capacity}
that the $d$-trace $F|_{S}^{d}$ does not depend on the choice of a $(1,p)$-good representative $\overline{F}$ of $F$.

Clearly, the $d$-trace operator $\operatorname{Tr}|^{d}_{S}$ is a linear and bounded mapping from $W_{p}^{1}(\mathbb{R}^{n})$ to $W^{1}_{p}(\mathbb{R}^{n})|^{d}_{S}$.
\end{Remark}

\begin{Remark}
\label{Rem26}
By the H\"older's inequality $W^{1,\operatorname{loc}}_{p}(\mathbb{R}^{n}) \subset W^{1,\operatorname{loc}}_{q}(\mathbb{R}^{n})$ for all $1 \le q \le p < \infty$.
Hence, given parameters $p \in (1,n]$, $d \in (n-p,n]$, and a compact
set $S \subset \mathbb{R}^{n}$ with $\mathcal{H}^{d}_{\infty}(S) > 0$,
it is easy to show using smooth cut-off functions that, for each $\varepsilon \in (0,\min\{p-(n-d),p-1\})$,
\begin{equation}
\notag
W^{1}_{p}(\mathbb{R}^{n})|^{d}_{S} \subset W^{1}_{p-\varepsilon}(\mathbb{R}^{n})|^{d}_{S}
\end{equation}
and the operator $\operatorname{Tr}|^{d}_{S}$ is well defined on $W^{1}_{p-\varepsilon}(\mathbb{R}^{n})$.
\end{Remark}

Remarks \ref{Rem25}, \ref{Rem26} justify the following definition.

\begin{Def}
\label{def.ext.oper}
Let $p \in (1,n]$, $d \in (n-p,n]$, and
$\varepsilon \in [0,\min\{p-(n-d),p-1\})$. Let $S \subset \mathbb{R}^{n}$ be a compact set with $\mathcal{H}^{d}_{\infty}(S) > 0$.
By $\mathfrak{E}(S,p,d,\varepsilon)$ we denote the linear space of all mappings $\operatorname{Ext}:W_{p}^{1}(\mathbb{R}^{n})|^{d}_{S} \to W_{p-\varepsilon}^{1}(\mathbb{R}^{n})$ such that:

{\rm ($\mathfrak{E}$1)} $\operatorname{Ext}$ is linear and bounded;

{\rm ($\mathfrak{E}$2)} $\operatorname{Ext}$ is \textit{the right inverse} of the $d$-trace operator,
i.e., $\operatorname{Tr}|^{d}_{S} \circ \operatorname{Ext} = \operatorname{Id}$ on $W_{p}^{1}(\mathbb{R}^{n})|^{d}_{S}$.
\end{Def}

Now we make a simple but nontrivial observation.

\begin{Prop}
\label{Prop.d_trace_space_is_complete}
Let $p \in (1,n]$, $d \in (n-p,n]$. Let $S \subset \mathbb{R}^{n}$ be a compact set with $\mathcal{H}^{d}_{\infty}(S) > 0$. Then the space $W^{1}_{p}(\mathbb{R}^{n})|^{d}_{S}$
is a Banach space.
\end{Prop}

\begin{proof}
It is sufficient to show that the space $W^{1}_{p}(\mathbb{R}^{n})|^{d}_{S}$ is complete. Due to the standard facts from the Banach-space theory
it is sufficient to show that $N^{d}:=\{F \in W^{1}_{p}(\mathbb{R}^{n}):F|^{d}_{S}=0\}$
is a closed linear subspace of $W^{1}_{p}(\mathbb{R}^{n})$.
Indeed, we fix $F \in W^{1}_{p}(\mathbb{R}^{n})$ and a sequence $\{F_{k}\} \subset N^{d}$ such that $\|F-F_{k}|W^{1}_{p}(\mathbb{R}^{n})\| \to 0$, $k \to \infty$.
Combing Proposition 7.3.1 and Theorem 7.4.5 from \cite{Kos} we conclude that there is a $(1,p)$-good representative $\overline{F}$ of $F$ and there exist $(1,p)$-good representatives
$\overline{F}_{k}$ of $F_{k}$, $k \in \mathbb{N}$ such that for some strictly
increasing sequence $\{k_{s}\} \subset \mathbb{N}$, for some set $E_{1}$ with $C_{1,p}(E_{1})=0$ we have $\overline{F}_{k_{s}}(x) \to \overline{F}(x)$, $s \to \infty$ for each $x \in S \setminus E_{1}$.
On the other hand, there is a set $E_{2} \subset S$ such that $\mathcal{H}^{d}(E_{2}) = 0$ and $\overline{F}_{k}(x)=0$ for each $x \in S \setminus E_{2}$ for all $k \in \mathbb{N}$. As a result, taking into account Proposition \ref{Hausdorf_capacity} we deduce that
$\mathcal{H}^{d}(E_{1} \cup E_{2}) = 0$ and $\overline{F}(x)=\lim_{s \to \infty}\overline{F}_{k_{s}}(x)=0$ for all $x \in S \setminus (E_{1} \cup E_{2})$. This gives $F \in N^{d}$.
The proof is complete.
\end{proof}

The following proposition is a minor modification
of the classical Poincar\'e-type inequality.

\begin{Prop}
\label{Lm4.1}
For every $c,c' \geq 1$ there exists a constant $C=C(n,c,c') > 0$ such that, for
any cubes $Q_{1}:=Q_{l}(x_{1})$, $Q_{2}:=Q_{cl}(x_{2})$ with $l > 0$ and $\|x_{1}-x_{2}\| \le c'l$,
\begin{equation}
\label{eq4.4}
\fint\limits_{Q_{1}}\fint\limits_{Q_{2}}|F(y)-F(z)|\,dz\,dy \le C l \fint\limits_{(2c'+c)Q_{1}}\|\nabla F(y)\|\,dy \quad
\text{for all} \quad F \in W_{1}^{1,\operatorname{loc}}(\mathbb{R}^{n}).
\end{equation}
\end{Prop}

\begin{proof}
Fix cubes $Q_{1}$, $Q_{2}$ satisfying the
assumptions of the lemma. Recall the classical Poincar\'e-type inequality (see (7.45) in \cite{Tr}). More precisely, there exists a constant $C'=C'(n) > 0$
such that, for any cube $Q$, the following inequality
\begin{equation}
\label{eq4.4''}
\fint\limits_{Q}\Bigl|F(y)-\fint\limits_{Q}F(z)\,dz\Bigr|\,dy
\le C'(n) l \fint\limits_{Q}\|\nabla F(y)\|\,dy
\end{equation}
holds for all $F \in W_{1}^{1,\operatorname{loc}}(\mathbb{R}^{n})$.

Now we fix an arbitrary $F \in W_{1}^{1,\operatorname{loc}}(\mathbb{R}^{n})$. Since $\|x_{1}-x_{2}\| \le c'l$ we clearly have $(2c'+c)Q_{1} \supset Q_{2}$.
Hence by \eqref{eq4.4''},
\begin{equation}
\begin{split}
&\fint\limits_{Q_{1}}\fint\limits_{Q_{2}}|F(y)-F(z)|\,dy\,dz\\
&\le (2c'+c)^{2n}\fint\limits_{(2c'+c)Q_{1}}
\fint\limits_{(2c'+c)Q_{1}}\Bigl|F(y)-\fint\limits_{(2c'+c)Q_{1}} F(x)\,dx+\fint\limits_{(2c'+c)Q_{1}} F(x)\,dx-F(z)\Bigr|\,dy\,dz\\
&\le 2(2c'+c)^{2n}\fint\limits_{(2c'+c)Q_{1}}\Bigl|F(y)-\fint\limits_{(2c'+c)Q_{1}} F(x)\,dx\Bigr|\,dy\\
&\le 2 C'(n)(2c'+c)^{2n}
l \fint\limits_{(2c'+c)Q_{1}}\|\nabla F(y)\|\,dy.
\end{split}
\end{equation}
This gives \eqref{eq4.4} with $C(n)=2C'(n)(2c'+c)^{2n}$.
\end{proof}

Now we recall the key analytical feature of Frostman-type measures.
Namely, given a parameter $d \in (0,n]$, for each cube $Q$ and any
$F \in W^{1,\operatorname{loc}}_{\sigma}(\mathbb{R}^{n})$, for any large enough $\sigma$
we can control effectively how close are the average value of $F$ over
$Q$ calculated with respect to a $d$-Frostman measure $\mathfrak{m}$
and the average value of $F$ over $Q$ calculated with respect to the classical Lebesgue measures $\mathcal{L}^{n}$.
More precisely, the following result was established in \cite{TV}. Given $l > 0$ we set
$k(l):=[\log_{2}(\frac{1}{l})]$.

\begin{ThA}
\label{Th4.3}
Let $d \in (0,n]$, $\lambda \in (0,1]$, $\sigma \in (\max\{1,n-d\},n]$. Let
$S \subset \mathbb{R}^{n}$ be a~closed set with $\mathcal{H}^{d}_{\infty}(S) > 0$ and $\{\mathfrak{m}_{k}\} \in \mathfrak{M}^{d}(S)$.
Then there exists a constant $C > 0$ depending only on $C_{\{\mathfrak{m}_{k}\},i}$, $i=1,2,3$ and parameters
$n$, $d$, $\lambda$, $\sigma$
(but independent of a construction of the sequence $\{\mathfrak{m}_{k}\}$)
such that the following inequality
\begin{equation}
\label{eq4.10}
\fint\limits_{Q \cap S}\Bigl|F|^{d}_{S}(y)-\fint\limits_{Q}F(z)\,dz\Bigr|\,d\mathfrak{m}_{k(l)}(y) \le C l
\Bigl(\fint\limits_{Q}\sum\limits_{|\gamma|=1}|D^{\gamma}F(t)|^{\sigma}\,dt\Bigr)^{\frac{1}{\sigma}}
\end{equation}
holds for each cube $Q=Q_{l}(x) \in \mathcal{F}_{S}(d,\lambda)$ with $l \in (0,1]$ and any $F \in W_{\sigma}^{1,\operatorname{loc}}(\mathbb{R}^{n})$.
\end{ThA}

\begin{Remark}
We should note that Theorem \ref{Th4.3} is in fact an easy consequence of the corresponding
beautiful trace theorem for the Riesz potentials \cite{Ver}.
\end{Remark}
\hfill$\Box$

\begin{Prop}
\label{Proposition.2.10}
Let $d \in (0,n]$ and $p \in (\max\{1,n-d\},\infty)$.~Let $S \subset Q_{0,0}$ be a compact set with $\mathcal{H}^{d}_{\infty}(S) > 0$ and
$\{\mathfrak{m}_{k}\} \in \mathfrak{M}^{d}(S)$.~If $F \in W_{p}^{1,\operatorname{loc}}(\mathbb{R}^{n})$ then $f=F|^{d}_{S} \in L_{p}(\{\mathfrak{m}_{k}\})$.~Furthermore, there exists a constant $C > 0$ depending only on $n,d,p$ and $\operatorname{C}_{\mathfrak{m}_{k},i}$, $i=1,2,3$ such that
the following inequality
\begin{equation}
\label{eqq.222}
\|f|L_{p}(\mathfrak{m}_{0})\| \le C \|F|W_{p}^{1}(3Q_{0,0})\|
\end{equation}
holds for all $F \in W_{p}^{1,\operatorname{loc}}(\mathbb{R}^{n})$ with $f=F|^{d}_{S}$.
\end{Prop}

\begin{proof}
We fix a parameter $p \in (\max\{1,n-d\},\infty)$ and an element $F \in W_{p}^{1}(\mathbb{R}^{n})$.
By Remark \ref{Rem25},
the $d$-trace $f:=F|^{d}_{S}$ is well defined. Furthermore, the measure $\mathfrak{m}_{0}$ is absolutely continuous with respect to $\mathcal{H}^{d}\lfloor_{S}$.
Hence, using classical estimates (see, for example, Section 2 in \cite{Haj}) and telescoping arguments
it is easy to see that, for each $\delta \in (0,1)$, there is $C > 0$ (independent on $F$) such that
\begin{equation}
\notag
\Bigl|f(x)-\fint\limits_{Q_{0,0}} F(y)\,dy\Bigr| \le C\int\limits_{Q_{0,0}} \frac{\|\nabla F (y)\|}{\|x-y\|^{n-1}}\,dy
\le C\mathcal{M}_{1,1-\delta}^{1}[\|\nabla F\|](x) \quad \text{for} \quad \mathfrak{m}_{0}-\text{a.e.} \quad x \in S.
\end{equation}
Now we fix $\delta \in (0,1)$ so small that $p(1-\delta) > n-d$ and recall that $S \subset Q_{0,0}$.~An application of Theorem \ref{Th.Sawyer} with $\mathfrak{m}=\mathfrak{m}_{0}$ and $s=1-\delta$
in combination with H\"older's inequality gives
\begin{equation}
\begin{split}
\notag
&\int\limits_{S}|f(x)|^{p}\,d\mathfrak{m}_{0}(x)\le \int\limits_{S}\Bigl|f(x)-\fint\limits_{Q_{0,0}} F(y)\,dy\Bigr|^{p}\,d\mathfrak{m}_{0}(x)+\Bigl|\fint\limits_{Q_{0,0}} F(y)\,dy\Bigr|^{p}\\
&\le C \Bigl(\int\limits_{Q_{0,0}}\Bigl(\mathcal{M}_{1,1-\delta}^{1}[\|\nabla F\|](x)\Bigr)^{p}\,d\mathfrak{m}_{0}(x)\Bigr)^{\frac{1}{p}} + \Bigl(\int\limits_{Q_{0,0}}|F(y)|^{p}\,dy\Bigr)^{\frac{1}{p}} \le C \|F|W_{p}^{1}(3Q_{0,0})\| < +\infty.
\end{split}
\end{equation}
The proposition is proved.
\end{proof}


\section{Combinatorial and measure-theoretic tools}

In this section we built combinatorial and geometric measure theory foundations needed for our purposes. Based on the machinery developed in this section
we introduce new Calder\'on-type maximal functions in Section 4 and present a new construction of the extension operator in Section 5.
Throughout the whole section, \textit{we fix $d^{\ast} \in (0,n]$ and a closed set} $S \subset Q_{0,0}$
with $\mathcal{H}^{d^{\ast}}_{\infty}(S) > 0$.

\subsection{Admissible sequences of coverings}

Most of the definitions and results of this subsection are borrowed from our recent paper \cite{T2}, where the reader
can find all necessary details.

In the construction of the extension operator we will need to work with the family of all
$(d,\lambda)$-dyadically thick dyadic cubes. Such family gives a ``skeleton'' for the construction.
This motivates us to introduce the following concept
(recall notation $\mathcal{F}_{S}(d,\lambda)$ given in \eqref{eqq.thick_family}).

\begin{Def}
\label{Def.dlambdakeystone}
Let $\lambda \in (0,1]$ and let $d \in (0,n]$ be such that $\mathcal{H}^{d}_{\infty}(S)>0$.
We define \textit{the $(d,\lambda)$-keystone (for $S$) family} by letting
\begin{equation}
\mathcal{DF}(d,\lambda):=\mathcal{DF}_{S}(d,\lambda):=\mathcal{D}_{+} \bigcap \mathcal{F}_{S}(d,\lambda).
\end{equation}
The corresponding index set $\mathcal{A}(d,\lambda):=\mathcal{A}_{S}(d,\lambda) \subset \mathbb{N}_{0} \times \mathbb{Z}^{n}$
is called \textit{the $(d,\lambda)$-keystone (for $S$) index set}, i.e.,
\begin{equation}
\notag
\mathcal{DF}(d,\lambda)=\{Q_{\alpha}\}_{\alpha \in \mathcal{A}(d,\lambda)}.
\end{equation}
\end{Def}

Clearly, it is difficult to work with the whole family $\mathcal{DF}(d,\lambda)$.
Given $d \in (0,n]$ and $\lambda \in (0,1]$, we need a natural decomposition of the
$(d,\lambda)$-keystone family $\mathcal{DF}(d,\lambda)$
in analogy with a natural decomposition of the family $\mathcal{D}_{+}$ into subfamilies $\mathcal{D}_{k}$, $k \in \mathbb{N}_{0}$.
Based on Netrusov's ideas \cite{Net} such a decomposition was recently obtained in \cite{T2} and is given by the following theorem.

\begin{Th}
\label{Th.adm.sys.cover.}
Let $\lambda \in (0,1)$ and $d \in (0,n]$ be such that $\mathcal{H}^{d}_{\infty}(S)>0$.
Then there exists a unique sequence of families $\{\mathcal{Q}^{j}(d,\lambda)\}_{j \in \mathbb{N}}$, called the canonical decomposition of $\mathcal{DF}(d,\lambda)$,
satisfying the following conditions:

{\rm ($\mathcal{F}$1)} $\mathcal{DF}(d,\lambda)=\cup_{j \in \mathbb{N}}\mathcal{Q}^{j}(d,\lambda)$;

{\rm ($\mathcal{F}$2)} for each $j \in \mathbb{N}$, the family $\mathcal{Q}^{j}(d,\lambda)$ is a dyadic nonoverlapping $d$-almost covering of  $S$;

{\rm ($\mathcal{F}$3)} $\mathcal{Q}^{j}(d,\lambda) \succ \mathcal{Q}^{j+1}(d,\lambda)$ for every $j \in \mathbb{N}$.


\end{Th}

\begin{Remark}
\label{Rem.31111}
It follows immediately from conditions ($\mathcal{F}$1)--($\mathcal{F}$3) that if, for some cubes $\overline{Q} \in \mathcal{Q}^{j}(d,\lambda)$ and $\underline{Q} \in \mathcal{Q}^{j+1}(d,\lambda)$, there is a dyadic cube $Q \in \mathcal{D}_{+}$  such that
\begin{equation}
\notag
\underline{Q} \subset Q \subset \overline{Q} \quad \hbox{and} \quad l(Q) \in (l(\underline{Q}), l(\overline{Q})),
\end{equation}
then $Q \notin \mathcal{DF}(d,\lambda)$.
\end{Remark}
\hfill$\Box$

The following result reflects the fundamental combinatorial property of the families $\mathcal{Q}^{j}(d,\lambda)$, $j \in \mathbb{N}$.
Informally speaking, each family $\mathcal{Q}^{j}(d,\lambda)$ satisfies a some sort of \textit{Carleson packing condition.} We recall
notation \eqref{restriction of the family}.

\begin{Th}
\label{Th.33}
Let $\lambda_{1},\lambda_{2} \in (0,1)$ and let $d \in (0,n]$ be such that $\mathcal{H}^{d}_{\infty}(S)>0$.
Let $\{\mathcal{Q}^{j}(d,\lambda_{1})\}_{j \in \mathbb{N}}$ and $\{\mathcal{Q}^{j}(d,\lambda_{2})\}_{j \in \mathbb{N}}$ be the canonical decompositions
of  $\mathcal{DF}(d,\lambda_{1})$ and $\mathcal{DF}(d,\lambda_{2})$, respectively. Given a cube $Q \in \mathcal{D}_{+}$, let
\begin{equation}
\notag
j_{0}:=j(Q):=\min\{j \in \mathbb{N}_{0}:\{Q\} \succ \mathcal{Q}^{j}(d,\lambda_{1})|_{Q}\}.
\end{equation}
Then the following inequality
\begin{equation}
\label{eq.12}
\sum \{(l(Q'))^{d}: Q' \in \mathcal{C}\} \le
\begin{cases}
&2^{n-d}\frac{(l(Q))^{d}}{\lambda_{2}}, \quad Q \in \mathcal{DF}(d,1),\\
&\frac{(l(Q))^{d}}{\lambda_{2}}, \quad Q \notin \mathcal{DF}(d,1),
\end{cases}
\end{equation}
holds for any family $\mathcal{C} \subset \mathcal{DF}(d,\lambda_{2})$ such that:

{\rm (1)} $\operatorname{int}Q'' \cap \operatorname{int}Q' = \emptyset$ for any $Q',Q'' \in \mathcal{C}$ with $Q' \neq Q''$;

{\rm (2)} $\{Q\} \succeq \mathcal{C} \succeq \mathcal{Q}^{j_{0}}(d,\lambda_{1})|_{Q}$.
\end{Th}

Sometimes it will be convenient to work with projections of $(d,\lambda)$-keystone families to $k$th ``dyadic levels''.
Hence, we introduce the following concept.

\begin{Def}
\label{Def.proj}
Let $\lambda \in (0,1]$ and let $d \in (0,n]$ be such that $\mathcal{H}^{d}_{\infty}(S) > 0$.
For each $k \in \mathbb{N}_{0}$ we define the families
\begin{equation}
\begin{split}
&\mathcal{DF}_{k}(d,\lambda):=\mathcal{DF}_{S,k}(d,\lambda):=\mathcal{DF}(d,\lambda) \cap \mathcal{D}_{k};\\
&\widetilde{\mathcal{DF}}_{k}(d,\lambda):=\widetilde{\mathcal{DF}}_{S,k}(d,\lambda):=\{Q \in  \mathcal{D}_{k}: \hbox{ there is } Q' \in
\operatorname{n}(Q) \cap \mathcal{DF}_{k}(d,\lambda)\}.
\end{split}
\end{equation}
The corresponding index sets will be denoted by $\mathcal{A}_{k}(d,\lambda):=\mathcal{A}_{S,k}(d,\lambda)$
and $\widetilde{\mathcal{A}}_{k}(d,\lambda):=\widetilde{\mathcal{A}}_{S,k}(d,\lambda)$, respectively, i.e.,
\begin{equation}
\notag
\mathcal{DF}_{k}(d,\lambda):=\{Q_{k,m}| m \in \mathcal{A}_{k}(d,\lambda)\}, \quad
\widetilde{\mathcal{DF}}_{k}(d,\lambda):=\{Q_{k,\widetilde{m}}| \widetilde{m} \in \widetilde{\mathcal{A}}_{k}(d,\lambda)\}.
\end{equation}
\end{Def}

\begin{Def}
\label{Def.essent}
Let $\lambda \in (0,1)$ and let $d \in (0,n]$ be such that $\mathcal{H}^{d}_{\infty}(S) > 0$.
Let $\{\mathcal{Q}^{j}(d,\lambda)\}_{j \in \mathbb{N}}$ be the canonical decomposition of the family $\mathcal{DF}(d,\lambda)$.
We define \textit{the $(d,\lambda)$-essential part of $S$} by
\begin{equation}
\label{eq3.32}
\underline{S}(d,\lambda) : =\bigcap\limits_{j \in \mathbb{N}}\bigcup \{Q: Q \in \mathcal{Q}^{j}(d,\lambda)\}.
\end{equation}
\end{Def}

Sometimes it will be important to collect all dyadic cubes which are $(d,\lambda)$-thick with respect to a given set $S \subset \mathbb{R}^{n}$
and whose dilations contain a given cube $Q$.
From the intuitive point of view, such a family looks like a ``tower'' of cubes.

\begin{Def}
\label{bigkeystone}
Let $d \in (0,n]$ be such that $\mathcal{H}^{d}_{\infty}(S) > 0$. Let $\lambda \in (0,1]$ and $c \geq 1$.
Given a cube $Q \subset \mathbb{R}^{n}$, we define \textit{the $(d,\lambda,c)$-tower of} $Q$ by
\begin{equation}
\notag
T_{d,\lambda,c}(Q):=\{Q': Q' \in \mathcal{DF}(d,\lambda) \hbox{ and } Q \subset cQ'\}.
\end{equation}
In the case $c=1$, we  write $T_{d,\lambda}(Q)$ instead of $T_{d,\lambda,1}(Q)$, and we call the family
$T_{d,\lambda}(Q)$ simply \textit{the $(d,\lambda)$-tower} of $Q$.
\end{Def}

\begin{Remark}
\label{cover.y}
Note that Definition \ref{bigkeystone} admits the case
when a cube $Q$ has side length $l(Q)=0$, i.e., $Q=\{x\}$ for some $x \in \mathbb{R}^{n}$.
Hence, one can consider the $(d,\lambda,c)$-tower of $x$, which will be denoted by $T_{d,\lambda,c}(x)$.
\end{Remark}

\begin{Prop}
\label{Prop41''}
Let $d \in (0,n]$ be such that $\mathcal{H}^{d}_{\infty}(S) > 0$. Let $\lambda \in (0,1)$ and $c \geq 1$.
Then the following properties hold true:

{ \rm (1)} $\underline{S}(d,\lambda) \subset S$;

{ \rm (2)} $\mathcal{H}^{d}_{\infty}(S \setminus \underline{S}(d,\lambda))=0$;

{ \rm (3)} $\#T_{d,\lambda,c}(x)=+\infty$ for any $x \in \underline{S}(d,\lambda)$.

\end{Prop}

\begin{proof}
To prove (1), we note that the following inclusion (given $\delta > 0$, by $U_{\delta}(S)$ we denote the $\delta$-neighborhood of $S$)
\begin{equation}
\notag
\bigcup \{Q: Q \in \mathcal{Q}^{j}(d,\lambda)\} \subset U_{2^{-j+1}}(S)
\end{equation}
holds for any $j \in \mathbb{N}$ and take into account that the set $S$ is closed.

To prove (2) we combine the assertion ($\mathcal{F}$2) of Theorem \ref{Th.adm.sys.cover.} with \eqref{eq3.32}.

By \eqref{eq3.32} for each $x \in \underline{S}(d,\lambda)$
for every $j \in \mathbb{N}$ there is a cube $Q^{j} \in \mathcal{Q}^{j}_{S}(d,\lambda)$ containing $x$.
Combining this with assertion ($\mathcal{F}$3) of Theorem \ref{Th.adm.sys.cover.} we get (3).
\end{proof}

\subsection{Covering cubes}

As far as we know, the concepts introduced in this section have been never explicitly used in the literature.
Recall that the set $S$ was fixed at the beginning of the section.
Recall Definition \ref{Def.dlambdakeystone} and the notation $\mathcal{DF}(d,\lambda)$.

\begin{Def}
\label{cover.cube}
Let $d \in (0,n]$ be such that $\mathcal{H}^{d}_{\infty}(S) > 0$. Let $\lambda \in (0,1]$, $c \geq 1$.
Given  a cube $Q \in \mathcal{D}_{+}$, we say that a cube $\overline{Q} \in \mathcal{D}_{+}$
\textit{is a $(d,\lambda,c)$-covering for} $Q$
if the following conditions hold:

{\rm (\textbf{C}1)} $l(\overline{Q}) \geq l(Q)$ and $Q \subset c\overline{Q}$;

{\rm (\textbf{C}2)} $\overline{Q} \in \mathcal{DF}(d,\lambda)$;

{\rm (\textbf{C}3)} if $Q' \in \mathcal{D}_{+}$, $Q \subset Q'$ and $l(Q') \in (l(Q),l(\overline{Q}))$ then $Q' \notin \mathcal{DF}(d,\lambda)$.

Given a cube $Q \in \mathcal{D}_{+}$, we denote by $\mathcal{K}_{d,\lambda,c}(Q)$ the family (perhaps empty)
of all $(d,\lambda,c)$-covering for $Q$ cubes. By $\mathcal{K}_{d,\lambda,c}$ we also mean a set-valued mapping
which with each cube $Q \in \mathcal{D}_{+}$ associates the set (perhaps empty) $\mathcal{K}_{d,\lambda,c}(Q)$.
\end{Def}

We will also need some special selections of set-valued mappings $\mathcal{K}_{d,\lambda,c}$.

\begin{Def}
\label{Def.selector}
Let $d \in (0,n]$ be such that $\mathcal{H}^{d}_{\infty}(S) > 0$. Let $\lambda \in (0,1]$, $c \geq 1$.
\textit{A selection of} $\mathcal{K}_{d,\lambda,c}$ \textit{with a domain $\mathfrak{D} \subset \mathcal{D}_{+}$} is a mapping $\kappa_{d,\lambda,c}:\mathfrak{D} \to \mathcal{DF}(d,\lambda)$ such that:

{\rm (\textbf{Sel}1)} $\mathcal{K}_{d,\lambda,c}(Q) \neq \emptyset$ for all $Q \in \mathfrak{D}$;

{\rm (\textbf{Sel}2)} $\kappa_{d,\lambda,c}(Q) \in \mathcal{K}_{d,\lambda,c}(Q)$
for all $Q \in \mathfrak{D}$.
\end{Def}

\begin{Def}
\label{str.cover.cube}
Let $d \in (0,n]$ be such that $\overline{\lambda}  = \mathcal{H}^{d}_{\infty}(S) > 0$. Let $\lambda \in (0,1]$, $c \geq 1$.
Given a dyadic cube $Q \subset Q_{0,0}$ with $l(Q) < 1$, we say that $\operatorname{K}(Q) \in \mathcal{D}_{+}$
is \textit{a strongly $(d,\lambda)$-covering cube for} $Q$ if it is $(d,\lambda,1)$-covering and $l(\operatorname{K}(Q)) > l(Q)$.
\end{Def}

\begin{Remark}
\label{Remm.3.3'}
Note that in the sequel Definition \ref{str.cover.cube} will be used in the range $\lambda \in (0,\overline{\lambda}]$ because in this case
we have $Q_{0,0} \in \mathcal{DF}(d,\lambda)$ and hence, $\mathcal{K}_{d,\lambda,1}(Q) \neq \emptyset$ for any  $Q \in \mathcal{D}_{+}$
such that $Q \subset Q_{0,0}$. The requirement $l(Q) < 1$ allows one to find a unique cube $\operatorname{K}(Q) \in \mathcal{K}_{d,\lambda,1}(Q)$
with $l(\operatorname{K}(Q)) > l(Q)$.
\end{Remark}
\hfill$\Box$

\begin{Remark}
\label{Remm.3.3}
Let $d,\lambda,c$ be the same as in Definition \ref{cover.cube}. Given a cube $Q \in \mathcal{D}_{+}$, it is clear that there can exist several $(d,\lambda,c)$-covering
dyadic cubes for $Q$.
However, it is easy to see that there is a constant $C=C(n,c) > 0$ such that
\begin{equation}
\notag
\# \mathcal{K}_{d,\lambda,c}(Q) \cap \mathcal{D}_{k} \le C \quad \hbox{for every} \quad k \in \mathbb{N}_{0}.
\end{equation}
\end{Remark}
\hfill$\Box$

\begin{Def}
\label{Def.shadow}
Let $d \in (0,n]$ be such that $\mathcal{H}^{d}_{\infty}(S) > 0$. Let $\lambda \in (0,1]$ and $c \geq 1$.
Given a cube $\overline{Q} \in \mathcal{DF}(d,\lambda)$, we define the \textit{$(d,\lambda,c)$-shadow family of $\overline{Q}$} by letting
\begin{equation}
\notag
\mathcal{SH}_{d,\lambda,c}(\overline{Q}):=\{\underline{Q} \in \mathcal{DF}(d,\lambda): l(\underline{Q}) < l(\overline{Q}) \hbox{ and } \overline{Q} \in \mathcal{K}_{d,\lambda,c}(\underline{Q})\}.
\end{equation}
\end{Def}

The following proposition collects elementary properties of $(d,\lambda,c)$-shadow families of cubes. We recall notation
\eqref{restriction of the family}.

\begin{Prop}
\label{Prop.3.1}
Let $d \in (0,n]$ be such that $\mathcal{H}^{d}_{\infty}(S) > 0$. Let $\lambda \in (0,1)$ and $c \geq 1$.
Then, for each cube $\overline{Q} \in \mathcal{DF}(d,\lambda)$, the following properties hold:

{\rm (1)} $\mathcal{SH}_{d,\lambda,c}(\overline{Q}) \neq \emptyset$;

{\rm (2)} the family $\mathcal{SH}_{d,\lambda,c}(\overline{Q})$ is nonoverlapping.
\end{Prop}

\begin{proof}
To prove the first claim we fix a cube $\overline{Q} \in \mathcal{DF}(d,\lambda)$. Note that there is a number $j \in \mathbb{N}_{0}$
such that $\overline{Q} \in \mathcal{Q}^{j}(d,\lambda)$. Since $\mathcal{H}^{d}_{\infty}(\overline{Q}) \geq \lambda (l(\overline{Q}))^{d} > 0$ and the family
$\mathcal{Q}^{j+1}(d,\lambda)$ is a dyadic nonoverlapping $d$-almost covering of the set $S$, we
deduce that $\mathcal{Q}^{j+1}(d,\lambda)|_{\overline{Q}} \neq \emptyset$. By property $(\mathcal{F}3)$ of Theorem \ref{Th.adm.sys.cover.} we have
$\mathcal{Q}^{j+1}(d,\lambda)|_{\overline{Q}} \prec \overline{Q}$. Combining these observations
with Definitions \ref{cover.cube}, \ref{Def.shadow} and Remark \ref{Rem.31111} we obtain $\mathcal{Q}^{j+1}(d,\lambda)|_{\overline{Q}} \subset \mathcal{SH}_{d,\lambda,c}(\overline{Q})$ which proves (1).

To prove the second claim we fix a cube $\overline{Q} \in \mathcal{DF}(d,\lambda)$ and two different cubes $\underline{Q}_{1}, \underline{Q}_{2} \in \mathcal{SH}_{d,\lambda,c}(\overline{Q})$. Since the cubes $\underline{Q}_{1}$ and $\underline{Q}_{2}$
are dyadic, there are only two possible cases. In the first case, the cubes have disjoint interiors, in the second case one of them
is contained in the other one. We claim that the second case is never realised.
Indeed, assume the contrary. Without loss of generality we may assume that $\underline{Q}_{1} \subset \underline{Q}_{2}$. Since $\underline{Q}_{1}
\neq \underline{Q}_{2}$ we have $l(\underline{Q}_{1}) \le \frac{1}{2}l(\underline{Q}_{2})$.
Furthermore, by Definition \ref{Def.shadow} it follows that $l(\underline{Q}_{2}) \le \frac{1}{2}l(\overline{Q})$. Hence,
$l(\underline{Q}_{2}) \in (l(\underline{Q}_{1}),l(\overline{Q}))$. According to condition (3) of Definition \ref{cover.cube} this implies that $\underline{Q}_{2} \notin \mathcal{DF}(d,\lambda)$.
On the other hand, by Definition \ref{Def.shadow} the cubes $\underline{Q}_{1}$, $\underline{Q}_{2} \in \mathcal{DF}(d,\lambda)$.
This contradiction proves the claim. The proof is complete.
\end{proof}

Given a cube $Q \in \mathcal{D}_{k}$ with $k \in \mathbb{N}_{0}$, we set
\begin{equation}
\label{eq.Gamma}
\Gamma_{c}(Q):=\{Q' \in \mathcal{D}_{k}:Q' \cap cQ \neq \emptyset\}.
\end{equation}

\begin{Remark}
\label{Rem.card.gamma}
By Proposition \ref{Prop21''''}, if $Q' \in \Gamma_{c}(Q)$,
then $[c]Q \cap Q' \neq \emptyset$. Hence, $Q' \subset ([c]+2)Q$. Since different cubes in $\Gamma_{c}(Q)$ have disjoint
interiors and equal side lengths, we get
\begin{equation}
\notag
\#\Gamma_{c}(Q) \le \frac{\mathcal{L}^{n}(([c]+2)Q)}{\mathcal{L}^{n}(Q')} \le ([c]+2)^{n}.
\end{equation}
\end{Remark}
\hfill$\Box$

The following concept will be extremely useful in proving the direct trace theorem in Section 7.

\begin{Def}
\label{Def.iceberg}
 Let $d \in (0,n]$ be such that $\mathcal{H}^{d}_{\infty}(S) > 0$. Let $\lambda \in (0,1)$ and $c \geq 1$.
Given a cube $\overline{Q} \in \mathcal{DF}(d,\lambda)$,
we define the \textit{$(d,\lambda,c)$-iceberg} $\mathcal{IC}_{d,\lambda,c}(\overline{Q})$ of the cube $\overline{Q}$ as the family of all dyadic cubes $Q' \in \mathcal{D}_{+}$
satisfying the following conditions:

{\rm (\textbf{I}1)} $Q' \notin \mathcal{SH}_{d,\lambda,c}(\overline{Q})$;

{\rm (\textbf{I}2)} $l(Q') \le l(\overline{Q})$;

{\rm (\textbf{I}3)} there exists $\underline{Q'} \in \mathcal{SH}_{d,\lambda,c}(\overline{Q})$ such that $\underline{Q'} \subset Q'$.

\end{Def}

\begin{Remark}
\label{Rem.iceberg_intuition}
It follows immediately from Definition \ref{Def.iceberg} that if $Q \in \mathcal{IC}_{d,\lambda,c}(\overline{Q})$ and $Q' \supset Q$ is such that $l(Q') \le l(\overline{Q})$,
then $Q' \in \mathcal{IC}_{d,\lambda,c}(\overline{Q})$. Roughly speaking, in order to imagine $\mathcal{IC}_{d,\lambda,c}(\overline{Q})$, given $Q \in \mathcal{SH}_{d,\lambda,c}(\overline{Q})$, one
should built a tower composed of nested cubes whose side lengths grow up to the length $l(\overline{Q})$. The term ``iceberg'' was chosen due to the following reasons.
On the one hand, we will see below that the ``top of a given iceberg'' $\mathcal{IC}_{d,\lambda,c}(\overline{Q})$, i.e., the cube $\overline{Q}$, contains
useful information about the behavior of a given function $f:S \to \mathbb{R}$. On the other hand, cubes from the ``invisible part'' of $\mathcal{IC}_{d,\lambda,c}(\overline{Q})$, i.e.,
all cubes from $\mathcal{IC}_{d,\lambda,c}(\overline{Q})$ whose side lengths are smaller than $l(\overline{Q})$ do not contain some useful information
for the extension of a given function.
\end{Remark}
\hfill$\Box$

The following proposition reflects basic geometric properties of \textit{$(d,\lambda,c)$-icebergs}.

\begin{Prop}
\label{Propp.shad}
Let $d \in (0,n]$ be such that $\mathcal{H}^{d}_{\infty}(S) > 0$. Let $\lambda \in (0,1)$ and $c \geq 1$. Given a cube $\overline{Q} \in \mathcal{DF}(d,\lambda)$,
the following properties hold:

{\rm (1)} if $Q \in \mathcal{IC}_{d,\lambda,c}(\overline{Q})$ and $l(Q) < l(\overline{Q})$, then $Q \notin \mathcal{DF}(d,\lambda)$;

{\rm (2)} if $Q'' \in \mathcal{D}_{+}$ and $l(Q'')=l(\overline{Q})$, then $Q'' \in \Gamma_{c}(\overline{Q}) \cap \mathcal{IC}_{d,\lambda,c}(\overline{Q})$
if and only if $\mathcal{SH}_{d,\lambda,c}(\overline{Q})|_{Q''} \neq \emptyset$.

\end{Prop}

\begin{proof}
To prove the first claim, note that by conditions (\textbf{I}1), (\textbf{I}3) of Definition \ref{Def.iceberg}, given a cube $Q \in \mathcal{IC}_{d,\lambda,c}(\overline{Q})$, there is a cube $\underline{Q} \in \mathcal{SH}_{d,\lambda,c}(\overline{Q})$ such that $\underline{Q} \subset Q$ and $l(\underline{Q}) < l(Q)$. On the other hand, by Definition \ref{Def.shadow}
we have $\overline{Q} \in \mathcal{K}_{d,\lambda,c}(\underline{Q})$. Hence,
condition (\textbf{C}3) of Definition \ref{cover.cube} gives the claim.

If $Q'' \in \mathcal{D}_{+}$, $l(Q'')=l(\overline{Q})$ and $Q'' \in \Gamma_{c}(\overline{Q}) \cap \mathcal{IC}_{d,\lambda,c}(\overline{Q})$,
then $\mathcal{SH}_{d,\lambda,c}(\overline{Q})|_{Q''} \neq \emptyset$ by condition (\textbf{I}3) in Definition \ref{Def.iceberg}.
Conversely, suppose that $\mathcal{SH}_{d,\lambda,c}(\overline{Q})|_{Q''} \neq \emptyset$ for some $Q'' \in \mathcal{D}_{+}$ such that $l(Q'')=l(\overline{Q})$.
Clearly, conditions (\textbf{I}2) and (\textbf{I}3) of Definition \ref{Def.iceberg} are hold true with $Q'$ replaced by $Q''$. On the other hand, by Definition \ref{Def.shadow}, $l(Q') < l(Q'')$
for all $Q' \in \mathcal{SH}_{d,\lambda,c}(\overline{Q})|_{Q''}$ and hence condition (\textbf{I}1) holds true with $Q'$ replaced by $Q''$. This shows that
$Q'' \in \mathcal{IC}_{d,\lambda,c}(\overline{Q})$. Finally, by (\textbf{C}1) of Definition \ref{cover.cube} and Definition \ref{Def.shadow},the  condition
$\mathcal{SH}_{d,\lambda,c}(\overline{Q})|_{Q''} \neq \emptyset$ implies the existence of a cube $\underline{Q} \in \mathcal{DF}(d,\lambda)$
such that $\underline{Q} \subset Q''$ and $\underline{Q} \subset c\overline{Q}$, Hence, by \eqref{eq.Gamma} it follows that $Q'' \in \Gamma_{c}(\overline{Q})$.
This proves the second claim.
\end{proof}

Now we show that any $(d,\lambda,c)$-shadow family
satisfies \textit{a certain Carleson packing condition.} This will be a key tool in proving the main results
of Section 7.

\begin{Prop}
\label{Lm.shadow}
Let $d \in (0,n]$ be such that $\mathcal{H}^{d}_{\infty}(S) > 0$. Let $\lambda \in (0,1)$ and $c \geq 1$.
Then, for each $\overline{Q} \in \mathcal{DF}(d,\lambda)$ and any $Q \in \mathcal{IC}_{d,\lambda,c}(\overline{Q})$,
\begin{equation}
\label{eq3.15}
\sum \{(l(Q'))^{\widetilde{d}}: Q' \in \mathcal{SH}_{d,\lambda,c}(\overline{Q})|_{Q}\} \le \frac{2^{n-\widetilde{d}}}{\lambda} (l(Q))^{\widetilde{d}} \quad \text{for all} \quad
\widetilde{d} \in [d,n].
\end{equation}
Furthermore,
\begin{equation}
\label{eq3.15'''}
\sum\{(l(Q'))^{\widetilde{d}}:Q' \in \mathcal{SH}_{d,\lambda,c}(\overline{Q})\} \le ([c]+2)^{n}\frac{2^{n-\widetilde{d}}}{\lambda} (l(\overline{Q}))^{\widetilde{d}}
\quad \hbox{for all} \quad \widetilde{d} \in [d,n].
\end{equation}
\end{Prop}

\begin{proof}
We fix arbitrary cubes $\overline{Q} \in \mathcal{DF}(d,\lambda)$ and $Q \in \mathcal{IC}_{d,\lambda,c}(\overline{Q})$, and define
\begin{equation}
\notag
j_{0}:=\min\{j \in \mathbb{N}_{0}:\{Q\} \succ \mathcal{Q}^{j}(d,\lambda)|_{Q}\}.
\end{equation}
By Remark \ref{Rem.iceberg_intuition}, the first assertion of Proposition
\ref{Propp.shad}, and (\textbf{C}3) of Definition \ref{cover.cube}, if $Q' \in \mathcal{SH}_{d,\lambda,c}(\overline{Q})|_{Q}$ and $\overline{Q'} \in \mathcal{K}_{d,\lambda,1}(Q')$, then $\overline{Q'} \supset Q$, and,
furthermore, $l(\overline{Q'}) \geq l(\overline{Q})$. Hence, by Definition \ref{Def.shadow} we conclude that
$Q' \in \mathcal{Q}^{j_{0}}(d,\lambda)|_{Q}$.
As a result, an application of Theorem \ref{Th.33} with $\lambda_{1}=\lambda_{2}=\lambda$ gives
\begin{equation}
\notag
\begin{split}
&\sum \{(l(Q'))^{\widetilde{d}}| Q' \in \mathcal{SH}_{d,\lambda,c}(\overline{Q})|_{Q}\}\\
&\le \sum \{(\frac{1}{2})^{\widetilde{d}-d}(l(Q))^{\widetilde{d}-d}(l(Q'))^{d}| Q' \in \mathcal{Q}^{j_{0}}(d,\lambda)|_{Q}\} \le \frac{2^{n-\widetilde{d}}}{\lambda}(l(Q))^{\widetilde{d}}.
\end{split}
\end{equation}
This proves the first claim.

Now we prove \eqref{eq3.15'''}. In view of the second assertion of Proposition \ref{Propp.shad} it is sufficient to sum estimate \eqref{eq3.15} over all cubes $Q \in \Gamma_{c}(\overline{Q}) \cap \mathcal{IC}_{d,\lambda,c}(\overline{Q})$. Taking into account Remark \ref{Rem.card.gamma}
we get, for any $\widetilde{d} \in [d,n]$, the required inequality
\begin{equation}
\begin{split}
\notag
&\sum\{(l(Q'))^{\widetilde{d}}:Q' \in \mathcal{SH}_{d,\lambda,c}(\overline{Q})\}=\sum\limits_{Q \in \Gamma_{c}(\overline{Q}) \cap \mathcal{IC}_{d,\lambda,c}(\overline{Q})}
\sum \{(l(Q'))^{\widetilde{d}}: Q' \in \mathcal{SH}_{d,\lambda,c}(\overline{Q})|_{Q}\}
\\
& \le
\#\Gamma_{c}(\overline{Q}) \frac{2^{n-\widetilde{d}}}{\lambda }  (l(\overline{Q}))^{\widetilde{d}} \le ([c]+2)^{n}\frac{2^{n-\widetilde{d}}}{\lambda}(l(\overline{Q}))^{\widetilde{d}}.
\end{split}
\end{equation}
\end{proof}

\subsection{Whitney-type cavities and hollow cubes}

Due to the great importance of this subsection in our further analysis, we would like to describe informally the main ideas
underlying in the core of the concepts introduced below.
Recall that the set $S$ was fixed at the beginning of the section.
In addition, we fix in this subsection a number $d \in (0,n)$
such that $\mathcal{H}^{d}_{\infty}(S) > 0$ and a parameter $\lambda \in (0,1]$.

Recall that, given $\tau \in (0,1]$, a cube $Q_{l}(x)$ is said to be \textit{$(S,\tau)$-porous} if there is a point $y(x) \in Q_{l}(x)$ such that
\begin{equation}
\notag
Q_{\tau l}(y(x)) \subset Q_{l}(x) \setminus S.
\end{equation}

Recall also \cite{St}, Ch.6 that (since $S$ is closed, nonempty and $S \neq \mathbb{R}^{n}$) there is a nonempty nonoverlapping family $\mathcal{W}_{S} \subset \mathcal{D}$
(called \textit{the Whitney decomposition}, or, sometimes \textit{Whitney covering}) such that $\mathbb{R}^{n} \setminus S = \cup \{Q:Q \in \mathcal{W}_{S}\}$ and
\begin{equation}
\notag
\operatorname{dist}(Q,S) \le l(Q) \le 4\operatorname{dist}(Q,S).
\end{equation}
Cubes $Q$ from the family $\mathcal{W}_{S}$ are called Whitney cubes.

It should be noted that $(S,\tau)$-porous cubes and Whitney cubes are indispensable tools
in different topics dealing with extensions of functions \cite{Ihn,Shv2,Tri,TV}.
These two concepts are intimately related to each other. Indeed, given a cube $Q=Q_{l}(x) \in \mathcal{W}_{S}$ one can find a cube $\widetilde{Q}=\widetilde{Q}_{l}(\widetilde{x})$
whose center $\widetilde{x}$ is the metric projection of $x$ to $S$ such that  $Q \subset c\widetilde{Q}$ for some universal constant $c \geq 1$.
This proves that the cube $c\widetilde{Q}$ is $(S,\frac{1}{c})$-porous. Conversely, given an $(S,\tau)$-porous cube $\widetilde{Q}$, one
can find a Whitney cube $Q$ such that $Q \subset \widetilde{c}\widetilde{Q}$ for some universal constant $\widetilde{c} \geq 1$
and $l(Q) \approx l(\widetilde{Q})$.

Let us informally describe why the notions of Whitney cubes and porous cubes are so useful.
If either $p \in (1,n]$ and $S$ is regular enough or
$p > n$ and $S$ is arbitrary, one can effectively
absorb the information about the behavior of a given function $f: S \to \mathbb{R}$ from any porous with respect to $S$ cube $\widetilde{Q}$
and then, in some sense, transfer this information to the corresponding Whitney cube $Q$ with comparable side length. This gives the rough idea
of the classical Whitney extension operator. Unfortunately, in the case $1 < p \le n$ and without any additional regularity assumptions on the set $S$,
only few cubes $\widetilde{Q}$ with $\widetilde{Q} \cap S \neq \emptyset$ can be effectively used
for gathering information about a given function $f:S \to \mathbb{R}$.
One cannot hope that these cubes are porous in general. Instead of Whitney cubes and $(S,\tau)$-porous cubes,
we introduce the special cavities.


\begin{Def}
\label{Def.hollow}
Let $c \geq 1$ and $\varkappa \in \mathbb{N}$.
Given a cube $Q \in \mathcal{D}_{+}$,
we define \textit{the special cavity}
\begin{equation}
\label{eq.cavity}
\Omega_{c,\varkappa}(Q):= \Omega_{d,\lambda,c,\varkappa}(Q):=cQ \setminus \bigcup \{cQ': Q' \in \mathcal{DF}(d,\lambda) \hbox{ and } l(Q') \le 2^{-\varkappa}l(Q)\}.
\end{equation}
\end{Def}

Recall that at the beginning of this subsection we fixed $S$, $d \in (0,n)$ and $\lambda \in (0,1)$. The following result was established in \cite{T2} (see Theorem 4.2 therein).

\begin{ThA}
\label{Th.spec.cavit.}
For each $c \geq 1$, there exist constants $\tau=\tau(n,d) > 0$
and $\underline{\varkappa}=\underline{\varkappa}(n,d,\lambda,c) \in \mathbb{N}$ such that
\begin{equation}
\mathcal{L}^{n}(\Omega_{c,\varkappa}(Q)) \geq \tau (l(Q))^{n}
\end{equation}
for each cube $Q \in \mathcal{D}_{+} \setminus \mathcal{DF}(d,\lambda)$ and any $\varkappa > \underline{\varkappa}$.
\end{ThA}

The first nice property of special cavities is that they do not intersect ``too much''.

\begin{Prop}
\label{Th.intersect.cavities}
Let $c \geq 1$ and $\varkappa \in \mathbb{N}$.
Then there is a constant $C=C(n,c,\varkappa) > 0$ such that
\begin{equation}
\label{eq3.14'}
M(\{\Omega_{c,\varkappa}(Q):Q \in \mathcal{DF}(d,\lambda)\}) \le C.
\end{equation}
\end{Prop}

\begin{proof}
Fix dyadic cubes $Q_{1}=Q_{k_{1},m_{1}} \in \mathcal{DF}_{k_{1}}(d,\lambda)$  and $Q_{2}=Q_{k_{2},m_{2}} \in \mathcal{DF}_{k_{2}}(d,\lambda)$ with $k_{1},k_{2} \in \mathbb{N}_{0}$.
By \eqref{eq.cavity} we have
\begin{equation}
\label{eq3.16'}
\Omega_{c,\varkappa}(Q_{k,m}) \subset cQ_{k,m}.
\end{equation}
Hence,
\begin{equation}
\notag
\Omega_{c,\varkappa}(Q_{k_{1},m_{1}}) \cap \Omega_{c,\varkappa}(Q_{k_{2},m_{2}}) \neq \emptyset \quad \text{implies} \quad cQ_{k_{1},m_{1}} \cap cQ_{k_{2},m_{2}} \neq \emptyset.
\end{equation}
By \eqref{eq.cavity} this gives $|k_{1}-k_{2}| \le \varkappa$ provided that $\Omega_{c,\varkappa}(Q_{k_{1},m_{1}}) \cap \Omega_{c,\varkappa}(Q_{k_{2},m_{2}}) \neq \emptyset$.
Hence, if for some point $x \in \mathbb{R}^{n}$ there are $k(x) \in \mathbb{N}_{0}$ and $m(x) \in \mathbb{Z}^{n}$ such that
$x \in \Omega_{c,\varkappa}(Q_{k(x),m(x)})$, then we get
\begin{equation}
\label{eq3.15'}
\{j \in \mathbb{N}_{0}: \hbox{there is } m \in \mathbb{Z}^{n} \hbox{ s.t. }  \Omega_{c,\varkappa}(Q_{j,m}) \ni x \}
\subset [k(x)-\varkappa,k(x)+\varkappa] \cap \mathbb{N}_{0}.
\end{equation}
We use \eqref{eq3.16'}, \eqref{eq3.15'}, and apply Proposition \ref{Prop21}. This gives
\begin{equation}
\notag
\sum\limits_{Q \in \mathcal{D}_{+}}\chi_{\Omega_{c,\varkappa}(Q)}(x) \le \sum\limits_{k=k(x)-\varkappa}^{k(x)+\varkappa}
 \sum\limits_{m \in \mathbb{Z}^{n}}\chi_{cQ_{k,m}}(x) \le 2\varkappa ([c]+2)^{n} \quad \hbox{for all} \quad x \in \mathbb{R}^{n}.
\end{equation}
The proof is complete.
\end{proof}

Fix $c \geq 1$ and let $\underline{\varkappa}$ be the same as in Theorem \ref{Th.spec.cavit.}.
Consider the family
\begin{equation}
\label{eqq.3.14}
\begin{split}
&\mathcal{P}(c):=\mathcal{P}_{S}(d,\lambda,c)\\
&:=\Bigl\{Q \in \mathcal{DF}(d,\lambda): cQ \supset Q' \hbox{ for some } Q' \in \mathcal{D}_{+}\setminus \mathcal{DF}(d,\lambda) \hbox{ with } l(Q') \geq \frac{l(Q)}{4}\Bigr\}.
\end{split}
\end{equation}
We will see below that in the case of essentially
irregular sets $S \subset \mathbb{R}^{n}$ the role of the family of cubes $\mathcal{P}(c)$ for the extension
of functions is essentially the same as that of the family of all $(S,\tau)$-porous cubes in the case of sufficiently regular
sets $S \subset \mathbb{R}^{n}$.


\section{Calder\'on-type maximal functions and new function spaces}

In this section we introduce a far-reaching generalization of the Calder\'on-type maximal function introduced in \cite{TV}.
The latter generalizes the classical Calder\'on maximal function \cite{Cal}.

Let $\{\mathfrak{m}_{k}\}$ be an admissible sequence of measures on $\mathbb{R}^{n}$ and $f \in L_{1}(\{\mathfrak{m}_{k}\})$.
Given dyadic cubes $Q_{1} \in \mathcal{D}_{k_{1}}$, $Q_{2} \in \mathcal{D}_{k_{2}}$ with $k_{1},k_{2} \in \mathbb{N}_{0}$, we recall \eqref{eq.average} and put
\begin{equation}
\begin{split}
\label{eq.6.1'}
&\Phi_{f,\{\mathfrak{m}_{k}\}}(Q_{1},Q_{2}):=\Phi_{f,\{\mathfrak{m}_{k}\}}(Q_{2},Q_{1})\\
&:=\frac{1}{\min\{l(Q_{1}),l(Q_{2})\}} \fint\limits_{Q_{1}}\fint\limits_{Q_{2}}
|f(x)-f(y)|\,d\mathfrak{m}_{k_{1}}(x)\,d\mathfrak{m}_{k_{2}}(y).
\end{split}
\end{equation}
By \eqref{eq.average} and \eqref{eq.6.1'} it is easy to see that
\begin{equation}
\label{eqq.42}
\fint\limits_{Q_{1}}\Bigl|f(y)-\fint\limits_{Q_{2}}f(x)\,d\mathfrak{m}_{k_{2}}(x)\Bigr|\,d\mathfrak{m}_{k_{1}}(y)
\le \min\{l(Q_{1}),l(Q_{2})\}\Phi_{f,\{\mathfrak{m}_{k}\}}(Q_{1},Q_{2}).
\end{equation}

\begin{Prop}
\label{Prop.4.1}
Let $\{\mathfrak{m}_{k}\}$ be an admissible sequence of measures on $\mathbb{R}^{n}$ and $f \in L_{1}(\{\mathfrak{m}_{k}\})$. Then, the following inequality
\begin{equation}
\begin{split}
&\min\{l(Q_{1}),l(Q_{3})\}\Phi_{f,\{\mathfrak{m}_{k}\}}(Q_{1},Q_{3}) \\
&\le \min\{l(Q_{1}),l(Q_{2})\}\Phi_{f,\{\mathfrak{m}_{k}\}}(Q_{1},Q_{2}) + \min\{l(Q_{2}),l(Q_{3})\}\Phi_{f,\{\mathfrak{m}_{k}\}}(Q_{2},Q_{3})
\end{split}
\end{equation}
holds for any cubes $Q_{i} \in \mathcal{D}_{+}$ with $\mathfrak{m}_{k_{i}}(Q_{i}) \neq 0$, $i=1,2,3$.
\end{Prop}

\begin{proof}
Let $Q_{i} \in \mathcal{D}_{k_{i}}$, $i=1,2,3$ with $k_{1},k_{2},k_{3} \in \mathbb{N}_{0}$. Hence, by \eqref{eq.6.1'} and \eqref{eqq.42} we have
\begin{equation}
\begin{split}
&\min\{l(Q_{1}),l(Q_{3})\}\Phi_{f,\{\mathfrak{m}_{k}\}}(Q_{1},Q_{3})\\
&=\fint\limits_{Q_{1}}\fint\limits_{Q_{2}}
\Bigl|f(x)-\fint\limits_{Q_{3}}f(z)\,d\mathfrak{m}_{k_{3}}(z)+\fint\limits_{Q_{3}}f(z)\,d\mathfrak{m}_{k_{3}}(z)-f(y)\Bigr|\,d\mathfrak{m}_{k_{2}}(y)\,d\mathfrak{m}_{k_{1}}(x)\\
&\le \fint\limits_{Q_{1}}\Bigl|f(x) -\fint\limits_{Q_{3}}f(z)\,d\mathfrak{m}_{k_{3}}(z)\Bigr|\,d\mathfrak{m}_{k_{1}}(x) +
\fint\limits_{Q_{3}}\Bigl|f(y)-\fint\limits_{Q_{2}}f(y)\,d\mathfrak{m}_{k_{2}}(z)\Bigr|\,d\mathfrak{m}_{k_{3}}(y)\\
&\le \min\{l(Q_{1}),l(Q_{2})\}\Phi_{f,\{\mathfrak{m}_{k}\}}(Q_{1},Q_{2}) + \min\{l(Q_{2}),l(Q_{3})\}\Phi_{f,\{\mathfrak{m}_{k}\}}(Q_{2},Q_{3}).
\end{split}
\end{equation}
The proof is complete.
\end{proof}

Now we define the dyadic generalized Calder\'on-type maximal function. It will be an indispensable tool in our further analysis.
We recall Definitions \ref{Def.Frostman} and \ref{cover.cube}.

\begin{Def}
\label{Cal.max.function}
Let $d \in (0,n]$ and let $S \subset Q_{0,0}$ be a compact set with $\mathcal{H}^{d}_{\infty}(S) > 0$. Let $\lambda \in (0,1]$ and $c \geq 1$. Let
$\{\mathfrak{m}_{k}\} \in \mathfrak{M}^{d}(S)$ and $f \in L_{1}(\{\mathfrak{m}_{k}\})$. Given
a point $x \in \mathbb{R}^{n}$, we define
\begin{equation}
f^{\natural}_{\{\mathfrak{m}_{k}\},\lambda,c}(x):=\sup \Phi_{f,\{\mathfrak{m}_{k}\}}(\underline{Q},\overline{Q}),
\end{equation}
where the supremum is taken over all pairs of cubes
$\underline{Q},\overline{Q}$ satisfying the following conditions:

{ \rm (\textbf{f}1)} $x \in c\underline{Q}$ and $\underline{Q},\overline{Q} \in \mathcal{DF}_{S}(d,\lambda)$;

{ \rm (\textbf{f}2)} $0 < l(\underline{Q}) \le l(\overline{Q}) \le 1$;

{ \rm (\textbf{f}3)} $\overline{Q} \in \mathcal{K}_{d,\lambda,c}(\underline{Q})$.

If there are no pairs $\underline{Q},\overline{Q}$ satisfying conditions (\textbf{f}1)--(\textbf{f}3),
we put $f^{\natural}_{\{\mathfrak{m}_{k}\},\lambda,c}(x):=0$.
The mapping $x \to f^{\natural}_{\{\mathfrak{m}_{k}\},\lambda,c}(x)$ is called \textit{the dyadic generalized Calder\'on-type maximal function}.
\end{Def}

The dyadic generalized Calder\'on-type maximal functions have some straightforward monotonicity properties, which follow
immediately from Definition
\ref{Cal.max.function}. We omit an elementary proof.

\begin{Prop}
\label{Rem51}
Let $d \in (0,n]$ and let $S \subset Q_{0,0}$ be a compact set with $\mathcal{H}^{d}_{\infty}(S) > 0$. Let
$\{\mathfrak{m}_{k}\} \in \mathfrak{M}^{d}(S)$ and $f \in L_{1}(\{\mathfrak{m}_{k}\})$.
Then, given a parameter $\lambda \in (0,1]$ and a point $x \in \mathbb{R}^{n}$,
\begin{equation}
\notag
f^{\natural}_{\{\mathfrak{m}_{k}\},\lambda,c_{1}}(x) \le f^{\natural}_{\{\mathfrak{m}_{k}\},\lambda,c_{2}}(x) \quad \text{for any} \quad 1 \le c_{1} \le c_{2} < \infty.
\end{equation}

\end{Prop}

\begin{Def}
\label{Deff.5.4}
Let $d \in (0,n]$ and let $S \subset Q_{0,0}$ be a compact set with $\mathcal{H}^{d}_{\infty}(S) > 0$. Given
$f \in \mathfrak{B}(S) \cap L_{1}^{loc}(\{\mathfrak{m}_{k}\})$, we say that $x \in S$
is a $d$-\textit{regular point of} $f$ and write $x \in S_{f}(d)$ if, for each sequence $\{\widetilde{\mathfrak{m}}_{k}\} \in \mathfrak{M}^{d}(S)$,
\begin{equation}
\label{eq222}
\lim\limits_{k \to \infty}\max\limits_{\substack{Q \in T_{d,\lambda,c}(x) \cap \mathcal{D}_{k}}}\fint\limits_{Q}|f(x)-f(y)|\,d\widetilde{\mathfrak{m}}_{k}(y)=0
\quad \hbox{for any $c \geq 1$ and any $\lambda \in (0,1)$}.
\end{equation}
If, for some $k \in \mathbb{N}_{0}$, the set $T_{d,\lambda,c}(x) \cap \mathcal{D}_{k} = \emptyset$, the corresponding maximum is defined to be zero.
By $S_{f}(d)$ we denote \textit{the set of all $d$-regular points of $f$}.
\end{Def}

Now we are ready to introduce the function spaces, which will play the role of intermediate spaces between trace spaces of Sobolev spaces.

\begin{Def}
\label{def.x.trace}
Let $d^{\ast} \in (0,n]$ and let $S \subset Q_{0,0}$ be a compact set with $\lambda^{\ast}:=\mathcal{H}^{d^{\ast}}_{\infty}(S) > 0$.
Let $\lambda \in (0,1]$ and $c \geq 1$ be some fixed constants.
Let $d \in (0,d^{\ast}]$ and let $\{\mathfrak{m}_{k}\} \in \mathfrak{M}^{d}(S)$. Given $p \in (1,\infty)$, we say that $f \in L_{1}(\{\mathfrak{m}_{k}\})$ belongs to
$\widetilde{\operatorname{X}}^{d^{\ast}}_{p,d,\{\mathfrak{m}_{k}\}}(S)$ if the following conditions are satisfied:

{ \rm (1)} $\mathcal{H}^{d^{\ast}}(S \setminus S_{f}(d'))=0$ for all $d' \in [d,d^{\ast}]$;

{ \rm (2)} $\widetilde{\mathcal{N}}_{p,\lambda,\{\mathfrak{m}_{k}\},c}(f) < +\infty$, where we put
\begin{equation}
\label{eqq.tr.sp.norm}
\mathcal{N}_{p,\{\mathfrak{m}_{k}\},\lambda,c}(f):=\|f^{\natural}_{\{\mathfrak{m}_{k}\},\lambda,c}|L_{p}(\mathbb{R}^{n})\|, \quad
\widetilde{\mathcal{N}}_{p,\lambda,\{\mathfrak{m}_{k}\},c}(f):=\|f|L_{p}(\mathfrak{m}_{0})\|+\mathcal{N}_{p,\{\mathfrak{m}_{k}\},\lambda,c}(f).
\end{equation}
We define the space $\operatorname{X}^{d^{\ast}}_{p,d,\{\mathfrak{m}_{k}\}}(S)$ as the quotient space, i.e.,
\begin{equation}
\notag
\operatorname{X}^{d^{\ast}}_{p,d,\{\mathfrak{m}_{k}\}}(S):=\widetilde{\operatorname{X}}^{d^{\ast}}_{p,d,\{\mathfrak{m}_{k}\}}(S) / \{f \in \widetilde{\operatorname{X}}^{d^{\ast}}_{p,d,\{\mathfrak{m}_{k}\}}(S): \widetilde{\mathcal{N}}_{p,\{\mathfrak{m}_{k}\},\lambda,c}(f)=0\}.
\end{equation}
We equip the space $\operatorname{X}^{d^{\ast}}_{p,d,\{\mathfrak{m}_{k}\}}(S)$ with the norm given by the functional $\widetilde{\mathcal{N}}_{p,\{\mathfrak{m}_{k}\},\lambda,c}$, i.e.,
given a class of equivalent functions $[f] \in \operatorname{X}^{d^{\ast}}_{p,d,\{\mathfrak{m}_{k}\}}(S)$,
we put
\begin{equation}
\|[f]|\operatorname{X}^{d^{\ast}}_{p,d,\{\mathfrak{m}_{k}\}}(S)\|:=\widetilde{\mathcal{N}}_{p,\{\mathfrak{m}_{k}\},\lambda,c}(f).
\end{equation}
\end{Def}

\begin{Remark}
\label{Remm.4.2}
In what follows, we will identify the class of equivalent functions $[f] \in \operatorname{X}^{d^{\ast}}_{p,d,\{\mathfrak{m}_{k}\}}(S)$ with its arbitrary
representative $f$. It should be remarked that the structure of the class $[f]$ is not so straightforward.

Indeed, let
$S=S_{1} \cup S_{2}$ and $S_{1} \cap S_{2} = \emptyset$. Assume that $d \in (0,d^{\ast})$, $S_{1}$ is a closed Ahlfors--David $d$-regular set, and $S_{2}$ is a closed Ahlfors--David $d^{\ast}$-regular set.
Then, keeping in mind Example \ref{Ex.2.1}, it is easy to see that
changing a given $f: S \to \mathbb{R}$ on an $\mathcal{H}^{d^{\ast}}$-null set we can violate
condition (2) of Definition \ref{def.x.trace}.

On the other hand, keeping in mind Example \ref{Ex.2.1}, it is easy to see that if $S$ is an Ahlfors--David $n$-regular set and $d^{\ast} \in (0,n)$,
then changing a given $f: S \to \mathbb{R}$ on an $\mathfrak{m}_{0}$-null set we can violate condition (1) of the definition (because in this case $\mathfrak{m}_{0}$ coincides,
up to some constant, with the measure $\mathcal{L}^{n}\lfloor_{S}$).

In fact, the above definition of the space $\operatorname{X}^{d^{\ast}}_{p,d,\{\mathfrak{m}_{k}\}}(S)$ depends on the choice
of parameters $\lambda$, $c$. Typically, these parameters are always fixed and hence, we omit them from the corresponding notation.
\end{Remark}
\hfill$\Box$

\begin{Remark}
\label{Remm.4.3}
Let us verify that Definition \ref{def.x.trace} is correct, i.e., that $\operatorname{X}^{d^{\ast}}_{p,d,\{\mathfrak{m}_{k}\}}(S)$ is a normed linear space. First of all, we note that, for each $d' \in [d,d^{\ast}]$,
\begin{equation}
\notag
S_{f_{1}+f_{2}}(d') \supset S_{f_{1}}(d') \cap S_{f_{2}}(d').
\end{equation}
Hence, it remains to verify the triangle inequality. To this end, it suffices to verify that
\begin{equation}
\label{eqq.409}
\widetilde{\mathcal{N}}_{p,\{\mathfrak{m}_{k}\},\lambda,c}(f_{1}+f_{2}) \le \widetilde{\mathcal{N}}_{p,\{\mathfrak{m}_{k}\},\lambda,c}(f_{1})+
\widetilde{\mathcal{N}}_{p,\{\mathfrak{m}_{k}\},\lambda,c}(f_{2}).
\end{equation}
holds for any $f_{1},f_{2} \in L_{1}(\{\mathfrak{m}_{k}\})$. Indeed, by \eqref{eq.6.1'} and the triangle inequality, it is easy to see that,
for any cubes $Q_{1},Q_{2} \in \mathcal{D}_{+}$ and any $f_{1},f_{2} \in L_{1}(\{\mathfrak{m}_{k}\})$,
\begin{equation}
\notag
\Phi_{f_{1}+f_{2},\{\mathfrak{m}_{k}\}}(Q_{1},Q_{2}) \le \Phi_{f_{1},\{\mathfrak{m}_{k}\}}(Q_{1},Q_{2})+
\Phi_{f_{2},\{\mathfrak{m}_{k}\}}(Q_{1},Q_{2}).
\end{equation}
Using this inequality we get by Definition \ref{Cal.max.function}
\begin{equation}
\label{eqq.410}
(f_{1}+f_{2})^{\natural}_{\{\mathfrak{m}_{k}\},\lambda,c}(x) \le
(f_{1})^{\natural}_{\{\mathfrak{m}_{k}\},\lambda,c}(x)+(f_{2})^{\natural}_{\{\mathfrak{m}_{k}\},\lambda,c}(x), \quad x \in \mathbb{R}^{n}.
\end{equation}
Combining \eqref{eqq.410} with the triangle inequalities for the $L_{p}(\mathbb{R}^{n})$-norm and for the $L_{p}(\mathfrak{m}_{0})$-norm, respectively, we obtain
\eqref{eqq.409} and complete the proof.
\end{Remark}

\begin{Remark}
We prove in Section 8 that the spaces $\operatorname{X}^{d^{\ast}}_{p,d,\{\mathfrak{m}_{k}\}}(S)$ are complete provided that
$p \in (1,\infty)$, $d^{\ast} > n-p$, $d \in (n-p,d^{\ast}]$ and $\{\mathfrak{m}_{k}\} \in \mathfrak{M}^{d}(S)$.
\end{Remark}


\section{Extension operators}

The extension operator which will be constructed in this section is the \textit{most challenging part} of the present paper.
In all previously known studies concerned with extension problems for the first-order
Sobolev-spaces $W_{p}^{1}(\mathbb{R}^{n})$ \cite{Shv1}, \cite{Shv2}, \cite{TV} the authors basically used the classical Whitney extension operator
with minor modifications. Surprisingly, it perfectly worked.
However, in our case we should introduce completely new extension operators.

\textit{Throughout the section we fix the following data:}

{\rm(\textbf{D.5.1})} a parameter $d \in (0,n]$ and a compact set $S \subset Q_{0,0}$ with $\overline{\lambda}:=\mathcal{H}^{d}_{\infty}(S) > 0$;

{\rm(\textbf{D.5.2})} an arbitrary parameter $\lambda \in (0,\overline{\lambda})$ ;

{\rm(\textbf{D.5.3})} an arbitrary sequence of measures $\{\mathfrak{m}_{k}\} \in \mathfrak{M}^{d}(S)$.

Since the parameters $d,\lambda$ and the set $S$ are fixed, we will use the following simplified notation.
We set $\mathcal{DF}:=\mathcal{DF}_{S}(d,\lambda)$, and for each $k \in \mathbb{N}_{0}$, we set $\mathcal{DF}_{k}:=\mathcal{DF}_{S,k}(d,\lambda)$, $\widetilde{\mathcal{DF}}_{k}:=\widetilde{\mathcal{DF}}_{S,k}(d,\lambda)$,
$\mathcal{A}_{k}:=\mathcal{A}_{S,k}(d,\lambda)$, $\widetilde{\mathcal{A}}_{k}:=\widetilde{\mathcal{A}}_{S,k}(d,\lambda)$.
Furthermore, we recall Definition \ref{Def.essent} and set $\underline{S}:=\underline{S}(d,\lambda)$ for brevity.
Note that in accordance with our notation if a function $g: \mathbb{R}^{n} \to \mathbb{R}$ is differentiable at some point
$y \in \mathbb{R}^{n}$, then
\begin{equation}
\notag
\|\nabla g(y)\|:=\|\nabla g(y)\|_{\infty}:=\max\Bigl\{\Bigl|\frac{\partial g}{\partial x_{1}}(y)\Bigr|,...,\Bigl|\frac{\partial g}{\partial x_{n}}(y)\Bigr|\Bigr\}.
\end{equation}

\subsection{Construction of the extension operator}

First of all, we present a formal construction of the new extension operator and then
informally describe the driving ideas
of our construction.

Let a function $\widetilde{\psi}_{0} \in C_{0}^{\infty}(\mathbb{R})$ be such that:

{\rm ($i$)} $\chi_{[\frac{1}{10},\frac{9}{10}]} (\cdot)\le \widetilde{\psi}_{0}(\cdot) \le \chi_{[-\frac{1}{10},\frac{11}{10}]}(\cdot)$
and $\widetilde{\psi}_{0}(\cdot) > 0$ on the interval $(-1/10,11/10)$;

{\rm ($ii$)} $\sum\limits_{m \in \mathbb{Z}}\widetilde{\psi}_{0}(\cdot-m) \equiv 1$ on $\mathbb{R}$.

We set
\begin{equation}
\notag
C_{\widetilde{\psi}_{0}} := \max_{t \in \mathbb{R}}|\frac{d\widetilde{\psi}_{0}}{dt}(t)|.
\end{equation}
We define a function $\psi_{0} \in C_{0}^{\infty}(\mathbb{R}^{n})$ by
\begin{equation}
\notag
\psi_{0}(x):=\prod_{i=1}^{n}\widetilde{\psi}_{0}(x_{i}) \quad \hbox{for every} \quad x=(x_{1},...,x_{n}) \in \mathbb{R}^{n}
\end{equation}
and set $\psi_{k,m} (\cdot): = \psi_{0}(2^{k}(\cdot-m))$ for every $(k,m) \in \mathbb{N}_{0} \times \mathbb{Z}^{n}$. Clearly, the following properties hold:

{\rm ($\operatorname{i}$)} for each $k \in \mathbb{N}_{0}$ and any $m \in \mathbb{Z}^{n}$,
\begin{equation}
\label{eqq.5.1}
\chi_{\frac{4}{5}Q_{k,m}}(\cdot) \le \psi_{k,m}(\cdot) \le  \chi_{\frac{6}{5}Q_{k,m}}(\cdot) \quad \hbox{and} \quad \psi_{k,m}(\cdot) > 0 \hbox{ on }
\operatorname{int}(\frac{6}{5} Q_{k,m});
\end{equation}

{\rm ($\operatorname{ii}$)} for each $k \in \mathbb{N}_{0}$,
\begin{equation}
\label{eqq.5.2}
\sum\limits_{m \in \mathbb{Z}^{n}}\psi_{k,m}(\cdot) \equiv 1 \quad \hbox{on} \quad \mathbb{R}^{n}.
\end{equation}

{\rm ($\operatorname{iii}$)} for each $k \in \mathbb{N}_{0}$,
\begin{equation}
\label{eqq.5.3}
\sum\limits_{m \in \mathbb{Z}^{n}}\|\nabla \psi_{k,m}(y)\| \le C(n)C_{\widetilde{\psi}_{0}}2^{k} \quad \text{for all} \quad y \in \mathbb{R}^{n}.
\end{equation}

Recall notation $\operatorname{n}_{k}(m)$ (see Section 2) and  define
\begin{equation}
\label{eq4.1}
\mathfrak{c}_{k,m}:=\#(\operatorname{n}_{k}(m) \cap \mathcal{A}_{k}), \qquad (k,m) \in \mathbb{N}_{0} \times \mathbb{Z}^{n}.
\end{equation}

\begin{Remark}
\label{Remm.5.0}
By Definition \ref{Def.proj} and \eqref{eq4.1}, it readily follows that $\mathfrak{c}_{k,\widetilde{m}} \neq 0$,
for all $\widetilde{m} \in \widetilde{\mathcal{A}}_{k}$.
\end{Remark}

Given $f \in L_{1}(\{\mathfrak{m}_{k}\})$, we set
\begin{equation}
\label{eq2.16}
f_{k,m}:=
\begin{cases}
\frac{1}{\mathfrak{c}_{k,m}}\sum\limits_{m' \in \operatorname{n}_{k}(m) \cap \mathcal{A}_{k}}\fint\limits_{Q_{k,m'}}f(x)d\mathfrak{m}_{k}(x), \quad \mathfrak{c}_{k,m} \neq 0;\\
0, \quad \mathfrak{c}_{k,m}=0.
\end{cases}
\end{equation}

Now we are going to define inductively the special sequence of functions which play the role of an approximating sequence for the extension.
Note that since $S \subset Q_{0,0}$ and $\mathcal{H}^{d}_{\infty}(S) > 0$ we have
\begin{equation}
\notag
f_{0,0}=\fint\limits_{Q_{0,0}}f(x)\,d\mathfrak{m}_{0}(x).
\end{equation}
This observation justifies the following definition.

\begin{Def}
\label{Def.spec.app.seq}
Given $f \in L_{1}(\{\mathfrak{m}_{k}\})$, we define \textit{the special approximating sequence} $\{f_{k}\}:=\{f_{k}\}(\{\mathfrak{m}_{k}\})$ \textit{for} $f$ inductively.
At the zero step we set (note that by (\textbf{D.5.1}) we have $f_{0,m}=f_{0,0}$ for any cube $Q_{0,m} \subset 3Q_{0,0}$)
\begin{equation}
\notag
f_{0}(x):=  \sum_{Q_{0,m} \subset 3Q_{0,0}}\psi_{0,m}(x)f_{0,0} ,\quad x \in \mathbb{R}^{n}.
\end{equation}
Assume that, for some $k \in \mathbb{N}$, we have already constructed functions $f_{0},...,f_{k-1}$. We set
\begin{equation}
\label{eq2.18}
f_{k}(x):= f_{k-1}(x)+\sum \limits_{\widetilde{m} \in \widetilde{\mathcal{A}}_{k}} \psi_{k,\widetilde{m}}(x)(f_{k,\widetilde{m}}-f_{k-1}(x)), \quad x \in \mathbb{R}^{n}.
\end{equation}
\end{Def}

\begin{Remark}
\label{Rem.spec_app_seq_representatives}
Clearly, the sequence $\{f_{k}\}:=\{f_{k}\}(\{\mathfrak{m}_{k}\})$ is well defined, i.e., does not depend on the choice of representatives $\overline{f}$ of $f$.
Hence, in what follows, if $\overline{f} \in \mathfrak{B}(S)$ is such that the equivalence class $f$ of $\overline{f}$ belongs to $L_{1}(\{\mathfrak{m}_{k}\})$, then by
the special approximating sequence for $\overline{f}$ we always mean the special approximating sequence $\{f_{k}\}:=\{f_{k}\}(\{\mathfrak{m}_{k}\})$ for $f$.
\end{Remark}

\begin{Def}
\label{k.index}
For each $y \in \mathbb{R}^{n}$, we define \textit{the lower and the upper supporting index sets}, respectively, by letting
\begin{equation}
\notag
\begin{split}
&\underline{K}(y):=\underline{K}_{d,\lambda}(y):=\{k \in \mathbb{N}_{0}: \hbox{there exists} \quad m \in \mathcal{A}_{k} \quad \hbox{such that} \quad  y \in \frac{14}{5}Q_{k,m}\};\\
&\overline{K}(y):=\overline{K}_{d,\lambda}(y):=\{k \in \mathbb{N}_{0}: \hbox{there exists} \quad  \widetilde{m} \in \widetilde{\mathcal{A}}_{k} \quad \hbox{such that} \quad  \psi_{k,\widetilde{m}}(y) \neq 0\}.
\end{split}
\end{equation}
\end{Def}

\begin{Remark}
\label{Remm.5.1}
By Proposition \ref{Prop41''} we have
\begin{equation}
\label{eqq.cardinality}
\#\underline{K}(y)=\#\overline{K}(y)=+\infty \quad \text{for all} \quad y \in \underline{S}.
\end{equation}
Furthermore, by \eqref{eqq.5.1} and Definition \ref{k.index} we clearly
have
\begin{equation}
\notag
\underline{K}(y)=\emptyset \quad \Longleftrightarrow \quad y \notin \frac{14}{5}Q_{0,0} \quad \hbox{and}
\quad \overline{K}(y)=\emptyset \quad \Longleftrightarrow \quad y \notin \frac{16}{5}\operatorname{int}Q_{0,0}.
\end{equation}
\end{Remark}

The following proposition collects the basic properties of the lower and the upper supporting index sets.

\begin{Prop}
\label{Propp.5.1}
The lower and the upper supporting index sets have the following properties:

{\rm (1)} $0 \in \underline{K}(y)$ for each $y \in \frac{14}{5}Q_{0,0} \setminus \underline{S}$;

{\rm (2)} $0 \in \overline{K}(y)$ for each $y \in \frac{16}{5}Q_{0,0} \setminus \underline{S}$;

{\rm (3)} $\underline{K}(y) \subset \overline{K}(y)$ for each $y \in \mathbb{R}^{n}$;

{\rm (4)} $\#\underline{K}(y) \le \#\overline{K}(y) < +\infty$ for each $y \in \frac{16}{5}Q_{0,0} \setminus \underline{S}$.
\end{Prop}

\begin{proof}
Note that $Q_{0,0} \in \mathcal{DF}_{0}$ because $S \subset Q_{0,0}$ and
$\mathcal{H}^{d}_{\infty}(S)=\overline{\lambda} > \lambda$ according to our assumptions. Hence, by \eqref{eqq.5.1} and Definition \ref{k.index}
we get properties (1) and (2).

Property (3) follows directly from \eqref{eqq.5.1} and Definitions \ref{Def.proj}, \ref{k.index}.

To prove property (4) it is sufficient to show that $\#\overline{K}(y) < +\infty$  for every $y \in \frac{16}{5}Q_{0,0} \setminus \underline{S}$
because the inequality $\#\underline{K}(y) \le \#\overline{K}(y)$ follows directly from property (3) just proved.
By \eqref{eq3.32} it follows that, for each $y \in \frac{16}{5}Q_{0,0} \setminus \underline{S}$, there exists $j(y) \in \mathbb{N}$
such that
\begin{equation}
\notag
y \notin \bigcup \{Q:Q \in \mathcal{Q}^{j(y)}(d,\lambda)\}.
\end{equation}
Combining this fact with property $(\mathcal{F}3)$ of Theorem \ref{Th.adm.sys.cover.} we deduce that in fact
\begin{equation}
\notag
y \notin \bigcup\limits_{j \geq j(y)} \bigcup \{Q:Q \in \mathcal{Q}^{j}(d,\lambda)\}.
\end{equation}
This gives the required claim.
\end{proof}

Proposition \ref{Propp.5.1} justifies the following concept.

\begin{Def}
\label{Deff.5.3}
We define \textit{the lower supporting index} $\underline{k}(y)$ and \textit{the upper supporting index}  $\overline{k}(y)$, respectively, by letting
\begin{equation}
\label{eqq.5.11}
\begin{split}
&\underline{k}(y):=\max\{k: k \in \underline{K}(y)\}  \quad \hbox{for each} \quad  y \in \frac{14}{5}Q_{0,0} \setminus \underline{S};\\
&\overline{k}(y):=\max\{k:k \in \overline{K}(y)\} \quad \hbox{for each} \quad y \in \frac{16}{5}Q_{0,0} \setminus \underline{S}.
\end{split}
\end{equation}
\end{Def}

\begin{Prop}
\label{Rem1}
Let $f \in L_{1}(\mathfrak{m}_{k})$ and let $\{f_{k}\}$ be the special approximating sequence for $f$.
Then, for each point $y \in \frac{14}{5}Q_{0,0} \setminus \underline{S}$, the following properties hold true:

{\rm (1)} for every $k \in \underline{K}(y)$,
\begin{equation}
\label{eq47'''}
\mathfrak{c}_{k,\widetilde{m}} \neq 0 \quad \hbox{if} \quad y \in Q_{k,\widetilde{m}};
\end{equation}

{\rm (2)} for every $k \in \underline{K}(y)$,
\begin{equation}
\label{eq.resol.unity}
\sum_{\widetilde{m} \in \widetilde{\mathcal{A}}_{k}}\psi_{k,\widetilde{m}}(y)=1;
\end{equation}

{\rm (3)} $\underline{k}(y) \le \overline{k}(y) < +\infty$ and
\begin{equation}
\label{eq48'''}
f_{k}(y)=f_{\overline{k}(y)}(y)  \quad \hbox{for any} \quad k > \overline{k}(y).
\end{equation}
\end{Prop}

\begin{proof}
By Definition \ref{k.index} if  $k \in \underline{K}(y)$, then $y \in \frac{14}{5}Q_{k,m}$ for some $m \in \mathcal{A}_{k}$.
Hence, if in addition $y \in Q_{k,\widetilde{m}}$, then $\widetilde{m} \in \widetilde{\mathcal{A}}_{k} \cap \operatorname{n}_{k}(m)$. Consequently, Definition \ref{Def.proj} and \eqref{eq4.1}
implies that $\mathfrak{c}_{k,\widetilde{m}} \neq 0$.

To prove \eqref{eq.resol.unity} it is sufficient to combine Definitions \ref{Def.proj}, \ref{k.index} with \eqref{eqq.5.1} and \eqref{eqq.5.2}.

The inequality $\underline{k}(y) \le \overline{k}(y)$ is an immediate consequence of property (3) of Proposition \ref{Propp.5.1}. The inequality $\overline{k}(y) < +\infty$
follows from property (4) of Proposition \ref{Propp.5.1}.
Finally, by Definition \ref{Deff.5.3} we have $\psi_{k,\widetilde{m}}(y)=0$ for each $k > \overline{k}(y)$ and any $\widetilde{m} \in \widetilde{\mathcal{A}}_{k}$. Now property (3) follows from \eqref{eq2.18}.
\end{proof}

Now having at our disposal Proposition \ref{Rem1} and property (1)
of Proposition \ref{Prop41''} we can built the desirable extension operator.

\begin{Def}
\label{extension_operator}
Given $f \in \mathfrak{B}(S) \cap L_{1}(\{\mathfrak{m}_{k}\})$, let  $\{f_{k}\}$ be the special approximating for $f$.
We define
\begin{equation}
\label{eq.exst.op}
\operatorname{Ext}_{S,\{\mathfrak{m}_{k}\},\lambda}(f)(x):=\operatorname{Ext}(f)(x):=
\begin{cases}
f(x), \quad x \in S;\\
\lim\limits_{k \to \infty}f_{k}(x), \quad x \in \mathbb{R}^{n} \setminus S.
\end{cases}
\end{equation}
\end{Def}


The following obvious observation is an immediate consequence of Definition \ref{extension_operator}.

\begin{Prop}
\label{Propp.5.2}
The operator $\operatorname{Ext}$ defined by \eqref{eq.exst.op} is a linear mapping from
$\mathfrak{B}(S) \cap L_{1}(\{\mathfrak{m}_{k}\})$ into $\mathfrak{B}(\mathbb{R}^{n})$.
\end{Prop}

Since the construction of our extension operator is quite tricky, we would like to
describe the driving ideas informally.

\textit{The first idea} consists in using only cubes from the family $\mathcal{DF}$
to extract some useful information about the behavior of a given function $f \in L_{1}(\{\mathfrak{m}_{k}\})$.
Informally speaking, the family $\mathcal{DF}$ gives some sort of a skeleton for the extension operator.
Indeed, Theorem \ref{Th4.3} allows us to hope that averaging over these cubes with respect to measures $\mathfrak{m}_{k}$, $k \in \mathbb{N}_{0}$
is necessary in constructions of almost optimal Sobolev extensions.

\textit{The second idea looks} a little bit technical. Nevertheless, it is quite important.
Recall (see Subsection 3.3) that by $\mathcal{W}_{S}$ we denote the Whitney decomposition of $\mathbb{R}^{n} \setminus S$.
In the majority of the available investigations the family $\mathcal{W}_{S}$ plays a crucial role in constructions
of extension operators. It allowed one in some sense to transfer the information about a given function $f:S \to \mathbb{R}$ from $S$ into $\mathbb{R}^{n} \setminus S$.
In contrast, our approach uses the family $\cup_{k \in \mathbb{N}_{0}}\widetilde{\mathcal{DF}}_{k}\setminus \mathcal{DF}$ .
In the case when either $S$ is regular enough or $p > n$, then this innovation gives nothing new in comparison
with the classical approach of H. Whitney.
If $S$ is highly irregular and $p \in (1,n]$, the modification becomes essential. It helps
to avoid the study of the complicated combinatorial structure of the family $\mathcal{W}_{S}$. Informally
speaking, it is difficult to built a ``nice tree'' associated
with the family $\mathcal{W}_{S}$.

\textit{The third idea} involves an additional averaging over neighboring cubes in \eqref{eq2.16}.
This simple trick together with the use of families $\mathcal{DF}_{k}$, $k \in \mathbb{N}_{0}$ helps one to avoid large derivatives.
Roughly speaking, given $f \in L_{1}(\{\mathfrak{m}_{k}\})$, pointwise estimates of $\operatorname{Ext}(f)$ from above will contain only terms
like
\begin{equation}
\notag
\Bigl|\fint\limits_{Q_{k,m}}f(y)\,d\mathfrak{m}_{k}(y)-\fint\limits_{Q_{l,m'}}f(y)\,d\mathfrak{m}_{l}(y)\Bigr|,
\end{equation}
where $Q_{l,m'} \in \mathcal{K}_{d,\lambda,c}(Q_{k,m})$ for some $c > 1$. This is crucial in proving the optimality of the extension.

We should note that the roots of the above ideas go back to the paper of V.~Rychkov \cite{Ry}. However, in this paper only $d$-thick sets were considered.
In the case when $S$ is $d$-thick, the analysis of the pointwise behavior
of special approximating sequences is much more simple and transparent.

\subsection{Fine properties of the special approximating sequence}

In this subsection, given  $f \in L_{1}(\{\mathfrak{m}_{k}\})$, we
investigate a pointwise behavior of $\operatorname{Ext}_{S,d,\lambda}(f)$.

We start with a technical observation, which will be commonly used.

\begin{Prop}
\label{Prop4.1'''}
Let $f \in L_{1}(\{\mathfrak{m}_{k}\})$ and $c \in \mathbb{R}$. Then, for each $k \in \mathbb{N}_{0}$ and any $\widetilde{m}^{k} \in \widetilde{\mathcal{A}}_{k}$,
\begin{equation}
\label{eqq.5.12}
|f_{k,\widetilde{m}^{k}}-c| \le \max\limits_{m^{k} \in \operatorname{n}_{k}(\widetilde{m}^{k})\cap \mathcal{A}_{k}}\fint\limits_{Q_{k,m^{k}}}|f(x)-c|d\mathfrak{m}_{k}(x).
\end{equation}

Furthermore, for each $k,j \in \mathbb{N}_{0}$ and any $\widetilde{m}^{k} \in \widetilde{\mathcal{A}}_{k}$, $\widetilde{m}^{j} \in \widetilde{\mathcal{A}}_{j}$,
\begin{equation}
\label{eqq.5.13}
\begin{split}
&|f_{k,\widetilde{m}^{k}}-f_{j,\widetilde{m}^{j}}| \le \max \fint\limits_{Q_{k,m^{k}}}\fint\limits_{Q_{j,m^{j}}}|f(x)-f(y)|\,d\mathfrak{m}_{j}(y)d\mathfrak{m}_{k}(x)
\end{split}
\end{equation}
where the maximum is taken over all $m^{k} \in \operatorname{n}_{k}(\widetilde{m}^{k}) \cap \mathcal{A}_{k}$ and all $m^{j} \in \operatorname{n}_{j}(\widetilde{m}^{j}) \cap \mathcal{A}_{j}$.
\end{Prop}

\begin{proof}
To prove \eqref{eqq.5.12} we fix $k \in \mathbb{N}_{0}$ and $\widetilde{m}^{k} \in \widetilde{\mathcal{A}}_{k}$. By Remark \ref{Remm.5.0}, we have $\mathfrak{c}_{k,\widetilde{m}^{k}} > 0$.
Hence,
\begin{equation}
\label{eqq.5.14}
1=\frac{1}{\mathfrak{c}_{k,\widetilde{m}^{k}}}
\sum\limits_{m^{k} \in \operatorname{n}_{k}(\widetilde{m}^{k})\cap \mathcal{A}_{k}}1.
\end{equation}
Using this observation and \eqref{eq2.16}, we deduce the required estimate
\begin{equation}
\begin{split}
&|f_{k,\widetilde{m}^{k}}-c| =
\Bigl|\frac{1}{\mathfrak{c}_{k,\widetilde{m}^{k}}}\sum\limits_{m^{k} \in \operatorname{n}_{k}(\widetilde{m}^{k})\cap \mathcal{A}_{k}}
\fint\limits_{Q_{k,m^{k}}}f(x)d\mathfrak{m}_{k}(x)-\frac{1}{\mathfrak{c}_{k,\widetilde{m}^{k}}}\sum\limits_{m^{k} \in \operatorname{n}_{k}(\widetilde{m}^{k})\cap \mathcal{A}_{k}}c\Bigr|\\
&\le \max\limits_{m^{k} \in \operatorname{n}_{k}(\widetilde{m}^{k})\cap \mathcal{A}_{k}}\Bigl|\fint\limits_{Q_{k,m^{k}}}f(x)d\mathfrak{m}_{k}(x)-c\Bigr| \le \max\limits_{m^{k} \in \operatorname{n}_{k}(\widetilde{m}^{k})\cap \mathcal{A}_{k}}\fint\limits_{Q_{k,m^{k}}}|f(x)-c|d\mathfrak{m}_{k}(x).
\end{split}
\end{equation}

Now we fix arbitrary $k,j \in \mathbb{N}_{0}$ and $\widetilde{m}^{k} \in \widetilde{\mathcal{A}}_{k}$, $\widetilde{m}^{j} \in \widetilde{\mathcal{A}}_{j}$.
We firstly apply \eqref{eqq.5.12} with $c=f_{j,\widetilde{m}^{j}}$ and then, for $\mathfrak{m}_{k}$-a.e. $y \in S$, we apply
\eqref{eqq.5.12} with $c=f(y)$.
This gives
\begin{equation}
\label{eq44}
\begin{split}
&|f_{k,\widetilde{m}^{k}}-f_{j,\widetilde{m}^{j}}| \le \max\limits_{m^{k} \in \operatorname{n}_{k}(\widetilde{m}^{k})\cap \mathcal{A}_{k}}\fint\limits_{Q_{k,m^{k}}}|f(y)-f_{j,\widetilde{m}^{j}}|d\mathfrak{m}_{k}(y)\\
&\le \max\limits_{m^{k} \in \operatorname{n}_{k}(\widetilde{m}^{k})\cap \mathcal{A}_{k}}\fint\limits_{Q_{k,m^{k}}}
\Bigl(\max\limits_{m^{j} \in \operatorname{n}_{j}(\widetilde{m}^{j})\cap \mathcal{A}_{j}}\fint\limits_{Q_{j,m^{j}}}|f(y)-f(x)|\,d\mathfrak{m}_{j}(x)\Bigr)\,d\mathfrak{m}_{k}(y)\\
&\le \max \fint\limits_{Q_{k,m^{k}}}\fint\limits_{Q_{j,m^{j}}}|f(y)-f(x)|\,d\mathfrak{m}_{j}(x)d\mathfrak{m}_{k}(y),
\end{split}
\end{equation}
where the maximum is taken over all $m^{k} \in \operatorname{n}_{k}(\widetilde{m}^{k}) \cap \mathcal{A}_{k}$ and all $m^{j} \in \operatorname{n}_{j}(\widetilde{m}^{j}) \cap \mathcal{A}_{j}$.

The proof is complete.
\end{proof}

Now we introduce the keystone tool of this section.
More precisely, given $f \in L_{1}(\{\mathfrak{m}_{k}\})$,
the inductive definition of the sequence $\{f_{k}\}_{k\in\mathbb{N}_{0}}$, as given in \eqref{eq2.18}, is not so useful for practical computations.
In view of this, we present an explicit formula for functions $f_{k}$, $k \in \mathbb{N}_{0}$.

\begin{Lm}
\label{Lm.function}
Let $f \in L_{1}(\mathfrak{m}_{k})$ and let $\{f_{k}\}$ be the special approximating sequence for $f$. Then, for every $i,k \in \mathbb{N}_{0}$ with
$k > i$,
\begin{equation}
\label{eq2.19}
\begin{split}
&f_{k}(x)=f_{i}(x)+\sum\limits_{j=i+1}^{k}S_{i,k}^{j}(x), \quad x \in \mathbb{R}^{n},
\end{split}
\end{equation}
where, for each  $j \in \{i+1,...,k\}$ and every $x \in \mathbb{R}^{n}$, we set
\begin{equation}
\label{eq221}
\begin{split}
&S^{j}_{i,k}(x):=
\begin{cases}
\sum\limits_{\widetilde{m}^{j} \in \widetilde{\mathcal{A}}_{j}}\psi_{j,\widetilde{m}^{j}}(x)\Bigl(
\prod\limits_{r=j+1}^{k}\Bigl(\sum\limits_{\widetilde{m}^{r} \in \mathbb{Z}^{n} \setminus \widetilde{\mathcal{A}}_{r}}
\psi_{r,\widetilde{m}^{r}}(x)\Bigr)\Bigr)(f_{j,\widetilde{m}^{j}}-f_{i}(x)), \quad j \in \{i+1,...,k-1\};\\
\sum\limits_{\widetilde{m}^{k} \in \widetilde{\mathcal{A}}_{k}}\psi_{k,\widetilde{m}^{k}}(x)(f_{k,\widetilde{m}^{k}}-f_{i}(x)), \quad j=k.
\end{cases}
\end{split}
\end{equation}
The corresponding product in \eqref{eq221} disappears for $k < i+2$.
\end{Lm}

\begin{proof}
Given a fixed $i \in \mathbb{N}_{0}$, we prove \eqref{eq2.19} by induction.

\textit{The base.} For $k=i+1$, the statement is obvious in view of our construction.

\textit{The induction step.}
Suppose that \eqref{eq2.19} is proved for some $k=l \in \mathbb{N}$, $l > i$. We show that \eqref{eq2.19}
holds true with $k=l+1$.

Indeed, first of all, we note that by \eqref{eq221}, for each $j=i+1,...,l$, we have
\begin{equation}
\begin{split}
\label{eq222'}
&\Bigl(1-\sum\limits_{\widetilde{m}^{l+1} \in \widetilde{\mathcal{A}}_{l+1}}\psi_{l+1,\widetilde{m}^{l+1}}(x)\Bigr)S^{j}_{i,l}(x)\\
&=\sum\limits_{\widetilde{m}^{l+1} \in \mathbb{Z}^{n} \setminus \widetilde{\mathcal{A}}_{l+1}}\psi_{l+1,\widetilde{m}^{l+1}}(x)S^{j}_{i,l}(x)
=S^{j}_{i,l+1}(x), \quad x \in \mathbb{R}^{n}.
\end{split}
\end{equation}
On the other hand, by \eqref{eq221}
\begin{equation}
\label{eq415}
S^{l+1}_{i,l+1}(x)=\sum\limits_{\widetilde{m}^{l+1} \in \widetilde{\mathcal{A}}_{l+1}}\psi_{l+1,\widetilde{m}^{l+1}}(x)(f_{l+1,\widetilde{m}^{l+1}}-f_{i}(x)), \quad x \in \mathbb{R}^{n}.
\end{equation}

Now we plug \eqref{eq2.19} with $k=l$ into \eqref{eq2.18} and use \eqref{eq222'}, \eqref{eq415}. This gives the required identity
\begin{equation}
\begin{split}
&f_{l+1}(x)=f_{i}(x) + \sum\limits_{j=i+1}^{l}S_{i,l}^{j}(x)+\sum\limits_{\widetilde{m}^{l+1} \in \widetilde{\mathcal{A}}_{l+1}}\psi_{l+1,\widetilde{m}^{l+1}}(x)
\Bigl(f_{l+1,\widetilde{m}^{l+1}}-f_{i}(x) - \sum\limits_{j=i+1}^{l}S_{i,l}^{j}(x)\Bigr)\\
&=f_{i}(x)+\sum\limits_{j=i+1}^{l}\Bigl(1-\sum\limits_{\widetilde{m}^{l+1} \in \widetilde{\mathcal{A}}_{l+1}}\psi_{l+1,\widetilde{m}^{l+1}}(x)\Bigr)S^{j}_{i,l}(x)+\\
&+\sum\limits_{\widetilde{m}^{l+1} \in \widetilde{\mathcal{A}}_{l+1}}\psi_{l+1,\widetilde{m}^{l+1}}(x)(f_{l+1,\widetilde{m}^{l+1}}-f_{i}(x))
=f_{i}(x)+\sum\limits_{j=i+1}^{l+1}S^{j}_{i,l+1}(x), \quad x \in \mathbb{R}^{n}.\\
\end{split}
\end{equation}

The lemma is proved.

\end{proof}

\begin{Remark}
\label{Remm.5.2}
By Lemma \ref{Lm.function} applied with $i=0$ and \eqref{eq.resol.unity} we have
\begin{equation}
\label{43''}
f_{k}(y)=\sum\limits_{\widetilde{m} \in \widetilde{\mathcal{A}}_{k}}\psi_{k,\widetilde{m}}(y)
f_{k,\widetilde{m}} \quad \hbox{for each} \quad y \in \frac{14}{5}Q_{0,0} \quad \hbox{for every}  \quad k \in \underline{K}(y).
\end{equation}
This nice reproducing formula will simplify some intermediate computations in the proof of the forthcoming assertions.
\end{Remark}
\hfill$\Box$

By Proposition \ref{Propp.5.1}, we can pass to the limit in $f_{k}(x)$ for every $x \in \mathbb{R}^{n} \setminus S$.
This is not the case for an arbitrary $x \in S$. However, if a given function
$f:S \to \mathbb{R}$ is sufficiently regular we can extract
convergent subsequences $\{f_{k_{s}}(x)\}$
for appropriate points $x \in \mathbb{R}^{n}$.

\begin{Lm}
\label{Lm.function2}
Given $f \in \mathfrak{B}(S) \cap L_{1}(\{\mathfrak{m}_{k}\})$, for every point $x \in \underline{S} \cap S_{f}(d)$, there exists
an increasing sequence $\{k_{s}\}=\{k_{s}(x)\}_{s \in \mathbb{N}_{0}} \subset \mathbb{N}_{0}$ such that
\begin{equation}
\notag
\lim_{s \to \infty}f_{k_{s}}(x)=f(x).
\end{equation}
\end{Lm}

\begin{proof}
Fix a point $x \in \underline{S} \cap S_{f}(d)$.
By Remark \ref{Remm.5.1}, the set $\underline{K}(x)$ is infinite and can be written as a strictly
increasing sequence $\{k_{s}\}=\{k_{s}(x)\}_{s \in \mathbb{N}_{0}} \subset \mathbb{N}_{0}$.
Hence, by \eqref{eq2.16}, \eqref{43''} and \eqref{eq222}, we get
\begin{equation}
\begin{split}
&|f(x)-f_{k_{s}}(x)| \le \sum\limits_{\substack{m \in \mathcal{A}_{k_{s}}\\ x \in \frac{14}{5}Q_{k_{s},m}}}\fint\limits_{Q_{k_{s},m}}|f(x)-f(y)|\,d\mathfrak{m}_{k_{s}}(y)\\
&\le C \max\limits_{\substack{m \in \mathcal{A}_{k_{s}}\\x \in \frac{14}{5}Q_{k_{s},m}}}\fint\limits_{Q_{k_{s},m}}|f(x)-f(y)|\,d\mathfrak{m}_{k_{s}}(y) \to 0, \quad s \to \infty.
\end{split}
\end{equation}
The lemma is proved.
\end{proof}

The following technical assertion will be important in proving the main results of this section.

\begin{Prop}
\label{Prop43}
Let $s \in \mathbb{N}$ and $\{k_{i}\}_{i=1}^{s} \subset \mathbb{N}_{0}$ be such that $k_{1} < .... < k_{s}$. Then
\begin{equation}
\label{eq.416''}
\begin{split}
&1-\sum\limits_{j=1}^{s-1}\sum\limits_{\widetilde{m}^{k_{j}} \in \widetilde{\mathcal{A}}_{k_{j}}}
\psi_{k_{j},\widetilde{m}^{k_{j}}}(x)\prod\limits_{r=j+1}^{s}\Bigl(\sum\limits_{\widetilde{m}^{k_{r}} \in \mathbb{Z}^{n} \setminus \widetilde{\mathcal{A}}_{k_{r}}}\psi_{k_{r},\widetilde{m}^{k_{r}}}(x)\Bigr)
-\sum\limits_{\widetilde{m}^{k_{s}} \in \widetilde{\mathcal{A}}_{k_{s}}}\psi_{k_{s},\widetilde{m}^{k_{s}}}(x)\\
&=\prod\limits_{j=1}^{s}\Bigl(\sum\limits_{\widetilde{m}^{k_{j}} \in \mathbb{Z}^{n} \setminus \widetilde{\mathcal{A}}_{k_{j}}}\psi_{k_{j},\widetilde{m}^{k_{j}}}(x)\Bigr) \quad \hbox{for all} \quad
x \in \mathbb{R}^{n},
\end{split}
\end{equation}
where the first sum in the left-hand side of \eqref{eq.416''} is zero in the case $s=1$.
\end{Prop}

\begin{proof}
We prove \eqref{eq.416''} by induction.

\textit{The base}. For $s=1$, this is obvious because the second term in the left-hand side of \eqref{eq.416''} is zero by definition.

\textit{The induction step}. Suppose that we have proved \eqref{eq.416''} for some $s_{0} \in \mathbb{N}$
and arbitrary nonnegative integer numbers $k'_{1} < ... < k'_{s_{0}}$ (in place of $\{k_{i}\}_{i=1}^{s}$).
To make the induction step, we take an arbitrary $k_{1} < ... < k_{s_{0}+1}$ and apply \eqref{eq.416''} with
$k'_{i}=k_{i+1}$, $i=1,...,s_{0}$. This gives
\begin{equation}
\label{eq416''}
\begin{split}
&1-\sum\limits_{j=1}^{s_{0}}\sum\limits_{\widetilde{m}^{k_{j}} \in \widetilde{\mathcal{A}}_{k_{j}}}
\psi_{k_{j},\widetilde{m}^{k_{j}}}(x)\prod\limits_{r=j+1}^{s_{0}+1}\Bigl(\sum\limits_{\widetilde{m}^{k_{r}} \in \mathbb{Z}^{n} \setminus \widetilde{\mathcal{A}}_{k_{r}}}\psi_{k_{r},\widetilde{m}^{k_{r}}}(x)\Bigr)
-\sum\limits_{\widetilde{m}^{k_{s_{0}+1}} \in \widetilde{\mathcal{A}}_{k_{s_{0}+1}}}\psi_{k_{s_{0}+1},\widetilde{m}^{k_{s_{0}+1}}}(x)\\
&=\prod\limits_{j=2}^{s_{0}+1}\Bigl(\sum\limits_{\widetilde{m}^{k_{j}} \in \mathbb{Z}^{n} \setminus \widetilde{\mathcal{A}}_{k_{j}}}\psi_{k_{j},\widetilde{m}^{k_{j}}}(x)\Bigr)-
\sum\limits_{\widetilde{m}^{k_{1}} \in \widetilde{\mathcal{A}}_{k_{1}}}\psi_{k_{1},\widetilde{m}^{k_{1}}}(x)\prod\limits_{j=2}^{s_{0}+1}\Bigl(\sum\limits_{\widetilde{m}^{k_{j}} \in \mathbb{Z}^{n} \setminus \widetilde{\mathcal{A}}_{k_{j}}}\psi_{k_{j},\widetilde{m}^{k_{j}}}(x)\Bigr)\\
&=\Bigl(1-\sum\limits_{\widetilde{m}^{k_{1}} \in \mathcal{A}_{k_{1}}}\psi_{k_{1},m^{k_{1}}}(x)\Bigr)\prod\limits_{j=2}^{s_{0}+1}\Bigl(\sum\limits_{\widetilde{m}^{k_{j}} \in \mathbb{Z}^{n} \setminus \widetilde{\mathcal{A}}_{k_{j}}}\psi_{k_{j},\widetilde{m}^{k_{j}}}(x)\Bigr)=\prod\limits_{j=1}^{s_{0}+1}\Bigl(\sum\limits_{\widetilde{m}^{k_{j}} \in \mathbb{Z}^{n} \setminus \widetilde{\mathcal{A}}_{k_{j}}}\psi_{k_{j},\widetilde{m}^{k_{j}}}(x)\Bigr).
\end{split}
\end{equation}

\end{proof}

\begin{Remark}
\label{Remark4.2}
By \eqref{eqq.5.1} and Proposition \ref{Prop43}, for each $s \in \mathbb{N} \cap [2,+\infty)$, we have
\begin{equation}
\label{eq517''}
0 \le \sum\limits_{j=1}^{s-1}\sum\limits_{\widetilde{m}^{k_{j}} \in \widetilde{\mathcal{A}}_{k_{j}}}
\psi_{k_{j},\widetilde{m}^{k_{j}}}(x)\prod\limits_{r=j+1}^{s}\Bigl(\sum\limits_{\widetilde{m}^{k_{r}} \in
\mathbb{Z}^{n} \setminus \widetilde{\mathcal{A}}_{k_{r}}}\psi_{k_{r},\widetilde{m}^{k_{r}}}(x)\Bigr) \le 1
\end{equation}
for any $\{k_{i}\}_{i=1}^{s} \subset \mathbb{N}_{0}$ with  $k_{1} < .... < k_{s}$.
\end{Remark}
\hfill$\Box$

Now we formulate the main result of this section. This result contains an important computation, which will
be an indispensable tool in proving some pointwise estimates in Section 6.~We recall Definition \ref{k.index}.

\begin{Th}
\label{Lm.derivative}
There exists a constant $C > 0$ depending only on $C_{\widetilde{\psi}_{0}}$ and $n$ such that, for each
$f \in L_{1}(\{\mathfrak{m}_{k}\})$, for every $y \in \frac{14}{5}Q_{0,0} \setminus S$ and $k^{\ast} \in \underline{K}(y)$,
\begin{equation}
\label{eq.230''}
\|\nabla f_{k}(y)\| \le C \operatorname{M}_{k,c}(y) \quad \text{for all} \quad k \geq k^{\ast} \quad \text{and all} \quad c \in \mathbb{R},
\end{equation}
where
\begin{equation}
\label{eq4.21}
\operatorname{M}_{k,c}(y):=2^{k}\max\fint\limits_{Q_{j,m}}|f(x)-c|\,d\mathfrak{m}_{j}(x),
\end{equation}
the maximum in \eqref{eq4.21} is taken over all $j \in \{k^{\ast},...,k\}$ and  all $m \in \mathcal{A}_{j}$ with $\chi_{\frac{16}{5}Q_{j,m}}(y) \neq 0$.
\end{Th}

\begin{proof}
We fix arbitrary $k \geq k^{\ast}$ and $c \in \mathbb{R}$.
An application of Lemma \ref{Lm.function} with $i=k^{\ast}$ gives (below we assume that the corresponding sum is zero if $k=k^{\ast}$)
\begin{equation}
\label{eq227}
\nabla f_{k}(y)= \nabla f_{k^{\ast}}(y)+\sum\limits_{j=k^{\ast}+1}^{k}\nabla S^{j}_{k^{\ast},k}(y).
\end{equation}
Without loss of generality we will assume that $k > k^{\ast}$ because the case $k = k^{\ast}$ is much simpler.
We split the proof into several steps.

\textit{Step 1.} Since $k^{\ast} \in \underline{K}(y)$ by \eqref{eq.resol.unity} and Remark \ref{Remm.5.2}  we have
\begin{equation}
\label{eqq.5.28}
\nabla f_{k^{\ast}}(y) = \sum\limits_{\widetilde{m} \in \widetilde{\mathcal{A}}_{k^{\ast}}}\nabla\psi_{k^{\ast},\widetilde{m}}(y)f_{k^{\ast},\widetilde{m}} =
\sum\limits_{\widetilde{m} \in \widetilde{\mathcal{A}}_{k^{\ast}}}\nabla\psi_{k^{\ast},\widetilde{m}}(y)(f_{k^{\ast},\widetilde{m}}-c).
\end{equation}
Hence, by \eqref{eqq.5.1}--\eqref{eqq.5.3} we get
\begin{equation}
\label{eq.529''}
\|\nabla f_{k^{\ast}}(y)\| \le C 2^{k^{\ast}} \sum\limits_{\widetilde{m} \in \widetilde{\mathcal{A}}_{k^{\ast}}}\chi_{\frac{6}{5}Q_{k^{\ast},\widetilde{m}}}(y)|f_{k^{\ast},\widetilde{m}}-c|.
\end{equation}
Consequently, by \eqref{eqq.5.12} and \eqref{eq.529''} we obtain
\begin{equation}
\label{eq422''}
\begin{split}
&\|\nabla f_{k^{\ast}}(y)\| \le C 2^{k^{\ast}}\sum\limits_{\widetilde{m} \in \widetilde{\mathcal{A}}_{k^{\ast}}}
\chi_{\frac{6}{5}Q_{k^{\ast},\widetilde{m}}}(y)\max\limits_{m \in \operatorname{n}_{k^{\ast}}(\widetilde{m})\cap \mathcal{A}_{k^{\ast}}}
\fint\limits_{Q_{k^{\ast},m}}|f(x)-c|\,d\mathfrak{m}_{k^{\ast}}(x)\\
&\le C\Bigl(\sum\limits_{m \in \mathbb{Z}^{n}}\chi_{\frac{6}{5}Q_{k^{\ast},\widetilde{m}}}(y)\Bigr)\operatorname{M}_{k,c}(y)
\le C\operatorname{M}_{k,c}(y).
\end{split}
\end{equation}

\textit{Step 2.}
By \eqref{eq221} we get, for each $j \in \{k^{\ast}+1,...,k-1\}$ (the corresponding product disappears when $k < k^{\ast}+2$),
\begin{equation}
\label{eq228}
\begin{split}
&\nabla S^{j}_{k^{\ast},k}(y)=\sum\limits_{\widetilde{m}^{j} \in \widetilde{\mathcal{A}}_{j}}\nabla\Bigl(\psi_{j,\widetilde{m}^{j}}(y)\prod\limits_{r=j+1}^{k}\Bigl(
\sum\limits_{\widetilde{m}^{r} \in \mathbb{Z}^{n} \setminus \widetilde{\mathcal{A}}_{r}}\psi_{r,\widetilde{m}^{r}}(y)\Bigr)\Bigr)(f_{j,\widetilde{m}^{j}}-c+c-f_{k^{\ast}}(y))\\
&-\sum\limits_{\widetilde{m}^{j} \in \widetilde{\mathcal{A}}_{j}}
\psi_{j,\widetilde{m}^{j}}(y)\prod\limits_{r=j+1}^{k}\Bigl(\sum\limits_{\widetilde{m}^{r} \in \mathbb{Z}^{n} \setminus \widetilde{\mathcal{A}}_{r}}\psi_{r,\widetilde{m}^{r}}(y)\Bigr)\nabla
f_{k^{\ast}}(y).\\
\end{split}
\end{equation}
We also have
\begin{equation}
\label{eq228''}
\begin{split}
&\nabla S^{k}_{k^{\ast},k}(y)= \sum\limits_{\widetilde{m}^{k} \in \widetilde{\mathcal{A}}_{k}}\nabla \psi_{k,\widetilde{m}^{k}}(y)(f_{k,\widetilde{m}^{k}}-c+c -f_{k^{\ast}}(y))-\sum\limits_{\widetilde{m}^{k} \in \widetilde{\mathcal{A}}_{k}}
\psi_{k,\widetilde{m}^{k}}(y)\nabla f_{k^{\ast}}(y).\\
\end{split}
\end{equation}

\textit{Step 3.} For each $j\in \{k^{\ast}+1,...,k-1\}$ we use the
Leibniz rule. Using \eqref{eqq.5.1} -- \eqref{eqq.5.3}, we obtain (the corresponding product in the last string disappears if $k \le k^{\ast}+2$)
\begin{equation}
\begin{split}
\label{eq4.23}
&\Sigma^{j}_{k}(y):=\sum_{\widetilde{m}^{j} \in \widetilde{\mathcal{A}}_{j}}\Bigl\|\nabla\Bigl(\psi_{j,m^{j}}(y)\prod_{r=j+1}^{k}\Bigl(
\sum_{\widetilde{m}^{r} \in \mathbb{Z}^{n} \setminus \widetilde{\mathcal{A}}_{r}}\psi_{r,\widetilde{m}^{r}}(y)\Bigr)\Bigr)\Bigr\|\\
&\le \sum_{\widetilde{m}^{j} \in \widetilde{\mathcal{A}}_{j}}\|\nabla \psi_{j,\widetilde{m}^{j}}(y)\|\prod_{r=j+1}^{k}\Bigl(
\sum_{\widetilde{m}^{r} \in \mathbb{Z}^{n} \setminus \widetilde{\mathcal{A}}_{r}}\psi_{r,\widetilde{m}^{r}}(y)\Bigr)\\
&+ \sum_{\widetilde{m}^{j} \in \widetilde{\mathcal{A}}_{j}}\psi_{j,\widetilde{m}^{j}}(y)\Bigl\|\nabla \Bigl(\prod_{r=j+1}^{k}\Bigl(
\sum_{\widetilde{m}^{r} \in \mathbb{Z}^{n} \setminus \widetilde{\mathcal{A}}_{r}}\psi_{r,\widetilde{m}^{r}}(y)\Bigr)\Bigr)\Bigr\|\\
&\le C 2^{j}
+  \Bigl(\sum\limits_{\widetilde{m}^{j} \in \widetilde{\mathcal{A}}_{j}}\psi_{j,\widetilde{m}^{j}}(y)\Bigr)
\sum\limits_{r'=j+1}^{k}C 2^{r'}\prod\limits_{\substack{r=j+1 \\ r \neq r'}}^{k}
\Bigl(\sum\limits_{\widetilde{m}^{r} \in \mathbb{Z}^{n} \setminus \widetilde{\mathcal{A}}_{r}}\psi_{r,\widetilde{m}^{r}}(y)\Bigr)\Bigr).
\end{split}
\end{equation}
We use, for each $r' \in \{k^{\ast}+2,...,k\}$,
Remark \ref{Remark4.2} with $s=k-k^{\ast}-1$ and with
\begin{equation}
\begin{split}
\notag
&k_{1}:=k^{\ast}+1,...,k_{r'-k^{\ast}-1}:= r'-1 \quad \text{and} \quad k_{r'-k^{\ast}}:=r'+1,...,k_{s}:=k.
\end{split}
\end{equation}
This gives (we assume that $k > k^{\ast}+2$) the crucial estimate
\begin{equation}
\notag
\begin{split}
&\sum\limits_{j=k^{\ast}+1}^{r'-1}
\Bigl(\sum\limits_{\widetilde{m}^{j} \in \widetilde{\mathcal{A}}_{j}}\psi_{j,\widetilde{m}^{j}}(y)\Bigr)\prod\limits_{\substack{r=j+1 \\ r \neq r'}}^{k}
\Bigl(\sum\limits_{\widetilde{m}^{r} \in \mathbb{Z}^{n} \setminus \widetilde{\mathcal{A}}_{r}}\psi_{r,\widetilde{m}^{r}}(y)\Bigr)\\
&\le \sum\limits_{\substack{j=k^{\ast}+1\\j \neq r'}}^{k-1}\Bigl(\sum\limits_{\widetilde{m}^{j} \in \widetilde{\mathcal{A}}_{j}}\psi_{j,\widetilde{m}^{j}}(y)\Bigr)\prod\limits_{\substack{r=j+1 \\ r \neq r'}}^{k}
\Bigl(\sum\limits_{\widetilde{m}^{r} \in \mathbb{Z}^{n} \setminus \widetilde{\mathcal{A}}_{r}}\psi_{r,\widetilde{m}^{r}}(y)\Bigr) \le 1.
\end{split}
\end{equation}
Hence, using \eqref{eq4.23} and changing the order of summing, we get (the corresponding product below disappears if $k = k^{\ast}+2$)
\begin{equation}
\label{eq423''}
\begin{split}
&\sum\limits_{j=k^{\ast}+1}^{k-1}\Sigma^{j}_{k}(y)\\
&\le C\Bigl( 2^{k}
+\sum\limits_{r'=k^{\ast}+2}^{k}2^{r'}\sum\limits_{j=k^{\ast}+1}^{r'-1}
\Bigl(\sum\limits_{\widetilde{m}^{j} \in \widetilde{\mathcal{A}}_{j}}\psi_{j,\widetilde{m}^{j}}(y)\Bigr)\prod\limits_{\substack{r=j+1 \\ r \neq r'}}^{k}
\Bigl(\sum\limits_{\widetilde{m}^{r} \in \mathbb{Z}^{n} \setminus \widetilde{\mathcal{A}}_{r}}\psi_{r,\widetilde{m}^{r}}(y)\Bigr)\Bigr) \le
C2^{k}.
\end{split}
\end{equation}
Furthermore, by \eqref{eqq.5.3} we have
\begin{equation}
\label{eq4.23''}
\Sigma^{k}_{k}(y):=\sum_{\widetilde{m}^{k} \in \widetilde{\mathcal{A}}_{k}}\|\nabla \psi_{k,\widetilde{m}^{k}}(y)\| \le C 2^{k}.
\end{equation}

\textit{Step 4.} Now  we plug \eqref{eqq.5.28}, \eqref{eq228}, \eqref{eq228''}  into \eqref{eq227}.
This gives
\begin{equation}
\label{eq234'}
\begin{split}
\|\nabla f_{k}(y)\| \le \sum\limits_{i=1}^{7}R^{i}_{k}(y).
\end{split}
\end{equation}
In the right-hand side of \eqref{eq234'} we put
\begin{equation}
\begin{split}
&R^{1}_{k}(y):=\|\nabla f_{k^{\ast}}(y)\|, \quad R^{7}_{k}(y):=\sum\limits_{\widetilde{m}^{k} \in \widetilde{\mathcal{A}}_{k}}\psi_{k,\widetilde{m}^{k}}(y)\|\nabla f_{k^{\ast}}(y)\|\\
&R^{4}_{k}(y):=\sum\limits_{j=k^{\ast}+1}^{k-1}\sum\limits_{\widetilde{m}^{j} \in \widetilde{\mathcal{A}}_{j}}
\psi_{j,\widetilde{m}^{j}}(y)\prod\limits_{r=j+1}^{k}\Bigl(\sum\limits_{\widetilde{m}^{r} \in \mathbb{Z}^{n} \setminus \widetilde{\mathcal{A}}_{r}}\psi_{r,\widetilde{m}^{r}}(y)\Bigr)
\|\nabla f_{k^{\ast}}(y)\|,\\
&R^{3}_{k}(y):=\sum\limits_{j=k^{\ast}+1}^{k-1}\Sigma^{j}_{k}(y)|c-f_{k^{\ast}}(y)|, \quad R^{6}_{k}(y):= \sum\limits_{\widetilde{m}^{k} \in \widetilde{\mathcal{A}}_{k}}\|\nabla \psi_{k,\widetilde{m}^{k}}(y)\| |c-f_{k^{\ast}}(y)|,\\
&R^{2}_{k}(y):=\sum\limits_{j=k^{\ast}+1}^{k-1}\Sigma^{j}_{k}(y)\sum\limits_{\widetilde{m}^{j} \in \widetilde{\mathcal{A}}^{j}}\chi_{\frac{6}{5}Q_{j,\widetilde{m}^{j}}}(y)|f_{j,\widetilde{m}^{j}}-c|, \quad
R^{5}_{k}(y):=\sum\limits_{\widetilde{m}^{k} \in \widetilde{\mathcal{A}}_{k}}\|\nabla \psi_{k,\widetilde{m}^{k}}(y)\||f_{k,\widetilde{m}^{k}}-c|.
\end{split}
\end{equation}

\textit{Step 5.} By  \eqref{eqq.5.2} and \eqref{eq422''} we have
\begin{equation}
\label{eqq5.38}
\begin{split}
&R^{7}_{k}(y) \le R^{1}_{k}(y)
\le C \operatorname{M}_{k,c}(y).
\end{split}
\end{equation}

\textit{Step 6.} By Remark \ref{Remark4.2} and \eqref{eq422''} we obtain
\begin{equation}
\label{eqq5.39}
\begin{split}
&R^{4}_{k}(y)  \le \|\nabla f_{k^{\ast}}(y)\|=R^{1}_{k}(y) \le C \operatorname{M}_{k,c}(y).
\end{split}
\end{equation}

\textit{Step 7.} By \eqref{eq423''} and \eqref{eq4.23''} we clearly get
\begin{equation}
\notag
\begin{split}
&R^{3}_{k}(y) \le C 2^{k}|c-f_{k^{\ast}}(y)|, \quad R^{6}_{k}(y) \le C 2^{k}|c-f_{k^{\ast}}(y)|.\\
\end{split}
\end{equation}
Hence, using arguments similar to those used in step 1, we deduce
\begin{equation}
\label{eqq5.40}
\begin{split}
R^{3}_{k}(y)+R^{6}_{k}(y) \le C \operatorname{M}_{k,c}(y).
\end{split}
\end{equation}

\textit{Step 8.} Finally, we use Proposition \ref{Prop4.1'''} and take into account \eqref{eq4.21}, \eqref{eq423''}, \eqref{eq4.23''}. This leads us to the estimates
\begin{equation}
\label{eqq5.41}
\begin{split}
&R^{2}_{k}(y) \le C \operatorname{M}_{k,c}(y), \quad  R^{5}_{k}(y) \le C  \operatorname{M}_{k,c}(y).
\end{split}
\end{equation}

\textit{Step 9.} Collecting \eqref{eq234'}--\eqref{eqq5.41}, we deduce  \eqref{eq.230''} and complete the proof.

\end{proof}


\section{The reverse trace theorem}

In this section we prove the \textit{so-called reverse trace theorem.} More precisely, given a nonempty compact set $S \subset Q_{0,0}$
and a function $f \in \mathfrak{B}(S)$, we find conditions sufficient for the existence of a Sobolev extension $F$ of $f$.

\textit{Throughout the section we fix the following data:}

{\rm(\textbf{D.6.1})} parameters $d^{\ast} \in (0,n]$, $d \in (0,d^{\ast}]$ and a compact set $S \subset Q_{0,0}$ with $\overline{\lambda}:=\mathcal{H}^{d^{\ast}}_{\infty}(S) > 0$;

{\rm(\textbf{D.6.2})} an arbitrary parameter $\lambda \in (0,\overline{\lambda})$ ;

{\rm(\textbf{D.6.3})} an arbitrary sequence of measures $\{\mathfrak{m}_{k}\} \in \mathfrak{M}^{d}(S)$.


As a result, we simplify our notation.
More precisely, we set $\mathcal{DF}:=\mathcal{DF}_{S}(d,\lambda)$, $\underline{S}:=\underline{S}(d,\lambda)$, $S_{f}:=S_{f}(d)$.
For each $k \in \mathbb{N}_{0}$ we set $\mathcal{DF}_{k}:=\mathcal{DF}_{S,k}(d,\lambda)$
and $\widetilde{\mathcal{DF}}_{k}:=\widetilde{\mathcal{DF}}_{S,k}(d,\lambda)$.
Given $c \geq 1$, we put $T_{c}(Q):=T_{d,\lambda,c}(Q)$, $\mathcal{K}_{c}(Q):=\mathcal{K}_{d,\lambda,c}(Q)$.
We recall Definition \ref{k.index} and, for any given $y \in \mathbb{R}^{n}$, we put $\underline{K}(y):=\underline{K}_{d,\lambda}(y)$ and $\overline{K}(y):=\overline{K}_{d,\lambda}(y)$.
We also use the symbol $\operatorname{Ext}$ instead of $\operatorname{Ext}_{S,\{\mathfrak{m}_{k}\},\lambda}$ to denote the corresponding
extension operator constructed in Section 5.1. Finally, throughout the section we put $f^{\natural}_{c}:=f^{\natural}_{\{\mathfrak{m}_{k}\},\lambda,c}$.

The following elementary combinatorial fact will be a keystone in proving the main results of this section. We recall Remark \ref{Remm.5.1}.

\begin{Lm}
\label{Lm.keystone_combinatorial}
Let $y \in \frac{14}{5}Q_{0,0}$. Let $k_{1},k_{2} \in \mathbb{N}_{0}$ and $j_{1},j_{2} \in \overline{K}(y)$ be such that:

{\rm (1)} $k_{1} \le j_{1} \le j_{2} \le k_{2}$;

{\rm (2)} $(k_{1},k_{2}) \cap \underline{K}(y) = \emptyset$.

Then $Q_{j_{1},m_{1}} \in \mathcal{K}_{7}(Q_{j_{2},m_{2}})$ for any $Q_{j_{i},m_{i}} \in \mathcal{DF}_{j_{i}}$, $i=1,2$
satisfying $y \in \frac{16}{5}Q_{j_{1},m_{1}} \cap \frac{16}{5}Q_{j_{2},m_{2}}$.
\end{Lm}

\begin{proof}
Let $Q_{j_{i},m_{i}} \in \mathcal{DF}_{j_{i}}$, $i=1,2$ be such that $y \in \frac{16}{5}Q_{j_{1},m_{1}} \cap \frac{16}{5}Q_{j_{2},m_{2}}$.

Fix a dyadic cube $Q \supset Q_{j_{2},m_{2}}$ with $l(Q) \in (2^{-j_{2}},2^{-j_{1}})$.
We claim that $Q \notin \mathcal{DF}$. Indeed, assume the contrary.
On the one hand, by the construction
\begin{equation}
\label{eq.6.key1}
-\log_{2}l(Q) \in (j_{1},j_{2}) \subset (k_{1},k_{2}).
\end{equation}
On the other hand, since $Q \in \mathcal{D}_{+}$ we have $l(Q) \geq 2 l(Q_{j_{2},m_{2}})$.
Hence, taking into account that $y \in \frac{16}{5}Q_{j_{2},m_{2}}$ it is easy to see that $y \in \frac{14}{5}Q$. This implies
that
\begin{equation}
\label{eq.6.key2}
-\log_{2}l(Q) \in \underline{K}(y).
\end{equation}
By \eqref{eq.6.key1} and \eqref{eq.6.key2} we have $(k_{1},k_{2}) \cap \underline{K}(y) \neq \emptyset$, which contradicts
the assumptions of the lemma. This proves the claim.

Since $y \in \frac{16}{5}Q_{j_{1},m_{1}} \cap \frac{16}{5}Q_{j_{2},m_{2}}$ we clearly have
$\operatorname{dist}(Q_{j_{1},m_{1}},Q_{j_{2},m_{2}}) \le \frac{11}{5}2^{-j_{1}}$.
Since the cubes are dyadic by Proposition \ref{Prop21''''} we get $\operatorname{dist}(Q_{j_{1},m_{1}},Q_{j_{2},m_{2}}) \le 2^{1-j_{1}}$.
Hence, $Q_{j_{2},m_{2}} \subset 7Q_{j_{1},m_{1}}$.

The proof is complete.
\end{proof}

Having at our disposal Lemma \ref{Lm.keystone_combinatorial} we can establish an important estimate.

\begin{Lm}
\label{Lm.lp_minus_lq}
For each $f \in L_{1}(\{\mathfrak{m}_{k}\})$, for every $y \in \frac{14}{5}Q_{0,0}$, the following holds.
If $k_{1},k_{2} \in \mathbb{N}_{0}$ and $j_{1},j_{2} \in \overline{K}(y)$ satisfy the assumptions of Lemma \ref{Lm.keystone_combinatorial},
then
\begin{equation}
|f_{j_{1},\widetilde{m}_{1}}-f_{j_{2},\widetilde{m}_{2}}| \le   2^{-\max\{j_{1},j_{2}\}}f^{\sharp}_{7}(y)
\end{equation}
for any $Q_{j_{i},\widetilde{m}_{i}} \in \widetilde{\mathcal{DF}}_{j_{i}}$, $i=1,2$ satisfying $y \in \frac{6}{5}Q_{j_{1},\widetilde{m}_{1}} \cap \frac{6}{5}Q_{j_{2},\widetilde{m}_{2}}$.
\end{Lm}

\begin{proof}
Given $f \in L_{1}(\{\mathfrak{m}_{k}\})$ and $y \in \frac{14}{5}Q_{0,0}$,
we fix $k_{1},k_{2} \in \mathbb{N}_{0}$ and $j_{1},j_{2} \in \overline{K}(y)$ satisfying the assumptions of Lemma \ref{Lm.keystone_combinatorial}. Furthermore, we fix
$Q_{j_{i},\widetilde{m}_{i}} \in \widetilde{\mathcal{DF}}_{j_{i}}$, $i=1,2$ such that $y \in \frac{6}{5}Q_{j_{1},\widetilde{m}_{1}} \cap \frac{6}{5}Q_{j_{2},\widetilde{m}_{2}}$.
It is clear that,
for any $m_{i} \in \operatorname{n}_{j_{i}}(\widetilde{m}_{i})$, $i=1,2$, we have
\begin{equation}
\label{eq.6.key_inclusion}
y \in \frac{16}{5}Q_{j_{1},m_{1}} \cap \frac{16}{5}Q_{j_{2},m_{2}}
\end{equation}
We use Proposition \ref{Prop4.1'''}, take into account \eqref{eq.6.key_inclusion}, and
apply Lemma \ref{Lm.keystone_combinatorial}. This gives
\begin{equation}
\begin{split}
&|f_{j_{1},\widetilde{m}_{1}}-f_{j_{2},\widetilde{m}_{2}}|\\
&\le \max\limits_{\substack{ i =1,2 \\
m_{i} \in \operatorname{n}_{j_{i}}(\widetilde{m}_{i}) \cap \mathcal{A}_{j_{i}}}}\fint\limits_{Q_{j_{1}},m_{1}}\fint\limits_{Q_{j_{2}},m_{2}}|f(x)-f(y)|\,
d\mathfrak{m}_{j_{1}}(x)d\mathfrak{m}_{j_{2}}(y) \le 2^{-\max\{j_{1},j_{2}\}}f^{\sharp}_{7}(y).
\end{split}
\end{equation}
The proof is complete.
\end{proof}

We recall Definitions \ref{Def.spec.app.seq} and \ref{k.index}. Given $f \in L_{1}(\{\mathfrak{m}_{k}\})$, we establish a useful pointwise estimate of the functions $f_{k}-f_{l}$ for any $k,l \in \mathbb{N}$. This is crucial
to deduce convergence properties of the special approximating sequence $\{f_{k}\}$.

\begin{Th}
\label{Th.convergence}
Let $f \in L_{1}(\{\mathfrak{m}_{k}\})$ and $y \in \frac{14}{5}Q_{0,0}$.~Let $\{f_{k}\}$ be the special approximating sequence for $f$.~Then, for any $k,s \in \mathbb{N}_{0}$ with $s \geq k$,
\begin{equation}
\label{eqq.6.7}
|f_{k}(y)-f_{s}(y)| \le 2^{-\underline{k}(y,k)+3}f^{\natural}_{7}(y),
\end{equation}
where $\underline{k}(y,k)=\max\{k' \in \underline{K}(y):k' \le k\}$.
\end{Th}

\begin{proof}
We fix $k,s \in \mathbb{N}_{0}$ with $s \geq k$ and split the proof into several steps.

\textit{Step 1.}
By Remark \ref{Remm.5.1}, we have $\underline{K}(y) \neq \emptyset$. We
write the set $\underline{K}(y)$ in an increasing order, i.e.,
$\underline{K}(y) = \{k_{l}\}$ where $\{k_{l}\}=\{k_{l}\}_{l=1}^{N}$, $N \in \mathbb{N}\cup +\infty$ is an increasing sequence or a finite family of numbers taken in an increasing order. We put
\begin{equation}
\label{eqq.6.index}
\underline{l}:=\max\{l:k_{l} \le k\} \quad \text{and} \quad \overline{l}:=\max\{l:k_{l} \le s\}.
\end{equation}
We clearly have (the second sum disappears in the case $\underline{l}=\overline{l}$)
\begin{equation}
\label{eqq.6.8}
|f_{k}(y)-f_{s}(y)| \le |f_{k}(y)-f_{k_{\underline{l}}}(y)|+\sum\limits_{l=\underline{l}}^{\overline{l}-1}|f_{k_{l}}(y)-f_{k_{l+1}}(y)|+|f_{s}(y)-f_{k_{\overline{l}}}(y)|.
\end{equation}

\textit{Step 2.}
Since $\{k_{l}\} \subset  \underline{K}(y)$, by Remark \ref{Remm.5.2} we get
\begin{equation}
\label{eqq.6.9}
f_{k_{l}}(y)=\sum\limits_{\widetilde{m}^{k_{l}} \in \widetilde{\mathcal{A}}_{k_{l}}}\psi_{k_{l},\widetilde{m}^{k_{l}}}(y)f_{k_{l},\widetilde{m}^{k_{l}}} \quad \text{for each} \quad l.
\end{equation}
By \eqref{eqq.6.9}, for each $l \in \{\underline{l},...,\overline{l}-1\}$, we obtain
\begin{equation}
\label{eqq.609}
\begin{split}
&|f_{k_{l}}(y)-f_{k_{l+1}}(y)| \le \sum\limits_{\widetilde{m}^{k_{l}} \in \widetilde{\mathcal{A}}_{k_{l}}}\sum\limits_{\widetilde{m}^{k_{l+1}} \in \widetilde{\mathcal{A}}_{k_{l+1}}}\psi_{k_{l},\widetilde{m}^{k_{l}}}(y)
\psi_{k_{l+1},\widetilde{m}^{k_{l+1}}}(y)|f_{k_{l},\widetilde{m}^{k_{l}}}
-f_{k_{l+1},\widetilde{m}^{k_{l+1}}}|.
\end{split}
\end{equation}

The crucial observation is that, given $l \in \{\underline{l},...,\overline{l}-1\}$, we have $(k_{l},k_{l+1}) \cap \underline{K}(y) = \emptyset$. Hence,
an application
of Lemma \ref{Lm.lp_minus_lq} with $k_{1}=j_{1}=k_{l}$ and $k_{2}=j_{2}=k_{l+1}$ gives
\begin{equation}
\label{eqq.610}
|f_{k_{l},\widetilde{m}^{k_{l}}}
-f_{k_{l+1},\widetilde{m}^{k_{l+1}}}| \le 2^{-k_{l+1}}f^{\natural}_{7}(y)
\end{equation}
for every $l \in \{\underline{l},...,\overline{l}-1\}$ and any indexes $\widetilde{m}^{k_{l}} \in \widetilde{A}_{k_{l}}$, $\widetilde{m}^{k_{l+1}} \in \widetilde{A}_{k_{l+1}}$
satisfying $\psi_{k_{l},\widetilde{m}^{k_{l}}}(y) \neq 0$ and $\psi_{k_{l+1},\widetilde{m}^{k_{l+1}}}(y) \neq 0$.

As a result, we plug \eqref{eqq.610} into \eqref{eqq.609} and take into account \eqref{eqq.5.2}. We obtain
\begin{equation}
\label{eqq.6.10}
\begin{split}
&|f_{k_{l},\widetilde{m}^{k_{l}}}
-f_{k_{l+1},\widetilde{m}^{k_{l+1}}}|\\
&\le 2^{-k_{l+1}}f^{\natural}_{7}(y)\sum\limits_{\widetilde{m}^{k_{l}} \in \widetilde{\mathcal{A}}_{k_{l}}}
\sum\limits_{\widetilde{m}^{k_{l+1}} \in \widetilde{\mathcal{A}}_{k_{l+1}}}\psi_{k_{l},\widetilde{m}^{k_{l}}}(y)\psi_{k_{l+1},\widetilde{m}^{k_{l+1}}}(y) \le 2^{-k_{l+1}}f^{\natural}_{7}(y).
\end{split}
\end{equation}

\textit{Step 3.} We assume that $k > k_{\underline{l}}$ because otherwise $f_{k}(y)-f_{k_{\underline{l}}}(y)=0$.
An application of Lemma \ref{Lm.function} with $i=k_{\underline{l}}$ gives
\begin{equation}
\label{eqq.6.11}
|f_{k}(y)-f_{k_{\underline{l}}}(y)| \le
\sum\limits_{r=k_{\underline{l}}+1}^{k}
S^{r}_{k_{\underline{l}},k}(y).
\end{equation}
If $k-k_{\underline{k}} \geq 2$, then by \eqref{eq.resol.unity}, \eqref{eq221} and \eqref{eqq.6.9}, for each $r \in \{k_{\underline{l}}+1,...,k-1\}$ (the corresponding product below disappears for $k=k_{\underline{k}}+2$),
\begin{equation}
\label{eqq.611'}
\begin{split}
&S^{r}_{k_{\underline{l}},k}(y) \le \sum\limits_{\widetilde{m}^{r} \in \widetilde{\mathcal{A}}_{r}}\psi_{r,\widetilde{m}^{r}}(y)\Bigl(
\prod\limits_{r'=r+1}^{k}\Bigl(\sum\limits_{\widetilde{m}^{r'} \in \mathbb{Z}^{n} \setminus \widetilde{\mathcal{A}}_{r'}}
\psi_{r',\widetilde{m}^{r'}}(y)\Bigr)\Bigr)\sum\limits_{\widetilde{m}^{k_{\underline{l}}} \in \widetilde{\mathcal{A}}_{k_{\underline{l}}}}\psi_{k_{\underline{l}},\widetilde{m}^{k_{\underline{l}}}}(y)
|f_{r,\widetilde{m}^{r}}-f_{k_{\underline{l}},\widetilde{m}^{k_{\underline{l}}}}|.\\
\end{split}
\end{equation}
Similarly,
\begin{equation}
\label{eqq.611''}
\begin{split}
&S^{k}_{k_{\underline{l}},k}(y) \le \sum\limits_{\widetilde{m}^{k} \in \widetilde{\mathcal{A}}_{k}}\psi_{k,\widetilde{m}^{k}}(y)
\sum\limits_{\widetilde{m}^{k_{\underline{l}}} \in \widetilde{\mathcal{A}}_{k_{\underline{l}}}}\psi_{k_{\underline{l}},\widetilde{m}^{k_{\underline{l}}}}(y)
|f_{k,\widetilde{m}^{k}}-f_{k_{\underline{l}},\widetilde{m}^{k_{\underline{l}}}}|.
\end{split}
\end{equation}

The crucial observation is that by \eqref{eqq.6.index} we have $(k_{\underline{l}},k) \cap \underline{K}(y) = \emptyset$. Hence, given $r \in \{k_{\underline{l}}+1,...,k\}$, applying Lemma \ref{Lm.lp_minus_lq}
with $k_{1}=j_{1}=k_{\underline{l}}$ and $k_{2}=k$, $j_{2}=r$ we obtain
\begin{equation}
\label{eqq.611'''}
|f_{r,\widetilde{m}^{r}}-f_{k_{\underline{l}},\widetilde{m}^{k_{\underline{l}}}}| \le 2^{-r}f^{\natural}_{7}(y)
\end{equation}
for any indices $\widetilde{m}^{r} \in \widetilde{A}_{r}$, $\widetilde{m}^{k_{\underline{l}}} \in \widetilde{\mathcal{A}}_{k_{\underline{l}}}$ satisfying $\psi_{r,\widetilde{m}^{r}}(y) \neq 0$
and $\psi_{k_{\underline{l}},\widetilde{m}^{k_{\underline{l}}}}(y) \neq 0$.

As a result, collecting \eqref{eqq.6.11}--\eqref{eqq.611'''} and taking into account Remark \ref{Remark4.2} we obtain
\begin{equation}
\label{eqq.6.12}
\begin{split}
&|f_{k}(y)-f_{k_{\underline{l}}}(y)| \le f^{\natural}_{7}(y)\sum\limits_{r=k_{\underline{l}}+1}^{k-1}2^{-r}
\sum\limits_{\widetilde{m}^{r} \in \widetilde{\mathcal{A}}_{r}}\psi_{r,\widetilde{m}^{r}}(y)\Bigl(
\prod\limits_{r'=r+1}^{k}\Bigl(\sum\limits_{\widetilde{m}^{r'} \in \mathbb{Z}^{n} \setminus \widetilde{\mathcal{A}}_{r'}}
\psi_{r',\widetilde{m}^{r'}}(y)\Bigr)\Bigr)\\
&+2^{-k}f^{\natural}_{7}(y) \le 2^{-k_{\underline{l}}+1}f^{\natural}_{7}(y).
\end{split}
\end{equation}

\textit{Step 4.} Repeating the arguments of the previous step, we get
\begin{equation}
\label{eqq.6.13}
|f_{l}(y)-f_{k_{\overline{l}}}(y)| \le 2^{-k_{\overline{l}}+1}f^{\natural}_{7}(y).
\end{equation}

It remains to combine \eqref{eqq.6.8}, \eqref{eqq.6.10}, \eqref{eqq.6.12}, \eqref{eqq.6.13} and take into account that $2^{-k_{l}} \le 2^{-\underline{k}(y,k)}$ for
all $l \in \{\underline{l},...,\overline{l}\}$. As a result, we deduce \eqref{eqq.6.7}
and complete the proof.

\end{proof}

\begin{Ca}
\label{Ca.convergence}
Let $p \in (1,\infty)$ and $f \in \mathfrak{B}(S)$ be such that $\widetilde{\mathcal{N}}_{p,\{\mathfrak{m}_{k}\},\lambda,c}(f) < +\infty$.~Let
$\{f_{k}\}$ be the special approximating sequence for $f$.~Then the equivalence class $[\operatorname{Ext}(f)]$ of $\operatorname{Ext}(f)$ belongs to $L_{p}(\mathbb{R}^{n})$ and
the sequence $\{f_{k}\}$ converges to $[\operatorname{Ext}(f)]$ in $L_{p}(\mathbb{R}^{n})$-sense.
\end{Ca}

\begin{proof}
Given $y \in \frac{14}{5}Q_{0,0}$ and $k \in \mathbb{N}_{0}$, we put $\underline{k}(y,k):=\max\{k': k' \in \underline{K}(y) \text{ and } k' \le k\}.$
By Lemma \ref{Lm.function}, it is easy to see that, for any $k,s \in \mathbb{N}_{0}$ with $s > k$,
\begin{equation}
\notag
\operatorname{supp}(f_{k}-f_{s}) \subset \bigcup\limits_{j \geq k} \bigcup \{4Q_{j,m}: m \in \mathcal{A}_{j}\} \subset U_{\frac{4}{2^{k}}}(S).
\end{equation}
Hence, by Theorem \ref{Th.convergence} we have, for such $k,s$,
\begin{equation}
\notag
\int\limits_{\mathbb{R}^{n}}|f_{k}(y)-f_{s}(y)|^{p}\,dy \le \int\limits_{U_{\frac{4}{2^{k}}}(S)}2^{(3-\underline{k}(y,k))p}(f^{\natural}_{7}(y))^{p}\,dy.
\end{equation}
As a result, by absolute continuity of the Lebesgue integral we obtain
\begin{equation}
\label{eqq.6.14}
\varlimsup\limits_{k,s \to \infty}\int\limits_{\mathbb{R}^{n}}|f_{k}(y)-f_{s}(y)|^{p}\,dy \le \varlimsup\limits_{k \to \infty}\int\limits_{S}2^{(3-\underline{k}(y,k))p}(f^{\natural}_{7}(y))^{p}\,dy.
\end{equation}

To prove that $\{f_{k}\}$ is a Cauchy sequence, it is sufficient to show that the right-hand side of \eqref{eqq.6.14} is zero.
To show this, we proceed as follows. First of all, by \eqref{eqq.cardinality}, Proposition \ref{Prop41''}, and the absolute continuity of the Lebesgue
measure $\mathcal{L}^{n}$ with respect to the Hausdorff measure $\mathcal{H}^{d}$, we have
\begin{equation}
\notag
\lim\limits_{k \to \infty}\underline{k}(y,k)=+\infty \quad \text{for} \quad \mathcal{L}^{n}-a.e. \quad y \in S.
\end{equation}
Since $2^{(3-\underline{k}(y,k))p}(f^{\natural}(y))^{p} \le (8f^{\natural}(y))^{p}$ for all $y \in \mathbb{R}^{n}$, an application of the
Lebesgue dominated convergence theorem proves the claim.

Since $\{f_{k}\}$ is a Cauchy sequence in the Banach space $L_{p}(\mathbb{R}^{n})$, we get the existence of $g \in L_{p}(\mathbb{R}^{n})$ such that
$\|g-f_{k}|L_{p}(\mathbb{R}^{n})\| \to 0$, $k \to \infty$. Hence, there is a subsequence $\{f_{k_{s}}\}$ converging $\mathcal{L}^{n}$-a.e.\ to $g$.
Combining this fact with Definition \ref{extension_operator}, Lemma \ref{Lm.function2} and Proposition \ref{Prop41''},
we have $g(x)=f(x)=\operatorname{Ext}(f)(x)$ for $\mathcal{L}^{n}$-a.e. $x \in S$.
This completes the proof.

\end{proof}

The following fact is a folklore. Nevertheless, we present the proof for the completeness.

\begin{Prop}
\label{Prop.lp_admissible}
Given $p \in (1,\infty)$ and $c > 0$, there exists a constant $C > 0$ such that if $F \in L^{loc}_{1}(\mathbb{R}^{n})$ is such that $\operatorname{supp}F \subset cQ_{0,0}$ and the distributional gradient $\nabla F \in L_{p}(\mathbb{R}^{n},\mathbb{R}^{n})$, then
\begin{equation}
\int\limits_{\mathbb{R}^{n}}|F(x)|^{p}\,dx \le C \int\limits_{cQ_{0,0}}\|\nabla F(y)\|^{p}\,dy.
\end{equation}
\end{Prop}

\begin{proof}
It is well known \cite{Haj} that there is a constant $C > 0$ such that, for any
given $F \in L_{1}^{loc}(\mathbb{R}^{n})$ with $\|\nabla F\| \in L_{1}^{loc}(\mathbb{R}^{n})$,
there is a set $E_{F}$ with $\mathcal{L}^{n}(\mathbb{R}^{n} \setminus E_{F})=0$ such that
\begin{equation}
\notag
|F(x)-F(y)| \le C \|x-y\|\Bigl(\mathcal{M}^{\|x-y\|}_{1}[\|\nabla F\|](x)+\mathcal{M}^{\|x-y\|}_{1}[\|\nabla F\|](y)\Bigr) \quad \text{for all} \quad x,y \in \mathbb{R}^{n} \setminus E_{F}.
\end{equation}
If, an addition, $\operatorname{supp}F \subset cQ_{0,0}$, then an application of H\"older's inequality gives, for any point $x \in cQ_{0,0}\setminus E_{F}$,
\begin{equation}
\begin{split}
\notag
&|F(x)|^{p} \le  \Bigl(\fint\limits_{2cQ \setminus cQ}|F(x)-F(y)|\,dy\Bigr)^{p} \le C \fint\limits_{2cQ_{0,0} \setminus cQ_{0,0}}|F(x)-F(y)|^{p}\,dy\\
&\le C\Bigl(\mathcal{M}^{5c}_{1}[\|\nabla F\|](x)\Bigr)^{p}+C\fint\limits_{2cQ_{0,0} \setminus cQ_{0,0}}\Bigl(\mathcal{M}^{5c}_{1}[\|\nabla F\|](y)\Bigr)^{p}\,dy.
\end{split}
\end{equation}
As a result, applying Proposition \ref{Prop.max.funct} with $\sigma=1$ and $R=5c$, we have
\begin{equation}
\int\limits_{cQ_{0,0}}|F(x)|^{p}\,dx \le C \int\limits_{2cQ_{0,0}}\Bigl(\mathcal{M}^{5c}_{1}[\|\nabla F\|](y)\Bigr)^{p}\,dy
\le C \int\limits_{cQ_{0,0}}\|\nabla F (y)\|^{p}\,dy.
\end{equation}
The proof is complete.
\end{proof}

The following assertion is \textit{a keystone result} of this section.

\begin{Th}
\label{estimate.s's'small}
There exists a constant $C > 0$ depending only on $n$, $C_{\widetilde{\psi}}$, $\operatorname{C}_{\{\mathfrak{m}_{k}\},2}$ and $\overline{\lambda}$ such that,
for each $f \in L_{1}(\{\mathfrak{m}_{k}\})$, for every $k \in \mathbb{N}_{0}$,
\begin{equation}
\label{eq589}
\|\nabla f_{k}(x)\| \le C \Bigl((f^{\natural}_{7}(x)+\|f|L_{1}(\mathfrak{m}_{0})\|\Bigr) \quad \text{for $\mathcal{L}^{n}$-a.e.} \quad x \in \mathbb{R}^{n}.
\end{equation}
\end{Th}

\begin{proof}
We fix an element $f \in L_{1}(\{\mathfrak{m}_{k}\})$ and a number $k \in \mathbb{N}_{0}$.
Consider an open set (recall that all the cubes are assumed to be closed)
\begin{equation}
\label{eqq.623''''}
G_{k}:=\mathbb{R}^{n} \setminus \Bigl(\bigcup_{j=0}^{k}\bigcup_{Q \in \mathcal{D}_{j}}(\frac{16}{5}Q\setminus \operatorname{int}\frac{16}{5}Q)\Bigr).
\end{equation}
We fix a point $y \in \frac{14}{5}Q_{0,0} \cap G_{k}$ and split the proof into several steps.

\textit{Step 1.}  By  $(\textbf{D.6.1})$, $(\textbf{D.6.2})$ and
Proposition \ref{Propp.5.1}, we have $0 \in \underline{K}(y) \cap \overline{K}(y).$
We put
\begin{equation}
\begin{split}
\label{eq592}
\underline{k}^{\ast}:=\max\{k' \in \underline{K}(y):k' \le k\}, \quad
\overline{k}^{\ast}:=\max\{k' \in \overline{K}(y):k' \le k\}.
\end{split}
\end{equation}
Hence, by Definition \ref{k.index}, \eqref{eq592} and \eqref{eqq.5.1}, there exists a cube
\begin{equation}
\label{eq593}
Q \in \mathcal{DF}_{\overline{k}^{\ast}} \quad \hbox{such that} \quad y \in \operatorname{int}\frac{16}{5}Q.
\end{equation}

\textit{Step 2.}
Now we make a key observation. By \eqref{eqq.623''''}, it is easy to see that, for any $j \in [\overline{k}^{\ast},k] \cap \mathbb{N}$ and any $Q' \in \mathcal{D}_{j}$, we have (recall that all cubes are closed)
\begin{equation}
\text{either} \quad y \in \operatorname{int}\frac{16}{5}Q' \quad \text{or} \quad y \in \mathbb{R}^{n} \setminus \frac{16}{5}Q'.
\end{equation}
Hence, there exists a small $\delta=\delta(y,k) > 0$ such that $\overline{k}^{\ast}=\max\{k' \in \overline{K}(y'):k' \le k\}$. This implies that
$f_{k}(y')=f_{\overline{k}^{\ast}}(y')$ for all $y' \in Q_{\delta}(y)$. As a result,
\begin{equation}
\label{eqq.628}
\nabla f_{k}(y)=\nabla f_{\overline{k}^{\ast}}(y).
\end{equation}

\textit{Step 3.}
Assume that $\overline{k}^{\ast} > \underline{k}^{\ast}$ because the case $\overline{k}^{\ast}=
\underline{k}^{\ast}$ is based on the same idea, but technically simpler.
Let $j \in [\underline{k}^{\ast},\overline{k}^{\ast}] \cap \mathbb{N}_{0}$ and $Q_{j,m} \in \mathcal{DF}_{j}$ be such that
\begin{equation}
\label{eq595}
y \in \frac{16}{5} Q_{j,m}.
\end{equation}
By \eqref{eq593}, \eqref{eq595} we have $\operatorname{dist}(Q,Q_{j,m}) \le \frac{11}{5}2^{-j}.$
Since $j \le \overline{k}^{\ast}$ and $Q,Q_{j,m} \in \mathcal{D}_{+}$ by Proposition \ref{Prop21''''} we get
\begin{equation}
\label{eqq.615''}
Q \subset 7Q_{j,m}.
\end{equation}

Let $Q' \in \mathcal{D}_{+}$ be such that $Q' \supset Q$ and $l(Q') \in (l(Q),2^{-j})$ (recall that $2^{-j} \geq l(Q)$).
Since $l(Q') \geq 2l(Q)$ by \eqref{eq593} we have
\begin{equation}
\notag
\frac{14}{5} Q' \ni y \quad \text{and} \quad l(Q') < 2^{-j} \le 2^{-\underline{k}^{\ast}}.
\end{equation}
This implies that $Q' \notin \mathcal{DF}$
because otherwise we would immediately get a contradiction with the maximality of $\underline{k}^{\ast}$.
As a result, by \eqref{eqq.615''}  and Definition \ref{cover.cube} we have
\begin{equation}
\label{eq.6113}
Q_{j,m} \in \mathcal{K}_{7}(Q).
\end{equation}

\textit{Step 4.} We apply Theorem \ref{Lm.derivative}
with $k^{\ast}=\underline{k}^{\ast}$, with $k$ replaced by $\overline{k}^{\ast}$, and with $c=\fint_{Q}f(t)\,d\mu_{\overline{k}^{\ast}}(t)$. Then we use \eqref{eqq.628} and get
\begin{equation}
\notag
\begin{split}
&\|\nabla f_{k}(y)\| = \|\nabla f_{\overline{k}^{\ast}}(y)\| \le C 2^{\overline{k}^{\ast}} \max\Bigl|\fint\limits_{Q_{j,m}}f(t)\,d\mathfrak{m}_{j}(t)
-\fint\limits_{Q}f(t)\,d\mathfrak{m}_{\overline{k}^{\ast}}(t)\Bigr|,
\end{split}
\end{equation}
where the maximum is taken over all $j \in \{\underline{k}^{\ast},...,\overline{k}^{\ast}\}$ and all $m \in \mathcal{A}_{j}$ such that $y \in \frac{16}{5} Q_{j,m}$.
Hence, by \eqref{eq593}, \eqref{eq.6113} and Definition \ref{Cal.max.function} we obtain
\begin{equation}
\label{eqq.6.126}
\|\nabla f_{k}(y)\| \le C f^{\natural}_{7}(y).
\end{equation}

\textit{Step 5.} Given $z \in \mathbb{R}^{n} \setminus \frac{14}{5}Q_{0,0}$, by Definition \ref{Def.spec.app.seq} we have $f_{k}(z)=f_{0}(z)$ for all $k \in \mathbb{N}$.
As a result, taking into account that $\mathfrak{m}_{0}(Q_{0,0})=\mathfrak{m}_{0}(S) \geq \overline{\lambda}\operatorname{C}_{\{\mathfrak{m}_{k}\},2}$ we obtain
\begin{equation}
\label{eqq.6.127}
\begin{split}
&\|\nabla f_{k}(z)\| = \|\nabla f_{0}(z)\| \le C \fint\limits_{Q_{0,0}}|f(x)|\,d\mathfrak{m}_{0}(x)\\
&\le C \int\limits_{Q_{0,0}}|f(x)|\,d\mathfrak{m}_{0}(x) \quad \text{for all} \quad z \in \mathbb{R}^{n} \setminus \frac{14}{5}Q_{0,0}
\quad \text{and all} \quad k \in \mathbb{N}_{0}.
\end{split}
\end{equation}

Note that all the constants $C > 0$ in Steps 1--5 depend only on $n$, $C_{\widetilde{\psi}}$, $\operatorname{C}_{\{\mathfrak{m}_{k}\},2}$ and $\overline{\lambda}$.
Furthermore, it is obvious that $\mathcal{L}^{n}(\mathbb{R}^{n} \setminus G_{k})=0$.
Hence, combining \eqref{eqq.6.126}, \eqref{eqq.6.127} and taking into account that $y \in \frac{14}{5}Q_{0,0} \cap G_{k}$ was chosen arbitrarily
we obtain \eqref{eq589} and complete the proof.

\end{proof}

Now we are ready to present \textit{the main result of this section.}

\begin{Th}
\label{Th.Reverse_trace}
Let $p \in (1,\infty)$, $c \geq 7$ and $f \in \mathfrak{B}(S)$ be such that $\widetilde{\mathcal{N}}_{p,\{\mathfrak{m}_{k}\},\lambda,c}(f) < +\infty$. Then
the $\mathcal{L}^{n}$-equivalence class $[\operatorname{Ext}(f)]$ of the function $\operatorname{Ext}(f)$ belongs to $W^{1}_{p}(\mathbb{R}^{n})$. Furthermore, the special approximating sequence $\{f_{k}\}$
contains a subsequence $\{f_{k_{l}}\}$ such that the sequence of $\mathcal{L}^{n}$-equivalence classes $\{[f_{k_{l}}]\}$ converges weakly in $W^{1}_{p}(\mathbb{R}^{n})$ to $[\operatorname{Ext}(f)]$ and there exists a constant $C > 0$ depending only on the parameters $p,d,n,\lambda,c$ and the constants
$\operatorname{C}_{\{\mathfrak{m}_{k}\},i}$, $i=1,2,3$ such that
\begin{equation}
\label{eqq.73}
\|[\operatorname{Ext}(f)]|W^{1}_{p}(\mathbb{R}^{n})\| \le  C \widetilde{\mathcal{N}}_{p,\{\mathfrak{m}_{k}\},\lambda,c}(f) \quad \text{for every} \quad f \in L_{1}(\{\mathfrak{m}_{k}\}).
\end{equation}
\end{Th}

\begin{proof}
It is clear that $f_{k} \in C^{\infty}(\mathbb{R}^{n})$ and $\operatorname{supp}f_{k} \subset 4Q_{0,0}$ for all $k \in \mathbb{N}$. Hence, the class $[f_{k}]$ belongs to $L_{1}(\mathbb{R}^{n})$ for all $k \in \mathbb{N}$.
Applying Proposition \ref{Prop.lp_admissible}, Theorem \ref{estimate.s's'small} and using \eqref{eq29'}, \eqref{eqq.simplest_imbedding}, \eqref{eqq.tr.sp.norm},
we deduce that $[f_{k}] \in W^{1}_{p}(\mathbb{R}^{n})$ for all $k \in \mathbb{N}$ and, furthermore,
\begin{equation}
\label{eqq.623''}
\|[f_{k}]|W_{p}^{1}(\mathbb{R}^{n})\| \le C \widetilde{\mathcal{N}}_{p,\{\mathfrak{m}_{k}\},\lambda,c}(f),
\end{equation}
where the constant $C > 0$ does not depend on $f$ and $k$.
Hence, the sequence $\{[f_{k}]\}$ is bounded in $W_{p}^{1}(\mathbb{R}^{n})$. By reflexivity of $W_{p}^{1}(\mathbb{R}^{n})$, this gives the
existence of a subsequence $\{[f_{k_{l}}]\}$ that converges weakly in $W_{p}^{1}(\mathbb{R}^{n})$ to some element $G \in W_{p}^{1}(\mathbb{R}^{n})$.
On the other hand, by Corollary \ref{Ca.convergence}, we clearly have $G=\operatorname{Ext}(f)$. Furthermore, using the standard
arguments from the theory of weakly convergent sequences in combination with \eqref{eqq.623''} we get the required estimate
\begin{equation}
\|[\operatorname{Ext}(f)]|W^{1}_{p}(\mathbb{R}^{n})\| \le \varliminf\limits_{l \to \infty}
\|[f_{k_{l}}]|W_{p}^{1}(\mathbb{R}^{n})\| \le C \widetilde{\mathcal{N}}_{p,\{\mathfrak{m}_{k}\},\lambda,c}(f).
\end{equation}
\end{proof}

\section{The direct trace theorem}

\textit{Throughout the section we fix the following data}:

{\rm (\textbf{D.7.1}) } a parameter $d \in (0,n)$ and \textit{a compact set} $S \subset Q_{0,0}$ with $\overline{\lambda}=\mathcal{H}^{d}_{\infty}(S) > 0;$

{\rm (\textbf{D.7.2}) } an arbitrary parameter $\lambda \in (0,\overline{\lambda})$;

{\rm (\textbf{D.7.3}) } a  sequence of measures $\{\mathfrak{m}_{k}\} \in \mathfrak{M}^{d}(S)$.

The aim of this section is the proof of the so-called direct trace theorem. In other words, we establish
that the trace functional $\widetilde{\mathcal{N}}_{q,\{\mathfrak{m}_{k}\},\lambda,c}$ is bounded on the $d$-trace space $W^{1}_{p}(\mathbb{R}^{n})|_{S}^{d}$ for each $p \in (\max\{1,n-d\},\infty)$
and any $q \in (1,n-d)$.

For the reader's convenience we recall again some notation introduced in the present paper earlier.
Given a cube $Q \subset \mathbb{R}^{n}$, we set $k_{Q}:=[-\log_{2}l(Q)]$.
Recall Definitions \ref{Def.shadow} and \ref{Def.iceberg}. Since the set $S$ and the parameters $d,\lambda$
are fixed during the section, we set $\mathcal{F}:=\mathcal{F}_{S}(d,\lambda)$,
$\mathcal{DF}:=\mathcal{DF}_{S}(d,\lambda)$, $\mathcal{A}:=\mathcal{A}_{S}(d,\lambda)$.
Furthermore, given a cube $Q \in \mathcal{DF}$ and a parameter $c \geq 1$, we set $\mathcal{K}_{c}(Q):=\mathcal{K}_{d,\lambda,c}(Q)$, $\mathcal{SH}_{c}(Q):=\mathcal{SH}_{d,\lambda,c}(Q)$, $\mathcal{IC}_{c}(Q):=\mathcal{IC}_{d,\lambda,c}(Q)$ and $T_{c}(Q):=T_{d,\lambda,c}(Q)$
for brevity. We recall \eqref{eqq.3.14} and put $\mathcal{P}(c):=\mathcal{P}_{S}(d,\lambda,c)$.
Since the sequence $\{\mathfrak{m}_{k}\}$ was also fixed we will write $\Phi_{f}(Q_{1},Q_{2})$ instead of $\Phi_{f,\{\mathfrak{m}_{k}\}}(Q_{1},Q_{2})$
for each $f \in L_{1}(\{\mathfrak{m}_{k}\})$ and any $Q_{1},Q_{2} \in \mathcal{D}_{+}$. Finally,
we put $f^{\natural}_{c}:=f^{\natural}_{\{\mathfrak{m}_{k}\},\lambda,c}$.

We  formulate the following useful technical estimate.

\begin{Lm}
\label{Lm8.1}
Let $\sigma \in (\max\{1,n-d\},n]$, $q \in (1,\infty)$, $c \geq 1$ and $\varepsilon > 0$. Then
there exists a constant $C=C(n,d,\lambda,\sigma,q,\varepsilon) > 0$ such that
\begin{equation}
\label{eq8.1}
\begin{split}
&\Phi_{f}(Q_{1},Q_{2}) \le C\Bigl(\fint\limits_{Q_{1}}\|\nabla F(x)\|^{\sigma}\,dx\Bigr)^{\frac{1}{\sigma}} +C\frac{l(Q_{2})}{l(Q_{1})}\Bigl(\fint\limits_{3cQ_{2}}\|\nabla F(x)\|^{\sigma}\,dx\Bigr)^{\frac{1}{\sigma}}\\
&+\frac{C}{(l(Q_{1}))^{1+\varepsilon}}\Bigl(\sum\limits_{i=1}^{N}(l(Q^{i}))^{q+q\varepsilon}\Bigl(\fint\limits_{Q^{i}}\|\nabla F(x)\|\,dx\Bigr)^{q}\Bigr)^{\frac{1}{q}},
\end{split}
\end{equation}
for any $F \in W_{\sigma}^{1}(\mathbb{R}^{n})$ with $f=\operatorname{Tr}|^{d}_{S}[F]$ and any
cubes $Q_{1},Q_{2} \in \mathcal{DF}$ such that $Q_{2} \in \mathcal{K}_{c}(Q_{1})$.
In \eqref{eq8.1} the family $\{Q^{i}\}_{i=0}^{N}$ is uniquely determined by the following conditions:

{\rm (1)} $\{Q^{i}\}_{i=0}^{N} \subset \mathcal{D}_{+}$;

{\rm (2)} $Q^{0}:=Q_{1} \subset ... \subset Q^{N}$;

{\rm (3)} $l(Q^{i+1})=2l(Q^{i})$ for every $i \in \{0,...,N-1\}$ and $l(Q^{N})=l(Q_{2})$.
\end{Lm}

\begin{proof}
Clearly $Q_{1} \in \mathcal{DF}_{k_{1}}$ and $Q_{2} \in \mathcal{DF}_{k_{2}}$ for some $k_{1},k_{2} \in \mathbb{N}_{0}$ with $k_{2} \le k_{1}$.
Using the triangle inequality several times, we obtain
\begin{equation}
\label{eq8.2}
\begin{split}
&l(Q_{1})\Phi_{f}(Q_{1},Q_{2})
\le \fint\limits_{Q_{1}}
\Bigl|f(z)-\fint\limits_{Q_{1}}F(x)\,dx\Bigr|\,d\mathfrak{m}_{k_{1}}(z)+
\Bigl|\fint\limits_{Q_{1}}F(x)\,dx-\fint\limits_{Q^{N}}F(w)\,dw\Bigr|\\
&
+\Bigl|\fint\limits_{Q^{N}}F(w)\,dw-\fint\limits_{Q_{2}}F(y)\,dy\Bigr|
+\fint\limits_{Q_{2}}
\Bigl|f(t)-\fint\limits_{Q_{2}}F(y)\,dy\Bigr|\,d\mathfrak{m}_{k_{2}}(t)=:\sum_{j=1}^{4}
l(Q_{1})\Phi_{j}.\\
\end{split}
\end{equation}
Since $f=F|^{d}_{S}$, $\sigma \in (\max\{1,n-d\},n]$ and $Q_{1},Q_{2} \in \mathcal{DF}$ by Theorem \ref{Th4.3}
\begin{equation}
\label{eq8.3}
\begin{split}
&\Phi_{1} \le C  \Bigl(\fint\limits_{Q_{1}}\|\nabla F(x)\|^{\sigma}\,dx\Bigr)^{\frac{1}{\sigma}}, \quad \Phi_{4} \le C \frac{l(Q_{2})}{l(Q_{1})}\Bigl(\fint\limits_{Q_{2}}\|\nabla F(x)\|^{\sigma}\,dx\Bigr)^{\frac{1}{\sigma}}.
\end{split}
\end{equation}
Note that $Q_{1} \subset cQ_{2}$ by $(\textbf{C}1)$ of Definition \ref{cover.cube}. Hence, Using Proposition \ref{Lm4.1} with $c'=c$ and then applying H\"older's inequality, we get
\begin{equation}
\label{eq8.4}
\Phi_{3} \le C \frac{l(Q_{2})}{l(Q_{1})}
\fint\limits_{3cQ_{2}}\|\nabla F(x)\|\,dx \le C \frac{l(Q_{2})}{l(Q_{1})}
\Bigl(\fint\limits_{3cQ_{2}}\|\nabla F(x)\|^{\sigma}\,dx\Bigr)^{\frac{1}{\sigma}}.
\end{equation}
To estimate $\Phi_{2}$, we apply Proposition
\ref{Lm4.1} and then use H\"older's inequality for sums. This gives
\begin{equation}
\begin{split}
\label{eq8.5}
&l(Q_{1})\Phi_{2} \le \sum\limits_{i=0}^{N-1}\Bigl|\fint\limits_{Q^{i}}F(x)\,dx-\fint\limits_{Q^{i+1}}F(y)\,dy\Bigr|
\le C \sum\limits_{i=1}^{N}\fint\limits_{Q^{i}}\fint\limits_{Q^{i}}|F(x)-F(y)|\,dx\,dy\\
&\le C \sum\limits_{i=1}^{N}l(Q^{i})\fint\limits_{Q^{i}}\|\nabla F(\tau)\|\,d\tau=
C\Bigl(\sum\limits_{i=1}^{N}\frac{(l(Q^{i}))^{\varepsilon+1}}{(l(Q^{i}))^{\varepsilon}}
\fint\limits_{Q^{i}}\|\nabla F(\tau)\|\,d\tau\Bigr)^{\frac{q}{q}}\\
&\le \frac{C}{(l(Q^{0}))^{\varepsilon}}\Bigl(\sum\limits_{i=1}^{N}(l(Q^{i}))^{q+q\varepsilon}
\Bigl(\fint\limits_{Q^{i}}\|\nabla F(x)\|\,dx\Bigr)^{q}\Bigr)^{\frac{1}{q}}.
\end{split}
\end{equation}
Combining estimates \eqref{eq8.2}--\eqref{eq8.5} we obtain \eqref{eq8.1} and complete the proof.
\end{proof}

Given a function $f \in L_{1}(\{\mathfrak{m}_{k}\})$ and a constant $c>0$, we define for each $t > 0$, \textit{the superlevel
set} of $f^{\natural}_{c}$ by letting
\begin{equation}
\label{eq71}
\mathcal{U}_{c,t}(f):=\{x \in \mathbb{R}^{n}: f^{\natural}_{c}(x) >  t\}.
\end{equation}

\begin{Def}
\label{Def8.1}
Let $f \in L_{1}(\{\mathfrak{m}_{k}\})$ and $c \geq 1$. For each $t > 0$
we define \textit{the good part} $\mathcal{U}^{g}_{c,t}(f)$ of the set $\mathcal{U}_{c,t}(f)$ as the set of all points $x \in \mathbb{R}^{n}$ for each of which there exist
cubes  $\underline{Q}(x), \overline{Q}(x)$ satisfying the following conditions:

{ \rm (1)} $\underline{Q}(x) \in T_{c}(x)$ and $\overline{Q}(x) \in \mathcal{DF}$;

{ \rm (2)} $\overline{Q}(x) \in \mathcal{K}_{c}(\underline{Q}(x))$;

{ \rm (3)} $\Phi_{f}(\underline{Q}(x),\overline{Q}(x)) > t$;

{ \rm (4)} $l(\underline{Q}(x)) \geq 2^{-1}l(\overline{Q}(x))$.

We define \textit{the bad part} $\mathcal{U}^{b}_{c,t}(f)$ of the set $\mathcal{U}_{c,t}(f)$ by letting $\mathcal{U}^{b}_{c,t}(f):=\mathcal{U}_{c,t}(f) \setminus \mathcal{U}^{g}_{c,t}(f).$
Finally, we put $\mathcal{U}^{g}_{c}(f):=\cup_{t > 0}\mathcal{U}^{g}_{c,t}(f)$, $\mathcal{U}^{b}_{c}(f):=\cup_{t > 0}\mathcal{U}^{b}_{c,t}(f)$.
\end{Def}

\begin{Remark}
\label{Rem.81}
Given $f \in L_{1}(\{\mathfrak{m}_{k}\})$ and $c \geq 1$, it is clear that $\{x: f^{\natural}_{c}(x) > 0\}=\mathcal{U}^{g}_{c}(f) \cup \mathcal{U}^{b}_{c}(f).$
\end{Remark}

The following characteristic property is an immediate consequence of Definition \ref{Def8.1}.

\begin{Prop}
\label{Propos.8.1}
Let $f \in L_{1}(\{\mathfrak{m}_{k}\})$ and $c \geq 1$. For each $t > 0$, a point $x \in
\mathcal{U}^{b}_{c,t}(f)$ if and only if $x \in \mathcal{U}_{c,t}(f)$ and
\begin{equation}
\notag
l(\underline{Q}(x)) \le \frac{1}{4}l(\overline{Q}(x))
\end{equation}
for any pair of cubes  $\underline{Q}(x), \overline{Q}(x)$ satisfying the following conditions:

{ \rm (1)} $\underline{Q}(x) \in T_{c}(x)$ and $\overline{Q}(x) \in \mathcal{DF}$;

{ \rm (2)} $\overline{Q}(x) \in \mathcal{K}_{c}(\underline{Q}(x))$;

{ \rm (3)} $\Phi_{f}(\underline{Q}(x),\overline{Q}(x)) > t$.
\end{Prop}

Given a parameter $\sigma \in [1,\infty)$ and an element $F \in W_{\sigma}^{1}(\mathbb{R}^{n})$, we define, for each $t > 0$,
\begin{equation}
\notag
\mathcal{V}_{\sigma,t}(F):=\{x \in \mathbb{R}^{n}: \mathcal{M}_{\sigma}[\|\nabla F\|](x) > t\}.
\end{equation}

\begin{Lm}
\label{Lm7.2}
Let $\sigma \in (\max\{1,n-d\},\infty)$ and $c \geq 1$. Let $F \in W^{1}_{\sigma}(\mathbb{R}^{n})$ and $f:=\operatorname{Tr}|_{S}^{d}F$. Then
there exists a constant $C=C(n,d,\lambda,\sigma,c) > 0$ such that
\begin{equation}
\label{eq7.8}
\mathcal{U}^{g}_{c,t}(f) \subset \mathcal{V}_{\sigma,\frac{t}{C}}(F) \quad \hbox{for every} \quad t > 0.
\end{equation}
In particular, for each $p \in (\sigma,\infty)$, there exists a constant $C'=C'(n,d,\lambda,p,\sigma,c) > 0$ such that
\begin{equation}
\label{eq7.9}
\int\limits_{\mathcal{U}^{g}_{c}(f)}\Bigl(f^{\natural}_{c}(x)\Bigr)^{p}\,dx
\le C' \int\limits_{\mathbb{R}^{n}}\|\nabla F(x)\|^{p}\,dx.
\end{equation}
\end{Lm}

\begin{proof}
We fix a number $t > 0$. To prove the first claim we recall Definition \ref{Def8.1} and find cubes
$\underline{Q}(x) \in T_{c}(x)$, $\overline{Q}(x) \in \mathcal{DF}$ such that $\overline{Q}(x) \in \mathcal{K}_{c}(\underline{Q}(x))$, $l(\underline{Q}) \geq \frac{1}{2}l(\overline{Q}(x))$
and
\begin{equation}
\label{eqq.8.19}
\Phi_{f}(\underline{Q}(x),\overline{Q}(x)) > t.
\end{equation}
By Lemma \ref{Lm8.1} and \eqref{eq2.2} it is easy to see that
\begin{equation}
\label{eqq.8.20}
\Phi_{f}(\underline{Q}(x),\overline{Q}(x)) \le C \Bigl(\fint\limits_{3c\overline{Q}(x)}\|\nabla F (y)\|^{\sigma}\,dy\Bigr)^{\frac{1}{\sigma}}
\le C\mathcal{M}_{\sigma}[\|\nabla F\|](x).
\end{equation}
Combining \eqref{eqq.8.19} and \eqref{eqq.8.20} we get \eqref{eq7.8}.

To prove \eqref{eq7.9} we use \eqref{eq7.8} and then apply Proposition \ref{Prop.max.funct}. We have
\begin{equation}
\label{eq7.11}
\begin{split}
&\int\limits_{\mathcal{U}^{g}_{c}(f)}\Bigl(f^{\natural}(x)\Bigr)^{p}\,dx
=p\int\limits_{0}^{\infty}t^{p-1}\mathcal{L}^{n}(\mathcal{U}^{g}_{c,t}(f))\,dt \le p\int\limits_{0}^{\infty}t^{p-1}\mathcal{L}^{n}(\mathcal{V}_{\sigma,\frac{t}{C}}(F))\,dt\\
&\le C \int\limits_{0}^{\infty}t^{p-1}\mathcal{L}^{n}(\mathcal{V}_{\sigma,t}(F))\,dt=C\int\limits_{\mathbb{R}^{n}}
\Bigl(\mathcal{M}_{\sigma}[\|\nabla F\|](x)\Bigr)^{p}\,dx \le C \int\limits_{\mathbb{R}^{n}}\|\nabla F(x)\|^{p}\,dx.
\end{split}
\end{equation}
The proof is complete.
\end{proof}

We introduce some notation which will be useful below. Recall Proposition \ref{Prop.3.1}.
Given a parameter $c \geq 1$ and a cube  $Q \in \mathcal{DF}$, we
set
\begin{equation}
\notag
\mu_{c}(Q):=\inf \{l(Q'): Q' \in \mathcal{SH}_{c}(Q)\},
\quad \operatorname{N}_{c}(Q):=\log_{2}\Bigl(\frac{l(Q)}{\mu_{c}(Q)}\Bigr) \in \mathbb{N}_{0} \cup \{+\infty\}.
\end{equation}
Given a parameter $c \geq 1$ and a cube $Q \in \mathcal{DF}$, for each $j \in \mathbb{N}_{0} \cap [0,\operatorname{N}_{c}(Q))$ we define
\textit{the $j$th layer of the iceberg} $\mathcal{IC}_{c}(Q)$ by
\begin{equation}
\notag
\operatorname{L}^{j}_{c}(Q):=\{Q' \in \mathcal{IC}_{c}(Q):l(Q')=2^{-j}l(Q)\}.
\end{equation}
Furthermore, it will be convenient to put formally $\operatorname{L}^{\operatorname{N}_{c}(Q)}_{c}(Q):=\emptyset$.

\begin{Remark}
\label{Remm82}
It is easy to verify that the following properties of the sets $\operatorname{L}^{j}_{c}(Q)$:

{\rm (1)} $\operatorname{L}^{j}_{c}(Q) \cap \operatorname{L}^{j'}_{c}(Q) = \emptyset$ for $j \neq j'$;

{\rm (2)} $\mathcal{IC}_{c}(Q)=\cup_{j=0}^{\operatorname{N}_{c}(Q)}\operatorname{L}^{j}_{c}(Q)$.
\end{Remark}

\hfill$\Box$

The following assertion is a technical heart of this section.  We recall Remark \ref{Rem.iceberg_intuition}.

\begin{Lm}
\label{KeyLemma}
Suppose we are given a number $c \geq 1$
and a selection $\kappa_{c}$ of $\mathcal{K}_{c}$ with a domain $\mathfrak{D} \subset \mathcal{DF}$ such that:

{\rm (1)} $\mathcal{K}_{c}(Q) \cap \mathcal{P}(c) \neq \emptyset \quad \hbox{for all} \quad Q \in  \mathfrak{D}$;

{\rm (2)}  $\kappa_{c}(Q) \in \mathcal{P}(c)$
and $l(\kappa_{c}(Q)) > l(Q)$ for all $Q \in \mathfrak{D}$.

Then, for each $p > \max\{1,n-d\}$, $q \in (1,n-d)$ and $\tau \in (1,\frac{n-d}{q})$, there
exists a constant $C > 0$ depending only on $n,d,\lambda,p,q,\tau,c$ such that the following inequality
\begin{equation}
\label{eq.8.8}
R_{p,q,\tau}[g]:=\sum\limits_{Q \in \mathfrak{D}}(l(Q))^{n-q\tau}\sum\limits_{\substack{Q' \supset Q \\Q' \in \mathcal{IC}_{c}(\kappa_{c}(Q))}}
(l(Q'))^{q\tau}\Bigl(\fint\limits_{Q'}g(x)\,dx\Bigr)^{q} \le C \Bigl(\int\limits_{\mathbb{R}^{n}}g^{p}(x)\,dx\Bigr)^{\frac{q}{p}}
\end{equation}
holds for any nonnegative function $g \in L_{p}(\mathbb{R}^{n})$.
\end{Lm}

\begin{proof}
We fix parameters $p,q,\tau$ satisfying the assumptions of the lemma. We also fix a parameter
$\sigma \in (q,p)$
and an arbitrary nonnegative
$g \in L_{p}(\mathbb{R}^{n})$. We split the proof into several steps.

\textit{Step 1.} Since $l(\kappa_{c}(Q)) > l(Q)$ for all $Q \in \mathfrak{D} \subset \mathcal{DF}$ by the assumption (2) of the lemma
it follows from Definitions \ref{Def.shadow}, \ref{Def.iceberg} that if $Q \in \mathfrak{D}$,
$Q' \in \mathcal{IC}_{c}(\kappa_{c}(Q))$ and  $Q' \supset Q$ then $Q \in \mathcal{SH}_{c}(\kappa_{c}(Q))|_{Q'}$.
Hence, we change the order of summation in the definition of $R_{p,q,\tau}[g]$, use Remark \ref{Remm82},
and take into account that $\kappa_{c}(Q) \in \mathcal{P}(c)$ for all $Q \in \mathfrak{D}$
by the assumption (2) of the lemma. This gives
\begin{equation}
\label{eq.8.9}
\begin{split}
&R_{p,q,\tau}[g] \le \sum\limits_{Q \in \mathcal{P}(c)}
\sum\limits_{j=0}^{\operatorname{N}_{c}(Q)}\sum\limits_{Q' \in
\operatorname{L}^{j}_{c}(Q)} t_{Q}(Q') \Bigl(\fint\limits_{Q'}g(x)\,dx\Bigr)^{q},
\end{split}
\end{equation}
where, for each $Q \in \mathcal{P}(c)$, every number $j \in \mathbb{N}_{0} \cap [0,\operatorname{N}_{c}(Q))$
and any cube $Q' \in \operatorname{L}^{j}_{c}(Q)$, we put
\begin{equation}
\label{eqq.8.8}
t_{Q}(Q'):= (l(Q'))^{q\tau}\Bigl(\sum\limits_{\substack{Q'' \in \mathcal{SH}_{c}(Q)|_{Q'}}}(l(Q''))^{n-q\tau}\Bigr),
\end{equation}
and in the case $\operatorname{N}_{c}(Q) < +\infty$ we formally put $t_{Q}(Q')=0$ for all $Q' \in \operatorname{L}^{\operatorname{N}_{c}(Q)}_{c}(Q)$
(recall that $\operatorname{L}^{\operatorname{N}_{c}(Q)}_{c}(Q)=\emptyset$ in this case).

\textit{Step 2.}
Note that $n-q\tau > d$ by the assumptions of the lemma. Hence, using \eqref{eq3.15} with $\widetilde{d}=n-q\tau$ we get
\begin{equation}
\label{eq.8.13}
t_{Q}(Q')  \le C (l(Q'))^{n}.
\end{equation}
Furthermore, using \eqref{eq3.15'''} with $\widetilde{d}=n-q\tau$ we obtain, for each $Q \in \mathcal{P}(c)$ and every $j \in \mathbb{N}_{0} \cap [0,\operatorname{N}_{c}(Q))$,
\begin{equation}
\label{eq.8.12}
\begin{split}
&\sum\limits_{Q' \in \operatorname{L}^{j}_{c}(Q)}t_{Q}(Q') = \sum\limits_{Q' \in \operatorname{L}^{j}_{c}(Q)}(l(Q'))^{q\tau}\Bigl(\sum\limits_{\substack{Q'' \in \mathcal{SH}_{c}(Q)|_{Q'} }}(l(Q''))^{n-q\tau}\Bigr) \\
&\le \Bigl(\frac{l(Q)}{2^{j}}\Bigr)^{q\tau}
\sum\limits_{Q' \in \operatorname{L}^{j}_{c}(Q)} \sum\limits_{\substack{Q'' \in \mathcal{SH}_{c}(Q)|_{Q'}}}(l(Q''))^{n-q\tau}\\
&\le \Bigl(\frac{l(Q)}{2^{j}}\Bigr)^{q\tau}
\sum\limits_{\substack{Q'' \in \mathcal{SH}_{c}(Q)}}(l(Q''))^{n-q\tau} \le C\Bigl(\frac{l(Q)}{2^{j}}\Bigr)^{q\tau}(l(Q))^{n-q\tau} \le \frac{C}{2^{jq\tau}}(l(Q))^{n}.
\end{split}
\end{equation}

\textit{Step 3.}
An application of H\"older's inequality for sums with exponents
$\frac{\sigma}{q}$ and $\bigl(\frac{\sigma}{q}\bigr)'=\frac{\sigma}{\sigma-q}$ gives, for each cube $Q \in \mathcal{P}(c)$ and every $j \in \mathbb{N}_{0} \cap [0,\operatorname{N}_{c}(Q))$,
\begin{equation}
\label{eq.8.11}
\begin{split}
&\sum\limits_{Q' \in \operatorname{L}^{j}_{c}(Q)}\Bigl(t_{Q}(Q')\Bigr)^{\frac{q}{\sigma}+(\frac{q}{\sigma})'}\Bigl(\fint\limits_{Q'}g(x)\,dx\Bigr)^{q}\\
&\le C\Bigl(\sum\limits_{Q' \in \operatorname{L}^{j}_{c}(Q)}t_{Q}(Q')\Bigr)^{\frac{\sigma-q}{\sigma}}\Bigl(\sum\limits_{Q' \in \operatorname{L}^{j}_{c}(Q)}t_{Q}(Q')\Bigl(\fint\limits_{Q'}g(x)\,dx\Bigr)^{\sigma}\Bigr)^{\frac{q}{\sigma}}.\\
\end{split}
\end{equation}

\textit{Step 4.} Using \eqref{eq.8.13} and applying Proposition \ref{Prop2.2} with $\Omega=\underline{Q}=Q'$, we have
\begin{equation}
\label{eq.8.14}
\begin{split}
&\sum\limits_{Q' \in \operatorname{L}^{j}_{c}(Q)} t_{Q}(Q') \Bigl(\fint\limits_{Q'} g(x)\,dx\Bigr)^{\sigma}
\le C \sum\limits_{Q' \in \operatorname{L}^{j}_{c}(Q)}(l(Q'))^{n}\Bigl(\fint\limits_{Q'}g(x)\,dx\Bigr)^{\sigma}\\
&\le C\sum\limits_{Q' \in \operatorname{L}^{j}_{c}(Q)}\int\limits_{Q'} \Bigl(\mathcal{M}[g](x)\Bigr)^{\sigma}\,dx.
\end{split}
\end{equation}

\textit{Step 5.} We set $\theta:=q\tau\frac{\sigma-q}{\sigma} > 0.$
Now we plug \eqref{eq.8.14} and \eqref{eq.8.12} into \eqref{eq.8.11}.
This gives us, for each $Q \in \mathcal{P}(c)$ and any $j \in \mathbb{N}_{0} \cap [0,\operatorname{N}_{c}(Q))$,
\begin{equation}
\label{eqq.720}
\sum\limits_{Q' \in \operatorname{L}^{j}_{c}(Q)}t_{Q}(Q')\Bigl(\fint\limits_{Q'}g(x)\,dx\Bigr)^{q} \le
\frac{C}{2^{j\theta}}(l(Q))^{n}\Bigl(\fint\limits_{cQ}\Bigl(\mathcal{M}[g](x)\Bigr)^{\sigma}\,dx\Bigr)^{\frac{q}{\sigma}}.
\end{equation}

Let $\underline{\varkappa}=\underline{\varkappa}(n,d,\lambda,c)$ be the same as in Theorem \ref{Th.spec.cavit.}. We put $\varkappa:=\underline{\varkappa}+3.$
We recall \eqref{eqq.3.14} and apply Theorem \ref{Th.spec.cavit.}. This gives
\begin{equation}
\label{eqq.721}
\mathcal{L}^{n}(\Omega_{c,\varkappa}(Q)) \geq C (l(Q))^{n} \quad \text{for each cube} \quad Q \in \mathcal{P}(c).
\end{equation}
Finally, for each $Q \in \mathcal{P}(c)$  and every $j \in \mathbb{N}_{0} \cap [0,\operatorname{N}_{c}(Q))$ we combine \eqref{eqq.720}, \eqref{eqq.721} and apply
Proposition \ref{Prop2.2} with $\Omega=\Omega_{c,\varkappa}(Q)$, $\underline{Q}=Q$. We obtain
\begin{equation}
\label{eq.8.15}
\begin{split}
&\sum\limits_{Q' \in \operatorname{L}^{j}_{c}(Q)}t_{Q}(Q')\Bigl(\fint\limits_{Q'}g(x)\,dx\Bigr)^{q} \le  \frac{C}{2^{j\theta}} \mathcal{L}^{n}(\Omega_{c,\varkappa}(Q))\Bigl(\fint\limits_{cQ}\Bigl(\mathcal{M}[g](x)\Bigr)^{\sigma}\,dx\Bigr)^{\frac{q}{\sigma}} \\
& \le
\frac{C}{2^{j\theta}}\int\limits_{\Omega_{c,\varkappa}(Q)}
\Bigl(\mathcal{M}_{\sigma}[\mathcal{M}[g]](x)\Bigr)^{q}\,dx \quad \text{for each cube} \quad Q \in \mathcal{P}(c) \quad \text{and every} \quad j \in \mathbb{N}_{0} \cap [0,\operatorname{N}_{c}(Q)).
\end{split}
\end{equation}

\textit{Step 6.}
We plug \eqref{eq.8.15} into \eqref{eq.8.9} and take into account that $\theta > 0$. This gives
\begin{equation}
\notag
\begin{split}
&R_{p,q,\tau}[g] \le C\sum\limits_{Q \in \mathcal{P}(c)}
\sum\limits_{j=0}^{\operatorname{N}_{c}(Q)}\frac{1}{2^{j\theta}}\int\limits_{\Omega_{c,\varkappa}(Q)}\Bigl(
\mathcal{M}_{\sigma}[\mathcal{M}[g]](x)\Bigr)^{q}\,dx\\
&\le C\sum\limits_{Q \in \mathcal{P}(c)}\int\limits_{\Omega_{c,\varkappa}(Q)}\Bigl(
\mathcal{M}_{\sigma}[\mathcal{M}[g]](x)\Bigr)^{q}\,dx.
\end{split}
\end{equation}
It is clear that $\Omega_{c,\varkappa}(Q) \subset cQ_{0,0}$ for all $Q \in \mathcal{P}(c)$.
Furthermore, by Proposition \ref{Th.intersect.cavities} we have $M(\{\Omega_{c,\varkappa}(Q):Q \in \mathcal{P}(c)\}) \le C$ with a constant $C > 0$
depending only on $n,d,\lambda,c$. Using these
observations and applying Proposition \ref{Prop2.4} we continue the previous estimate and get
\begin{equation}
\label{eq.8.16}
\begin{split}
&R_{p,q,\tau}[g] \le  C \int\limits_{cQ_{0,0}}\Bigl(
\mathcal{M}_{\sigma}[\mathcal{M}[g]](x)\Bigr)^{q}\,dx.
\end{split}
\end{equation}

\textit{Step 7.}
Finally, we use H\"older's inequality for integrals with exponents $\frac{p}{q}$, $\frac{p}{p-q}$. Then we apply Proposition \ref{Prop.max.funct}
twice. This allows us to continue \eqref{eq.8.16} and deduce
\begin{equation}
\label{eq.8.17}
\begin{split}
&R_{p,q,\tau}[g] \le C \Bigl(\int\limits_{cQ_{0,0}}\Bigl(
\mathcal{M}_{\sigma}[\mathcal{M}[g]](x)\Bigr)^{p}\,dx\Bigr)^{\frac{q}{p}}\\
&\le
C \Bigl(\int\limits_{cQ_{0,0}}\Bigl(
\mathcal{M}[g](x)\Bigr)^{p}\,dx\Bigr)^{\frac{q}{p}} \le
C \Bigl(\int\limits_{\mathbb{R}^{n}}g^{p}(x)\,dx\Bigr)^{\frac{q}{p}}.
\end{split}
\end{equation}

The lemma is proved.
\end{proof}

Now we formulate the \textit{key lemma}, which will be the cornerstone in proving the main result of this
section.

\begin{Lm}
\label{KeyLemma2}
Let $p \in (\max\{1,n-d\},n]$, $q \in (1,n-d)$ and $c \geq 1$. Let a selection $\kappa_{c}$ of $\mathcal{K}_{c}$ with a domain $\mathfrak{D} \subset \mathcal{DF}$
be such that:

{\rm (1)} $\mathcal{K}_{c}(Q) \cap \mathcal{P}(c) \neq \emptyset \quad \hbox{for all} \quad Q \in  \mathfrak{D}$;

{\rm (2)} $\kappa_{c}(Q) \in \mathcal{P}(c)$ and $l(\kappa_{c}(Q)) > l(Q)$ for all $Q \in \mathfrak{D}$.

Then there exists a constant $C > 0$ depending only on $p,q,n,d,\lambda,c$ such that
\begin{equation}
\label{eq.8.21}
\sum\limits_{Q \in \mathfrak{D}}(l(Q))^{n}\Bigl(\Phi_{f}(Q,\kappa_{c}(Q))\Bigr)^{q} \le C \Bigl(\sum\limits_{|\gamma|=1}\|D^{\gamma}F|L_{p}(\mathbb{R}^{n})\|\Bigr)^{q}
\end{equation}
for any $F \in W_{p}^{1}(\mathbb{R}^{n})$ with $f = \operatorname{Tr}|_{S}^{d}[F]$.
\end{Lm}

\begin{proof}
We fix an arbitrary element $F \in W_{p}^{1}(\mathbb{R}^{n})$ and set $f = \operatorname{Tr}|_{S}^{d}[F]$.
Furthermore, we fix a parameter $\sigma \in (\max\{1,n-d\},p).$
By the assumptions of the lemma and Definitions \ref{Def.selector}, \ref{Def.shadow}, it is clear that
\begin{equation}
\label{eqq824}
\kappa_{c}^{-1}(\overline{Q}) \subset \mathcal{SH}_{c}(\overline{Q}) \quad \hbox{for each} \quad
\overline{Q} \in \mathcal{P}(c).
\end{equation}
We split the proof into several steps.

\textit{Step 1.} By Remark \ref{Rem.iceberg_intuition}, we see that if $Q \in \mathfrak{D}$, $Q \subset Q'$ and $l(Q') \in (l(Q),l(\kappa_{c}(Q))]$, then $Q' \in \mathcal{IC}_{c}(\kappa_{c}(Q))$. Hence, an application of Lemma \ref{Lm8.1}
gives, for any sufficiently small $\varepsilon > 0$ to be specified later,
\begin{equation}
\label{eq.8.22}
\begin{split}
&\sum\limits_{Q \in \mathfrak{D}}(l(Q))^{n}(\Phi_{f}(Q,\kappa_{c}(Q))^{q}
\le C \sum\limits_{Q \in \mathfrak{D}}(l(Q))^{n}\Bigl(\fint\limits_{Q}\|\nabla F(x)\|^{\sigma}\,dx\Bigr)^{\frac{q}{\sigma}} \\
&+C\sum\limits_{Q \in \mathfrak{D}}(l(Q))^{n-q} (l(\kappa_{c}(Q)))^{q}\Bigl(\fint\limits_{3c\kappa_{c}(Q)}\|\nabla F(y)\|^{\sigma}\,dy\Bigr)^{\frac{q}{\sigma}}\\
&+C\sum\limits_{Q \in \mathfrak{D}}(l(Q))^{n-q-q\varepsilon}\sum\limits_{\substack{Q' \supset Q\\Q' \in \mathcal{IC}_{c}(\kappa_{c}(Q))}}
(l(Q'))^{q+q\varepsilon}\Bigl(\fint\limits_{Q'}\|\nabla F(z)\|\,dz\Bigr)^{q}=:R^{1}+R^{2}+R^{3}.
\end{split}
\end{equation}

\textit{Step 2.} Let $\underline{\varkappa}=\underline{\varkappa}(n,d,\lambda,c)$ be the same as in Theorem \ref{Th.spec.cavit.}. We put $\varkappa:=\underline{\varkappa}+3.$
An application of Theorem \ref{Th.spec.cavit.} gives
\begin{equation}
\label{eqq.721''''}
\mathcal{L}^{n}(\Omega_{c,\varkappa}(\overline{Q})) \geq C (l(\overline{Q}))^{n} \quad \text{for each cube} \quad \overline{Q} \in \mathcal{P}(c).
\end{equation}
On the other hand, given a cube $\overline{Q} \in \mathcal{P}(c)$ with $\kappa^{-1}_{c}(\overline{Q}) \neq \emptyset$,
we use \eqref{eqq824}, apply Proposition \ref{Prop.3.1}
and then apply Theorem \ref{Th.spec.cavit.}. This gives
\begin{equation}
\label{eqq8.24}
\sum\limits_{\underline{Q} \in \kappa^{-1}_{c}(\overline{Q})} (l(\underline{Q}))^{n} \le  \sum\limits_{\underline{Q} \in \mathcal{SH}_{c}(\overline{Q})} (l(\underline{Q}))^{n} \le(l(c\overline{Q}))^{n}
\le C \mathcal{L}^{n}(\Omega_{c,\varkappa}(\overline{Q})).
\end{equation}

\textit{Step 3.}
Now we use H\"older's inequality for sums with exponents $\frac{\sigma}{q}$ and $\frac{\sigma}{\sigma-q}$. Then we take into account \eqref{eqq824}, \eqref{eqq.721''''} and \eqref{eqq8.24}. We obtain (we also
use the obvious inclusion $\Omega_{c,\varkappa}(\overline{Q}) \subset c\overline{Q}$)
\begin{equation}
\label{eq.8.23}
\begin{split}
&\sum\limits_{\underline{Q} \in \kappa^{-1}_{c}(\overline{Q})}(l(\underline{Q}))^{n}\Bigl(\fint\limits_{\underline{Q}}\|\nabla F(x)\|^{\sigma}\,dx\Bigr)^{\frac{q}{\sigma}}
=\sum\limits_{\underline{Q} \in \kappa^{-1}_{c}(\overline{Q})}(l(\underline{Q}))^{n(1-\frac{q}{\sigma})}\Bigl(\int\limits_{\underline{Q}}\|\nabla F(x)\|^{\sigma}\,dx\Bigr)^{\frac{q}{\sigma}}\\
&\le C \Bigl(\mathcal{L}^{n}(\Omega_{\varkappa,c}(\overline{Q}))\Bigr)^{1-\frac{q}{\sigma}}
\Bigl(\int\limits_{c\overline{Q}}
\|\nabla F(x)\|^{\sigma}\,dx\Bigr)^{\frac{q}{\sigma}} \le C \mathcal{L}^{n}(\Omega_{c,\varkappa}(\overline{Q}))\Bigl(\fint\limits_{c\overline{Q}}
\|\nabla F(x)\|^{\sigma}\,dx\Bigr)^{\frac{q}{\sigma}}.
\end{split}
\end{equation}

\textit{Step 4.} To estimate $R_{1}$ from above, we use \eqref{eq.8.23} and apply Proposition \ref{Prop2.2} with $\Omega=\Omega_{c,\varkappa}(\overline{Q})$. This gives
\begin{equation}
\label{eq.8.24}
\begin{split}
&R^{1} \le \sum\limits_{\overline{Q} \in \mathcal{P}(c)}
\sum\limits_{\underline{Q} \in \kappa^{-1}_{c}(\overline{Q})}(l(\underline{Q}))^{n}\Bigl(\fint\limits_{\underline{Q}}\|\nabla F(x)\|^{\sigma}\,dx\Bigr)^{\frac{q}{\sigma}}\\
&\le C \sum\limits_{\overline{Q} \in \mathcal{P}(c)}\mathcal{L}^{n}(\Omega_{c,\varkappa}(\overline{Q}))\Bigl(\fint\limits_{c\overline{Q}}
\|\nabla F(x)\|^{\sigma}\,dx\Bigr)^{\frac{q}{\sigma}} \le C \sum\limits_{\overline{Q} \in \mathcal{P}(c)}
\int\limits_{\Omega_{c,\varkappa}(\overline{Q})}\Bigl(\mathcal{M}_{\sigma}[\|\nabla F\|](x)\Bigr)^{q}\,dx.\\
\end{split}
\end{equation}
To continue \eqref{eq.8.24}, we combine Proposition \ref{Prop2.4} with Proposition \ref{Th.intersect.cavities}, apply
H\"older's inequality for integrals with exponents $\frac{p}{q}$, $\frac{p}{p-q}$, and finally use Proposition
\ref{Prop.max.funct}. We get
\begin{equation}
\label{eq.8.25}
\begin{split}
&(R^{1})^{\frac{p}{q}} \le C \Bigl(\int\limits_{cQ_{0,0}}\Bigl(\mathcal{M}_{\sigma}[\|\nabla F\|](x)\Bigr)^{q}\,dx\Bigr)^{\frac{p}{q}}\\
&\le C \int\limits_{cQ_{0,0}}\Bigl(\mathcal{M}_{\sigma}[\|\nabla F\|](x)\Bigr)^{p}\,dx
\le C \int\limits_{cQ_{0,0}}\|\nabla F(x)\|^{p}\,dx.
\end{split}
\end{equation}

\textit{Step 5.}
For each $\overline{Q} \in \mathcal{P}(c)$ we apply Proposition \ref{Lm.shadow} with $\widetilde{d}=n-q > d$, then we use
\eqref{eqq.721''''} and apply Proposition \ref{Prop2.2} with $\Omega=\Omega_{c,\varkappa}(\overline{Q})$. This gives
\begin{equation}
\label{eq.8.26}
\begin{split}
&\Bigl(\sum\limits_{\underline{Q} \in \mathcal{SH}_{c}(\overline{Q})}(l(\underline{Q}))^{n-q}\Bigr)(l(\overline{Q}))^{q}\Bigl(\fint\limits_{3c \overline{Q}}
\|\nabla F(y)\|^{\sigma}\,dy\Bigr)^{\frac{q}{\sigma}} \le  C (l(\overline{Q}))^{n}\Bigl(\fint\limits_{3c \overline{Q}}\|\nabla F(y)\|^{\sigma}\,dy\Bigr)^{\frac{q}{\sigma}}\\
&\le C \mathcal{L}^{n}(\Omega_{c,\varkappa}(\overline{Q}))\Bigl(\fint\limits_{3c \overline{Q}}\|\nabla F(y)\|^{\sigma}\,dy\Bigr)^{\frac{q}{\sigma}} \le C \int\limits_{\Omega_{c,\varkappa}(\overline{Q})}\Bigl(\mathcal{M}_{\sigma}[\|\nabla F\|](x)\Bigr)^{q}\,dx.
\end{split}
\end{equation}
By \eqref{eqq824}, \eqref{eq.8.26} and Proposition \ref{Th.intersect.cavities} we obtain
\begin{equation}
\label{eq.8.27}
\begin{split}
&R^{2} \le \sum\limits_{\overline{Q} \in \mathcal{P}(c)}
\Bigl(\sum\limits_{\underline{Q} \in \mathcal{SH}_{c}(\overline{Q})}(l(\underline{Q}))^{n-q}\Bigr)(l(\overline{Q}))^{q}\Bigl(\fint\limits_{3c \overline{Q}}\|\nabla F(y)\|^{\sigma}\,dy\Bigr)^{\frac{q}{\sigma}}\\
&\le C \sum\limits_{\overline{Q} \in \mathcal{P}(c)}
\int\limits_{\Omega_{c,\varkappa}(\overline{Q})}\Bigl(\mathcal{M}_{\sigma}[\|\nabla F\|](x)\Bigr)^{q}\,dx \le \int\limits_{cQ_{0,0}}\Bigl(\mathcal{M}_{\sigma}
[\|\nabla F\|](x)\Bigr)^{q}\,dx.
\end{split}
\end{equation}
The same arguments as in \eqref{eq.8.25} allow us continue \eqref{eq.8.27} and get
\begin{equation}
\label{eq.8.28}
R^{2} \le C \Bigl(\int\limits_{cQ_{0,0}}\|\nabla F(x)\|^{p}\,dx\Bigr)^{\frac{q}{p}}.
\end{equation}

\textit{Step 6.} We fix an arbitrary $\varepsilon > 0$ such that $\tau:=(1+\varepsilon) \in (1,\frac{n-d}{q})$. In this case we can apply Lemma \ref{KeyLemma}
with $g=\|\nabla F\|$ and deduce
\begin{equation}
\label{eq.8.29}
R^{3} \le \Bigl(\int\limits_{\mathbb{R}^{n}}\|\nabla F(x)\|^{p}\,dx \Bigr)^{\frac{q}{p}}.
\end{equation}

\textit{Step 7.} We combine estimates \eqref{eq.8.22}, \eqref{eq.8.25}, \eqref{eq.8.28}, \eqref{eq.8.29} and take into
account that, for some $C > 0$ depending only on $p$ and $n$,
\begin{equation}
\notag
\frac{1}{C}\int\limits_{\mathbb{R}^{n}}\|\nabla F(x)\|^{p}\,dx \le \Bigl(\sum\limits_{|\gamma|=1}\|D^{\gamma}F|L_{p}(\mathbb{R}^{n})\|\Bigr)^{p} \le C \int\limits_{\mathbb{R}^{n}}\|\nabla F(x)\|^{p}\,dx.
\end{equation}
As a result, we get \eqref{eq.8.21}
and complete the proof.
\end{proof}

The following lemma is an easy consequence of the corresponding definitions.

\begin{Lm}
\label{Lm8.4}
Let $f \in L_{1}(\{\mathfrak{m}_{k}\})$, $t > 0$ and $c \geq 1$ be such that $\mathcal{U}^{b}_{c,t}(f) \neq \emptyset$.
Then there exist a family
$\mathfrak{D}_{c,t}(f) \subset \mathcal{DF}$ and a selection $\kappa_{c,t}(f)$ of $\mathcal{K}_{c}$ with the domain $\mathfrak{D}_{c,t}(f)$ such that:

{ \rm (1)} $\mathcal{U}^{b}_{c,t}(f) \subset \cup \{c\underline{Q}: \underline{Q} \in \mathfrak{D}_{c,t}(f)\}$;

{ \rm (2)} $\mathfrak{D}_{c,t}(f)$ is nonoverlapping;

{ \rm (3)} $\kappa_{c,t}(f)(\underline{Q}) \in \mathcal{P}(c)$ for any $\underline{Q} \in \mathfrak{D}_{c,t}(f)$;

{ \rm (4)} $l(\kappa_{c,t}(f)(\underline{Q})) \geq 4 l (\underline{Q})$;

{ \rm (5)} $\Phi_{f} (\underline{Q}, \kappa_{c,t}(f)(\underline{Q})) > t$ for any $\underline{Q} \in \mathfrak{D}_{c,t}(f)$.

\end{Lm}

\begin{proof}
Given a point $x \in \mathcal{U}^{b}_{c,t}(f)$, let $\mathcal{F}(x)$ be the family of all cubes
$\underline{Q} \in \mathcal{DF}$ for each of which there exists a cube $\overline{Q} \in \mathcal{DF}$ such that:

{ \rm (i)} $x \in c \underline{Q}$;

{ \rm (ii)} $l(\overline{Q}) \geq 4l(\underline{Q})$;

{ \rm(iii)} $\overline{Q} \in \mathcal{K}_{c}(\underline{Q})$;

{ \rm(iv)} $\Phi_{f}(\underline{Q},\overline{Q}) > t$.

By Proposition \ref{Propos.8.1}, the family $\mathcal{F}(x)$ is nonempty for each $x \in \mathcal{U}^{b}_{c,t}(f)$.
We set
\begin{equation}
\notag
\operatorname{M}(x):=\max\{l(\underline{Q}):\underline{Q} \in \mathcal{F}(x)\}, \quad x \in \mathcal{U}^{b}_{c,t}(f).
\end{equation}
For each $x \in \mathcal{U}^{b}_{c,t}(f)$ we fix an arbitrary
cube $\underline{Q}(x) \in \mathcal{F}(x)$ with $l(\underline{Q}(x))=\operatorname{M}(x)$.
Now we define
\begin{equation}
\notag
\mathfrak{D}_{c,t}(f):=\{\underline{Q}(x): x \in \mathcal{U}^{b}_{c,t}(f)\}.
\end{equation}
By condition (i) we conclude that property (1) of the lemma holds true.

Since the family $\mathfrak{D}_{c,t}(f)$ is composed of dyadic cubes there are only two possible cases: either different cubes from $\mathfrak{D}_{c,t}(f)$ have disjoint interiors or one of them contains another one.
Hence, in order to establish property (2) we assume on the contrary that there exist cubes $\underline{Q}_{1}, \underline{Q}_{2} \in \mathfrak{D}_{c,t}(f)$
such that $\underline{Q}_{1} \subset \underline{Q}_{2}$ and $l(\underline{Q}_{2}) > l(\underline{Q}_{1})$.
By the construction, there is a point $x \in \mathcal{U}_{c,t}^{b}(f)$ such that $\underline{Q}(x)=\underline{Q}_{1}$. Thus, the inclusion $\underline{Q}_{1} \subset \underline{Q}_{2}$
implies that $x \in c\underline{Q}_{2}$ and thus $\underline{Q}_{2} \in \mathcal{F}(x)$ because $\underline{Q}_{2} \in \mathfrak{D}_{c,t}(f)$. On the other hand, the
inequality $l(\underline{Q}_{2}) > l(\underline{Q}_{1})$
is in contradiction with the maximality of the side length of $\underline{Q}(x)$.

For any $\underline{Q} \in \mathfrak{D}_{c,t}(f)$ by $\kappa_{c,t}(f)(\underline{Q})$ we denote an arbitrary cube $\overline{Q} \in \mathcal{K}_{c}(\underline{Q})$
for which conditions (ii)--(iv) above hold. As a result, we built a family $\mathfrak{D}_{c,t}(f)$ and a selection $\kappa_{c,t}(f)$ of $\mathcal{K}_{c}$
with the domain $\mathfrak{D}_{c,t}(f)$ such that properties (4) and (5) hold.

Fix an arbitrary cube $\underline{Q} \in \mathfrak{D}_{c,t}(f)$. By the construction
$\kappa_{c,t}(f)(\underline{Q}) \in \mathcal{K}_{c}(\underline{Q})$ and $l(\kappa_{c}(\underline{Q})) \geq 4l(\underline{Q})$. This fact together with condition $(\textbf{C}3)$
of Definition \ref{cover.cube}
implies the existence of a cube $Q \in \mathcal{D}_{+} \setminus \mathcal{DF}$ such that $\underline{Q} \subset Q \subset c(\kappa_{c,t}(f)(\underline{Q}))$
and $l(Q)=\frac{1}{2}l(\kappa_{c}(\underline{Q}))$. Coupling this observation with \eqref{eqq.3.14} we see that property (3) of the lemma holds true.

The proof is complete.
\end{proof}

\begin{Lm}
\label{Lm8.6}
Let  $p \in (\max\{1,n-d\},n]$ and $q \in (1,n-d)$. Then, for each $c \geq 1$ there exists a constant $C > 0$ depending only
on parameters $p,q,n,d,\lambda,c$ such that
\begin{equation}
\label{eq7.12}
\sum\limits_{\nu \in \mathbb{Z}}2^{\nu q}\mathcal{L}^{n}(\mathcal{U}^{b}_{c,2^{\nu}}(f) \setminus \mathcal{U}^{b}_{c,2^{\nu+1}}(f)) \le  C \Bigl(\sum\limits_{|\gamma| = 1} \|D^{\gamma}F|L_{p}(\mathbb{R}^{n})\|\Bigr)^{q}
\end{equation}
for any $F \in W_{p}^{1}(\mathbb{R}^{n})$ with $f=\operatorname{Tr}|^{d}_{S}[F]$.
\end{Lm}

\begin{proof}
Fix a constant $c \geq 1$.
Given $t \in (0,\infty)$, we fix a family $\mathfrak{D}_{c,t}(f) \subset \mathcal{DF}$ and a selection
$\kappa_{c,t}(f)$ of $\mathcal{K}_{c}$ with the domain $\mathfrak{D}_{c,t}(f)$ such that properties (1)--(5) of Lemma \ref{Lm8.4} hold.

Given $\nu \in \mathbb{Z}$, we put
\begin{equation}
\label{eqq8.33}
U_{\nu}:=\mathcal{U}^{b}_{c,2^{\nu}}(f) \setminus \mathcal{U}^{b}_{c,2^{\nu+1}}(f), \quad \mathfrak{D}_{\nu}:=\{Q \in \mathfrak{D}_{c,2^{\nu}}(f): cQ \cap U_{\nu} \neq \emptyset\}.
\end{equation}
By property (1) in Lemma \ref{Lm8.4} we have $\mathcal{U}^{b}_{c,2^{\nu}}(f) =\bigcup\{cQ \cap \mathcal{U}^{b}_{c,2^{\nu}}(f):Q \in \mathfrak{D}_{c,2^{\nu}}(f)\}$. As a result, using
\eqref{eqq8.33} we obtain
\begin{equation}
\label{eq.8.31}
U_{\nu} = \bigcup\{cQ \cap U_{\nu}:Q \in \mathfrak{D}_{c,2^{\nu}}(f)\} = \bigcup\{cQ \cap U_{\nu}:Q \in \mathfrak{D}_{\nu}\} \subset \bigcup\{cQ:Q \in \mathfrak{D}_{\nu}\} .
\end{equation}

We claim that
\begin{equation}
\label{eq.8.32}
\mathfrak{D}_{\nu} \cap \mathfrak{D}_{\nu'} = \emptyset \quad \hbox{for} \quad \nu \neq \nu'.
\end{equation}
Indeed, assume that there are $\nu,\nu' \in \mathbb{Z}$ such that $\nu < \nu'$ and there is a cube $\underline{Q} \in \mathfrak{D}_{\nu} \cap \mathfrak{D}_{\nu'}$.
Since $\underline{Q} \in \mathfrak{D}_{\nu}$, by \eqref{eqq8.33} we can find a fix a point $\underline{x} \in U_{\nu} \cap c\underline{Q}$.
On the other hand, since $\underline{Q} \in \mathfrak{D}_{\nu'} \subset \mathfrak{D}_{c,2^{\nu'}}(f)$ and $\underline{x} \in c\underline{Q}$ we have (we recall that our selection $\kappa_{c,2^{\nu'}}$ satisfies condition (5) of Lemma \ref{Lm8.4} with $t=2^{\nu'}$)
\begin{equation}
\notag
f^{\natural}_{c}(\underline{x}) \geq \Phi_{f}(\underline{Q},\kappa_{c,2^{\nu'}}(\underline{Q}))  > 2^{\nu'} \geq 2^{\nu+1}.
\end{equation}
This immediately implies that $\underline{x} \in \mathcal{U}^{b}_{c,2^{\nu'}}(f)$ and, consequently, $\underline{x} \notin U_{\nu}$. This contradiction proves the claim.

Now, having at our disposal \eqref{eq.8.32} we put
\begin{equation}
\label{eqq8.36}
\mathfrak{D}:=\bigcup_{\nu \in \mathbb{Z}}\mathfrak{D}_{\nu}
\end{equation}
and obtain a well defined selection $\kappa_{c}$ of the set-valued mapping $\mathcal{K}_{c}$ with the domain $\mathfrak{D}$ by letting
\begin{equation}
\label{eqq.8.38}
\kappa_{c}(Q):=
\kappa_{c,2^{\nu}}(f)(Q) \quad \hbox{if} \quad Q \in \mathfrak{D}_{\nu}.
\end{equation}
By property (3) of Lemma \ref{Lm8.4} and \eqref{eqq.8.38} it follows that
\begin{equation}
\label{eqq8.37}
\kappa_{c}(Q) \in \mathcal{P}(c) \quad \hbox{for every} \quad Q \in \mathfrak{D}.
\end{equation}

Combining \eqref{eq.8.31} and \eqref{eq.8.32} we have
\begin{equation}
\label{eq.8.33}
\begin{split}
&\sum\limits_{\nu \in \mathbb{Z}}2^{\nu q}\mathcal{L}^{n}(U_{\nu})
\le c^{n}\sum\limits_{\nu \in \mathbb{Z}}2^{\nu q}\sum \{(l(\underline{Q}))^{n}: \underline{Q} \in \mathfrak{D}_{\nu}\}\\
&\le c^{n}
\sum\limits_{\nu \in \mathbb{Z}}\sum\limits_{\underline{Q} \in \mathfrak{D}_{\nu}}(l(\underline{Q}))^{n}(\Phi_{f}(\underline{Q},\kappa_{c}(\underline{Q})))^{q}
\le c^{n} \sum\limits_{\underline{Q} \in \mathfrak{D}}(l(\underline{Q}))^{n}(\Phi_{f}(\underline{Q},\kappa_{c}(\underline{Q})))^{q}.
\end{split}
\end{equation}
By \eqref{eqq8.33}, \eqref{eqq8.36}, \eqref{eqq8.37}, the family $\mathfrak{D}$ and the map $\kappa_{c}$ satisfy all the
assumptions of  Lemma \ref{KeyLemma2}. Hence,
we can continue estimate \eqref{eq.8.33} and get
\begin{equation}
\sum\limits_{\nu \in \mathbb{Z}}2^{\nu q}\mathcal{L}^{n}(U_{\nu})
\le C \Bigl(\sum\limits_{|\gamma| = 1} \|D^{\gamma}F|L_{p}(\mathbb{R}^{n})\|\Bigr)^{q}.
\end{equation}

The proof is complete.

\end{proof}

\textit{The main result} of this section reads as follows.

\begin{Th}
\label{Th.direct.trace1}
Let  $p \in (\max\{1,n-d\},n]$, $q \in (1,n-d)$ and $c \geq 1$. Then there exists a constant $C > 0$ depending only on parameters $p,q,n,\lambda,d,c$ such that
\begin{equation}
\label{eq.713}
\widetilde{\mathcal{N}}_{q,\{\mathfrak{m}_{k}\},\lambda,c}(f) \le C \|F|W_{p}^{1}(\mathbb{R}^{n})\|
\end{equation}
for any $F \in W_{p}^{1}(\mathbb{R}^{n})$ with $f=\operatorname{Tr}|^{d}_{S}[F]$.
\end{Th}

\begin{proof}
By Remark \ref{Rem.81} it is clear that
\begin{equation}
\label{eqq8.41}
\int\limits_{\mathbb{R}^{n}}(f^{\natural}_{c}(x))^{q}\,dx =
\int\limits_{\mathcal{U}^{b}_{c}(f)}(f^{\natural}_{c}(x))^{q}\,dx + \int\limits_{\mathcal{U}^{g}_{c}(f)}(f^{\natural}_{c}(x))^{q}\,dx.
\end{equation}
It is easy to see that
\begin{equation}
\label{eqq8.42}
\int\limits_{\mathcal{U}^{b}_{c}(f)}(f^{\natural}_{c}(x))^{q}\,dx \le
2^{q}\sum\limits_{\nu \in \mathbb{Z}}2^{q\nu}\mathcal{L}^{n}(\mathcal{U}^{b}_{c,2^{\nu}}(f) \setminus \mathcal{U}^{b}_{c,2^{\nu+1}}(f)).
\end{equation}
Combining \eqref{eqq8.41}, \eqref{eqq8.42} and Lemmas \ref{Lm7.2}, \ref{Lm8.6} we obtain
\begin{equation}
\label{eqq.748}
\int\limits_{\mathbb{R}^{n}}(f^{\natural}_{c}(x))^{q}\,dx
\le C \Bigl(\sum\limits_{|\gamma| = 1} \|D^{\gamma}F|L_{p}(\mathbb{R}^{n})\|\Bigr)^{q}.
\end{equation}
By \eqref{eq29'} we have $\mathfrak{m}_{0}(S) \le \operatorname{C}_{\{\mathfrak{m}_{k}\},1}$. Hence, by Proposition \ref{Proposition.2.10} and \eqref{eqq.simplest_imbedding} we get
\begin{equation}
\label{eqq.749}
\|f|L_{q}(\mathfrak{m}_{0})\| \le  C \|F|W_{p}^{1}(\mathbb{R}^{n})\|.
\end{equation}
By \eqref{eqq.748} and \eqref{eqq.749} we obtain \eqref{eq.713} and complete the proof.
\end{proof}

\section{The main results}

In this section we prove the main result of the present paper. Namely, we give a complete solution to Problem B.
To this aim we recall Definitions \ref{Def.d.trace}, \ref{def.ext.oper}, \ref{def.x.trace}.

First of all, we show that the operator $\operatorname{Ext}_{S,\{\mathfrak{m}_{k}\},\lambda}$ defined in \eqref{eq.exst.op}
is an extension operator. More precisely,
the following result holds.

\begin{Th}
\label{Th.ext.tr}
Let $q \in (1,n]$, $d^{\ast} \in (n-q,n]$, $d \in (0,n-q)$. Let $S \subset Q_{0,0}$
be a compact set with $\lambda^{\ast}:=\mathcal{H}^{d^{\ast}}_{\infty}(S) > 0$ and $\{\mathfrak{m}_{k}\} \in \mathfrak{M}^{d}(S)$. Let $c \geq 7$ and $\lambda \in (0,\lambda^{\ast})$ be some fixed constants.
Then
\begin{equation}
\label{eqq.9.1}
\operatorname{Tr}|^{d^{\ast}}_{S} \circ \operatorname{Ext}_{S,\{\mathfrak{m}_{k}\},\lambda} = \operatorname{Id}
\quad \hbox{on} \quad  \operatorname{X}^{d^{\ast}}_{q,d,\{\mathfrak{m}_{k}\}}(S).
\end{equation}
\end{Th}

\begin{proof}
Let $f \in \operatorname{X}^{d^{\ast}}_{q,d,\{\mathfrak{m}_{k}\}}(S)$ and $F=\operatorname{Ext}_{S,\{\mathfrak{m}_{k}\},\lambda}(f)$.
Let $\{\widetilde{\mathfrak{m}}_{k}\}$ be an arbitrary sequence of measures from the class $\mathfrak{M}^{d^{\ast}}(S)$.
We recall \eqref{eq3.32} and, for any given point $x \in \underline{S}(d^{\ast},\lambda)$, we fix a strictly increasing sequence $\{k_{s}(x)\}=\{k_{s}(x)\}_{s=1}^{\infty} \subset \mathbb{N}$ and an arbitrary sequence of closed cubes $\{Q^{s}(x)\}$ such that $Q^{s}(x) \in \mathcal{DF}_{S,k_{s}(x)}(d^{\ast},\lambda)$
and $x \in Q^{s}(x)$ for all $s \in \mathbb{N}$.

\textit{Step 1.}
By \eqref{eq.exst.op} we have $F(x)=f(x)$ for all $x \in S$. Hence, for each number $s \in \mathbb{N}$, using the triangle inequality we obtain
\begin{equation}
\label{eqq.9.2}
\begin{split}
&\fint\limits_{Q^{s}(x)}|F(x)-F(y)|\,dy \le \fint\limits_{Q^{s}(x)}\fint\limits_{Q^{s}(x)}|f(z)-F(y)|\,d\widetilde{\mathfrak{m}}_{k_{s}(x)}(z)\,dy\\
&+\fint\limits_{Q^{s}(x)}|f(x)-f(z)|\,d\widetilde{\mathfrak{m}}_{k_{s}(x)}(z)=:R^{s}_{1}(x)+R^{s}_{2}(x), \quad x \in \underline{S}(d^{\ast},\lambda).
\end{split}
\end{equation}

\textit{Step 2.} Let $\{f_{k}\}=\{f_{k}\}(\{\mathfrak{m}_{k}\})$ be the special approximating sequence constructed in Section 5 for the $\mathfrak{m}_{0}$-equivalence class $[f]_{\mathfrak{m}_{0}}$ of $f$
(see Definition \ref{Def.spec.app.seq} and Remark \ref{Rem.spec_app_seq_representatives}).
Combining Theorem \ref{Th.Reverse_trace} with the Mazur lemma we deduce the existence of a strictly increasing sequence $\{N_{l}\}$ of integer nonnegative numbers
and a sequence of convex combinations
\begin{equation}
\label{eqq.9.3}
\widetilde{f}_{N_{l}}:=\sum\limits_{i=1}^{M_{l}}\lambda_{l}^{i}f_{N_{l}+i} \quad \text{with} \quad \lambda^{i}_{l}  \geq 0 \quad \text{and} \quad \sum_{i=1}^{M_{l}}\lambda^{i}_{l}=1
\end{equation}
such that the $\mathcal{L}^{n}$-equivalence class $[F]$ of $F$ (we will identify $F$ and $[F]$ below) belongs to the space $W_{q}^{1}(\mathbb{R}^{n})$ and, furthermore,
\begin{equation}
\label{eqq.9.4'''}
\|F-\widetilde{f}_{N_{l}}|W_{q}^{1}(\mathbb{R}^{n})\| \to 0, \quad l \to \infty.
\end{equation}
Combining this fact with Proposition \ref{Proposition.2.10} and taking into account that $q > n-d^{\ast}$, we deduce that $\{\widetilde{f}_{N_{l}}\}$ is a Cauchy sequence in $L_{q}(\widetilde{\mathfrak{m}}_{j})$
for each $j \in \mathbb{N}_{0}$. By the completeness of $L_{q}(\widetilde{\mathfrak{m}}_{j})$ there is $g \in L_{q}(\widetilde{\mathfrak{m}}_{j})$ such that
the sequence $\{\widetilde{f}_{N_{l}}\}$ converges to $g$ in $L_{q}(\widetilde{\mathfrak{m}}_{j})$-sense for every $j \in \mathbb{N}_{0}$. Note that, given $j \in \mathbb{N}_{0}$, the measure $\widetilde{\mathfrak{m}}_{j}$
is absolutely continuous with respect to $\mathfrak{m}_{j}$. As a result, combining this observation with Lemma \ref{Lm.function2} we obtain
\begin{equation}
\label{eqq.9.5}
\lim\limits_{l \to \infty}\|F-\widetilde{f}_{N_{l}}|L_{q}(\widetilde{\mathfrak{m}}_{j})\|= 0 \quad \text{for every} \quad j \in \mathbb{N}_{0}.
\end{equation}
Finally, by Corollary \ref{Ca.convergence} we have
\begin{equation}
\label{eqq.9.6}
\lim\limits_{l \to \infty}\|F-\widetilde{f}_{N_{l}}|L_{q}(\mathbb{R}^{n})\|= 0.
\end{equation}

\textit{Step 3.}
Since, for each function $\widetilde{f}_{N_{l}} \in C_{0}^{\infty}(\mathbb{R}^{n})$, the $d^{\ast}$-trace of $\widetilde{f}_{N_{l}}$ to $S$ coincides (up to a set of $\mathcal{H}^{d^{\ast}}$-measure
zero) with its
pointwise restriction. Given $x \in \underline{S}(d^{\ast},\lambda)$, an application of Theorem \ref{Th4.3} gives, for each $l \in \mathbb{N}$,
\begin{equation}
\label{eqq.9.7'}
\fint\limits_{Q^{s}(x)}\fint\limits_{Q^{s}(x)}|\widetilde{f}_{N_{l}}(z)-\widetilde{f}_{N_{l}}(y)|\,d\widetilde{\mathfrak{m}}_{k_{s}(x)}(z)\,dy \le C 2^{-k_{s}(x)} \Bigl(\sum\limits_{|\gamma| = 1}\fint\limits_{Q^{s}(x)}|D^{\gamma}\widetilde{f}_{N_{l}}(t)|^{q}\,dt\Bigr)^{\frac{1}{q}}.
\end{equation}
Using \eqref{eqq.9.4'''}--\eqref{eqq.9.6} we can pass to the limit in \eqref{eqq.9.7'}. Hence, for each $x \in \underline{S}(d^{\ast},\lambda)$,
\begin{equation}
\label{eqq.9.7}
\fint\limits_{Q^{s}(x)}\fint\limits_{Q^{s}(x)}|f(z)-F(y)|\,d\widetilde{\mathfrak{m}}_{k_{s}(x)}(z)\,dy
\le C 2^{-k_{s}(x)} \Bigl(\sum\limits_{|\gamma| = 1}\fint\limits_{Q^{s}(x)}|D^{\gamma}F(t)|^{q}\,dt\Bigr)^{\frac{1}{q}}.
\end{equation}
As a result, by Proposition \ref{Prop.2.5}, we obtain $\lim_{s \to \infty}R^{s}_{1}(x)=0$ for $\mathcal{H}^{n-q}$-a.e. $x \in \underline{S}(d^{\ast},\lambda)$.
Since $d^{\ast} > n-q$, this gives
\begin{equation}
\label{eqq.9.8}
\lim_{s \to \infty}R^{s}_{1}(x)=0 \quad \hbox{for}  \quad  \mathcal{H}^{d^{\ast}}-a.e. \quad x \in \underline{S}(d^{\ast},\lambda).
\end{equation}

\textit{Step 4.} Since $f \in \operatorname{X}^{d^{\ast}}_{q,d,\{\mathfrak{m}_{k}\}}(S)$, by Definition \ref{Deff.5.4}
\begin{equation}
\label{eqq.9.9}
\lim_{s \to \infty}R^{s}_{2}(x)=0  \quad \hbox{for each} \quad x \in S_{f}(d^{\ast}).
\end{equation}

\textit{Step 5.} 
Applying Proposition \ref{Th2.4} we find a $(1,q)$-good representative $\overline{F}$
of $F$. Hence, by Definition \ref{Def.good.rep} and Proposition \ref{Hausdorf_capacity}
\begin{equation}
\notag
\lim\limits_{l \to 0}\fint\limits_{Q_{l}(x)}\overline{F}(y)\,dy=
\lim\limits_{l \to 0}\fint\limits_{Q_{l}(x)}F(y)\,dy=\overline{F}(x) \quad \hbox{for} \quad \mathcal{H}^{d^{\ast}}-a.e. \quad x \in S.
\end{equation}
Combining this fact with Proposition \ref{Prop41''} and \eqref{eqq.9.2}, \eqref{eqq.9.8}, \eqref{eqq.9.9} we get $\overline{F}(x)=f(x)$
for $\mathcal{H}^{d^{\ast}}$-a.e.  $x \in S$.
As a result, taking into account Definition \ref{Def.d.trace} we obtain \eqref{eqq.9.1} and complete the proof.
\end{proof}

The following lemma is very similar in spirit to Lemma 4.3 from \cite{TV}. We present
the detailed proof for the completeness. Recall Definition \ref{Deff.5.4}.

\begin{Lm}
\label{Lebesguepoint1}
Let $p \in (1,n]$, $d^{\ast} \in (n-p,n]$ and let $S \subset Q_{0,0}$ be a compact set with $\mathcal{H}^{d^{\ast}}_{\infty}(S) > 0$.
Then, for each $[f] \in W_{p}^{1}(\mathbb{R}^{n})|^{d^{\ast}}_{S}$, there exists a Borel representative $f$ of $[f]$ such that $\mathcal{H}^{d^{\ast}}(S \setminus S_{f}(d'))=0$ for each $d' \in (n-p,d^{\ast}]$.
\end{Lm}

\begin{proof}
We set $\lambda^{\ast}:=\mathcal{H}^{d^{\ast}}_{\infty}(S)$ and fix arbitrary $d' \in (n-p,d^{\ast}]$, $\lambda \in (0,\lambda^{\ast})$, $c\geq1$.
We also fix
\begin{equation}
[f] \in W_{p}^{1}(\mathbb{R}^{n})|^{d^{\ast}}_{S} \quad \text{and} \quad F \in W_{p}^{1}(\mathbb{R}^{n}) \text{ with } F|^{d^{\ast}}_{S}=[f].
\end{equation}
Finally, we fix $\{\widetilde{\mathfrak{m}}_{k}\} \in \mathfrak{M}^{d'}(S)$.
Using Proposition \ref{Th2.4}, we find and fix a $(1,p)$-good representative $\overline{F}$ of $F$.
Furthermore, we put $f=\overline{F}|_{S}$.

Now let $S' \subset S$ be
the intersection of the following three sets: the set $\underline{S}(d',\lambda)$, the set of all Lebesgue points of the function $\overline{F}$, and
the set
$
\Bigl\{x \in S:\lim_{r \to 0}r^{p-n}\fint_{Q_{r}(x)}\sum_{|\gamma|=1}|D^{\gamma}F|(x)\,dx = 0\Bigr\}.
$
Using Propositions \ref{Prop.2.5}, \ref{Hausdorf_capacity}, \ref{Prop41''}
and Definition \ref{Def.good.rep} we obtain
\begin{equation}
\label{eqq.hausdorff_except_set}
\mathcal{H}^{d^{\ast}}(S \setminus S') = \mathcal{H}^{d'}(S \setminus S') = 0.
\end{equation}
Given a point $x \in S'$ and a cube $Q \in T_{d',\lambda,c}(x) \cap \mathcal{D}_{k}$, by the triangle inequality we have
\begin{equation}
\label{eqq.9.14}
\begin{split}
&\fint\limits_{Q \cap S}|f(x)-f(z)|\,d\widetilde{\mathfrak{m}}_{k}(z) \le \Bigl|\overline{F}(x)-\fint\limits_{Q}F(y)\,dy\Bigr|\\
&+\fint\limits_{Q \cap S}\Bigl|f(z)-\fint\limits_{Q}F(y)\,dy\Bigr|\,d\widetilde{\mathfrak{m}}_{k}(z)=:J^{1}(Q)+J^{2}(Q).
\end{split}
\end{equation}

If $x \in cQ$ for some $Q \in \mathcal{D}_{k}$ with $k \in \mathbb{N}_{0}$ then $Q \subset Q_{\frac{2c}{2^{k}}}(x)$. Hence, we have
\begin{equation}
\notag
\max\limits_{Q \in T_{d',\lambda,c}(x)\cap \mathcal{D}_{k}} J^{1}(Q) \le C \fint\limits_{Q_{\frac{2c}{2^{k}}}(x)}|\overline{F}(x)-\overline{F}(y)|\,dy.
\end{equation}
By the construction of $S'$, we have
\begin{equation}
\label{eqq.9.15}
\lim\limits_{k \to \infty}\max\limits_{Q \in T_{d',\lambda,c}(x)\cap \mathcal{D}_{k}} J^{1}(Q) = 0 \quad \hbox{for every} \quad x \in S'.
\end{equation}

Applying Theorem \ref{Th4.3} with $\sigma=p$ and $d=d'$ for each $x \in S'$, we get
\begin{equation}
\notag
\begin{split}
&\max\limits_{Q \in T_{d',\lambda,c}(x)\cap \mathcal{D}_{k}}(J^{2}(Q))^{p} \le C\max\limits_{Q \in T_{d',\lambda,c}(x)\cap \mathcal{D}_{k}}
(l(Q))^{p-n}\sum\limits_{|\gamma|=1}\int\limits_{Q}|D^{\gamma} F(z)|^{p}\,dz\\
&\le C \sum\limits_{|\gamma|=1} \Bigl(\frac{2^{k}}{2c}\Bigr)^{n-p} \int\limits_{Q_{\frac{2c}{2^{k}}}(x)}|D^{\gamma} F(z)|^{p}\,dz.
\end{split}
\end{equation}
Hence, by Proposition \ref{Prop.2.5} we deduce
\begin{equation}
\label{eqq.9.16}
\lim\limits_{k \to \infty}\max\limits_{Q \in T_{d',\lambda,c}(x)\cap \mathcal{D}_{k}} J^{2}(Q) = 0 \quad \hbox{for every} \quad x \in S'.
\end{equation}

As a result,
combining \eqref{eqq.hausdorff_except_set} with
\eqref{eqq.9.14}--\eqref{eqq.9.16}, by Definition \ref{Def.d.trace}, we obtain $\mathcal{H}^{d^{\ast}}(S \setminus S_{f}(d')) = 0$ and complete the proof.

\end{proof}

Now present \textit{the main result of this paper}. This gives a solution to Problem B.

\begin{Th}
\label{Th.main}
Let $p \in (1,n]$, $d^{\ast} \in (n-p,n]$ and $\varepsilon^{\ast}:=\min\{p-(n-d^{\ast}),p-1\}$. Let $S \subset Q_{0,0}$ be a compact set with $\lambda^{\ast}:=\mathcal{H}^{d^{\ast}}_{\infty}(S) > 0$.  Let $\lambda \in (0,\lambda^{\ast})$ and $c \geq 7$ be some fixed constants.
Then, for each $\varepsilon \in (0,\varepsilon^{\ast})$,
\begin{equation}
\label{eqq.9.17}
W_{p}^{1}(\mathbb{R}^{n})|^{d^{\ast}}_{S} \subset \operatorname{X}^{d^{\ast}}_{p-\varepsilon,d,\{\mathfrak{m}_{k}\}}(S) \subset W_{p-\varepsilon}^{1}(\mathbb{R}^{n})|^{d^{\ast}}_{S}
\end{equation}
for any $d \in (n-p,n-p+\varepsilon)$ and $\{\mathfrak{m}_{k}\} \in \mathfrak{M}^{d}(S)$.

Furthermore, for every $\varepsilon \in (0,\varepsilon^{\ast})$, $d \in (n-p,n-p+\varepsilon)$ and $\{\mathfrak{m}_{k}\} \in \mathfrak{M}^{d}(S)$, there exists a constant $C> 0$ depending only on $p,\varepsilon,n,d,\lambda,c$ and $\operatorname{C}_{\mathfrak{m}_{k},i}$, $i=1,2,3$ such that
\begin{equation}
\label{eqq.9.18}
\frac{1}{C}\|f|W_{p-\varepsilon}^{1}(\mathbb{R}^{n})|^{d^{\ast}}_{S}\| \le
\|f|\operatorname{X}^{d^{\ast}}_{p-\varepsilon,d,\{\mathfrak{m}_{k}\}}\| \le C \|f|W_{p}^{1}(\mathbb{R}^{n})|^{d^{\ast}}_{S}\| \quad \text{for all} \quad f \in W_{p}^{1}(\mathbb{R}^{n})|^{d^{\ast}}_{S}.
\end{equation}

The operator $\operatorname{Ext}_{S,\{\mathfrak{m}_{k}\},\lambda}$ is a bounded linear mapping
from $\operatorname{X}^{d^{\ast}}_{p-\varepsilon,d,\{\mathfrak{m}_{k}\}}(S)$ to $W_{p-\varepsilon}^{1}(\mathbb{R}^{n})$ and
$\operatorname{Tr}|_{S}^{d^{\ast}} \circ \operatorname{Ext}_{S,\{\mathfrak{m}_{k}\},\lambda} = \operatorname{Id}$ on $\operatorname{X}^{d^{\ast}}_{p-\varepsilon,d,\{\mathfrak{m}_{k}\}}(S)$.
In particular, $\operatorname{Ext}_{S,\{\mathfrak{m}_{k}\},\lambda} \in \mathfrak{E}(S,d^{\ast},p,\varepsilon)$.
\end{Th}

\begin{proof}
We fix arbitrary $\varepsilon \in (0,\varepsilon^{\ast})$, $d \in (n-p,n-p+\varepsilon)$ and $\{\mathfrak{m}_{k}\} \in \mathfrak{M}^{d}(S)$.
Given $f \in W_{p}^{1}(\mathbb{R}^{n})|^{d^{\ast}}_{S}$, we apply Lemma \ref{Lebesguepoint1} and then
use Theorem \ref{Th.direct.trace1} with $q=p-\varepsilon$. We conclude
that $f \in \operatorname{X}^{d^{\ast}}_{p-\varepsilon,d,\{\mathfrak{m}_{k}\}}(S)$ and furthermore, there is a constant $C > 0$ depending only on
the parameters $n,p,\varepsilon,d,\lambda,c$ and the constants $\operatorname{C}_{\{\mathfrak{m}_{k}\},i}$, $i=1,2,3$, such that
\begin{equation}
\notag
\widetilde{\mathcal{N}}_{p-\varepsilon,\{\mathfrak{m}_{k}\},\lambda,c}(f) \le C \|F|W_{p}^{1}(\mathbb{R}^{n})\| \quad \text{for all} \quad F \in W_{p}^{1}(\mathbb{R}^{n}) \text{ with } f=F|^{d^{\ast}}_{S}.
\end{equation}
This observation in combination with \eqref{eq214} proves the first inclusion in \eqref{eqq.9.17} and the second inequality in \eqref{eqq.9.18}.

Conversely, let $f \in \operatorname{X}^{d^{\ast}}_{p-\varepsilon,d,\{\mathfrak{m}_{k}\}}(S)$ and $F=\operatorname{Ext}_{S,\{\mathfrak{m}_{k}\},\lambda}(f)$.
By Theorem \ref{Th.Reverse_trace} we conclude that the $\mathcal{L}^{n}$-equivalence class $[F]$ of $F$ belongs to $W_{p-\varepsilon}^{1}(\mathbb{R}^{n})$,
and furthermore, there exists a constant $C > 0$ depending only on the
parameters $n,p,\varepsilon,d,\lambda,c$ and $\operatorname{C}_{\{\mathfrak{m}_{k}\},i}$, $i=1,2,3$ such that
\begin{equation}
\label{eqq.9.24}
\|[F]|W_{p-\varepsilon}^{1}(\mathbb{R}^{n})\|
\le C \|f|\operatorname{X}^{d^{\ast}}_{p-\varepsilon,d,\{\mathfrak{m}_{k}\}}\|.
\end{equation}
We apply Theorem \ref{Th.ext.tr} with
$q=p-\varepsilon$ and then take into account Definition \ref{Def.d.trace}. As a result, we have $f \in W_{p-\varepsilon}^{1}(\mathbb{R}^{n})|^{d^{\ast}}_{S}$.
This proves the second inclusion in \eqref{eqq.9.17}. Finally, the first inequality in \eqref{eqq.9.18} follows
from \eqref{eq214} and \eqref{eqq.9.24}.

The theorem is proved.
\end{proof}

Finally, having at our disposal the above results we can establish the following result.

\begin{Lm}
\label{Lm.completeness}
Let $d^{\ast} \in (0,n]$ and let $S \subset Q_{0,0}$ be a compact set with $\lambda^{\ast}:=\mathcal{H}^{d^{\ast}}_{\infty}(S) > 0$.
Let $\lambda \in (0,1]$ and $c \geq 1$ be some fixed constants.
If $p \in (1,\infty)$, $d^{\ast} > n-p$, $d \in (n-p,d^{\ast}]$ and $\{\mathfrak{m}_{k}\} \in \mathfrak{M}^{d}(S)$, then
$\operatorname{X}^{d^{\ast}}_{p,d,\{\mathfrak{m}_{k}\}}(S)$ is a Banach space.
\end{Lm}

\begin{proof}
In view of Remark \ref{Remm.4.3}, it is sufficient to show that the space $\operatorname{X}^{d^{\ast}}_{p,d,\{\mathfrak{m}_{k}\}}(S)$
is complete. We fix an arbitrary Cauchy sequence $\{f_{j}\} \subset \operatorname{X}^{d^{\ast}}_{p,d,\{\mathfrak{m}_{k}\}}(S)$. This implies
that the sequence $\{[f_{j}]_{\mathfrak{m}_{0}}\}$ of $\mathfrak{m}_{0}$-equivalence classes of $f_{j}$, $j \in \mathbb{N}$ is a Cauchy sequence in $L_{p}(\mathfrak{m}_{0})$.
By the arguments from the second part of the proof of Theorem \ref{Th.main} it follows that the sequence $\{[f_{j}]_{d^{\ast}}\}$ of $\mathcal{H}^{d^{\ast}}$-equivalence classes
of $f_{j}$, $j \in \mathbb{N}_{0}$ is a Cauchy sequence
in the space $W_{p}^{1}(\mathbb{R}^{n})|_{S}^{d^{\ast}}$. Using the completeness
of $L_{p}(\mathfrak{m}_{0})$ and $W_{p}^{1}(\mathbb{R}^{n})|_{S}^{d^{\ast}}$ (recall that the later follows from Proposition \ref{Prop.d_trace_space_is_complete})
in combination with Theorem \ref{Th.main} we deduce existence $g_{1} \in L_{p}(\mathfrak{m}_{0})$ and $g_{2} \in W_{p}^{1}(\mathbb{R}^{n})|_{S}^{d^{\ast}}$ such that
$[f_{j}]_{\mathfrak{m}_{0}} \to g_{1}$, $j \to \infty$ in $L_{p}(\mathfrak{m}_{0})$-sense and $[f_{j}]_{d^{\ast}} \to g_{2}$, $j \to \infty$  in $W^{1}_{p}(\mathbb{R}^{n})|_{S}^{d^{\ast}}$-sense.
The crucial observation is that using arguments from the proof of Proposition \ref{Prop.d_trace_space_is_complete} one can deduce that
there exists a subsequence $\{f_{j_{l}}\}$ of $\{f_{j}\}$ such that $f_{j_{l}}(x) \to g_{1}(x)$, $l \to \infty$ for $\mathfrak{m}_{0}$-a.e. $x \in S$
and $f_{j_{l}}(x) \to g_{2}(x)$, $l \to \infty$ for $\mathcal{H}^{d^{\ast}}$-a.e. $x \in S$. As a result, we have $g_{1}=[f]_{\mathfrak{m}_{0}}$ and
$g_{2}=[f]_{d^{\ast}}$ for some $f \in \mathfrak{B}(S)$. Furthermore, by Lemma \ref{Lebesguepoint1} we have $\mathcal{H}^{d^{\ast}}(S \setminus S_{f}(d'))=0$
for all $d' \in [d,d^{\ast}]$.

Taking into account Remark \ref{Rem.Frostman}, given $x \in \mathbb{R}^{n}$ and $\underline{Q},\overline{Q} \in \mathcal{D}_{+}$ satisfying
conditions (\textbf{f}1)--(\textbf{f}3) of Definition \ref{Cal.max.function}, we can pass to the limit and deduce that
$$
\lim_{i \to \infty} \Phi_{f_{i}-f_{j},\{\mathfrak{m}_{k}\}}(\underline{Q},\overline{Q})= \Phi_{f-f_{j},\{\mathfrak{m}_{k}\}}(\underline{Q},\overline{Q}).
$$
Taking the corresponding supremum we find that, for each $j \in \mathbb{N}_{0}$,
\begin{equation}
\notag
(f-f_{j})^{\natural}_{\{\mathfrak{m}_{k}\},\lambda,c}(x)
\le \varliminf\limits_{i \to \infty}(f_{i}-f_{j})^{\natural}_{\{\mathfrak{m}_{k}\},\lambda,c}(x) \quad \hbox{for all} \quad x \in \mathbb{R}^{n}.
\end{equation}
Using Fatou's lemma and the fact that $\{f_{j}\}$ is a Cauchy sequence in
$\operatorname{X}^{d^{\ast}}_{p,d,\{\mathfrak{m}_{k}\}}(S)$, we obtain
\begin{equation}
\label{eqq.9.21}
\|(f-f_{j})^{\natural}_{\{\mathfrak{m}_{k}\},\lambda,c}|L_{p}(\mathbb{R}^{n})\| \le
\varliminf\limits_{i \to \infty}\|(f_{i}-f_{j})^{\natural}_{\{\mathfrak{m}_{k}\},\lambda,c}|L_{p}(\mathbb{R}^{n})\| \to 0, \quad j \to \infty.
\end{equation}
Combining the above observations with Remark \ref{Remm.4.3} we find that $f \in \operatorname{X}^{d^{\ast}}_{p,d,\{\mathfrak{m}_{k}\}}(S)$
and $f_{j} \to f$, $j \to \infty$ in the space $\operatorname{X}^{d^{\ast}}_{p,d,\{\mathfrak{m}_{k}\}}(S)$. This completes the proof.
\end{proof}

\end{document}